\documentclass[11pt,final]{amsart}

\usepackage{lmodern}

\usepackage[active]{srcltx}
\usepackage[utf8]{inputenc}
\usepackage[english]{babel}
\usepackage{graphicx}
\usepackage[T1]{fontenc}
\usepackage{amsmath}
\usepackage{amssymb} 
\usepackage{amsthm}
\usepackage{mathabx}
\usepackage{bbm} 
\usepackage{bm} 
\usepackage{eqnarray}
\usepackage{tikz}
\usepackage[numbers]{natbib}
\usepackage{array}
\usepackage[left=2.05cm,right=2.05cm,top=2.95cm,bottom=2.8cm]{geometry}
\usepackage[foot]{amsaddr} 

\usepackage{esint}

\usepackage{enumitem}

\usepackage{makecell}

\usepackage{setspace}

\usepackage[pdfborder={0 0 0}]{hyperref}

\numberwithin{equation}{section} 

\let\oldtocsection=\tocsection
 
\let\oldtocsubsection=\tocsubsection
 
\renewcommand{\tocsection}[2]{\hspace{0em}\oldtocsection{#1}{#2}}
\renewcommand{\tocsubsection}[2]{\hspace{3em}\oldtocsubsection{#1}{#2}}

\newcommand{\R}{\mathbb{R}}
\newcommand{\Rc}{\mathcal{R}}

\newcommand{\N}{\mathbb{N}}
\newcommand{\Prob}{\mathcal{P}}
\newcommand{\M}{\mathcal{M}} 
\newcommand{\J}{\mathcal{J}} 
\newcommand{\A}{\mathcal{A}} 
\newcommand{\D}{\mathcal{D}} 
 
\newcommand{\T}{\mathcal{T}}
\newcommand{\E}{\mathcal{E}}

\newcommand{\V}{\mathcal{V}}

\newcommand{\X}{\mathcal{X}}
\newcommand{\Y}{\mathcal{Y}}

\newcommand{\CE}{\mathcal{CE}}

\newcommand{\vbf}{\mathbf{v}} 
\newcommand{\wbf}{\mathbf{w}} 
\newcommand{\mbf}{\mathbf{m}} 
\newcommand{\Mbf}{\mathbf{M}} 
\newcommand{\nbf}{\mathbf{n}} 
\newcommand{\Bbf}{\mathbf{B}} 
\newcommand{\bbf}{\mathbf{b}} 

\newcommand{\Bc}{\mathcal{B}}

\newcommand{\ddr}{\mathrm{d}} 
\newcommand{\dr}{\partial}

\newcommand{\Div}{\mathrm{Div}} 

\newcommand{\1}{\mathbbm{1}} 

\newcommand{\dst}[1]{\displaystyle{#1}}

\newtheorem{theo}{Theorem}[section]
\newtheorem*{theo*}{Theorem}
\newtheorem{prop}[theo]{Proposition}

\newtheorem{lm}[theo]{Lemma}
\newtheorem{defi}[theo]{Definition}

\newtheorem*{asmp*}{Assumption}

\theoremstyle{remark}
\newtheorem{rmk}[theo]{Remark}

\makeatletter
\DeclareFontFamily{OMX}{MnSymbolE}{}
\DeclareSymbolFont{MnLargeSymbols}{OMX}{MnSymbolE}{m}{n}
\SetSymbolFont{MnLargeSymbols}{bold}{OMX}{MnSymbolE}{b}{n}
\DeclareFontShape{OMX}{MnSymbolE}{m}{n}{
    <-6>  MnSymbolE5
   <6-7>  MnSymbolE6
   <7-8>  MnSymbolE7
   <8-9>  MnSymbolE8
   <9-10> MnSymbolE9
  <10-12> MnSymbolE10
  <12->   MnSymbolE12
}{}
\DeclareFontShape{OMX}{MnSymbolE}{b}{n}{
    <-6>  MnSymbolE-Bold5
   <6-7>  MnSymbolE-Bold6
   <7-8>  MnSymbolE-Bold7
   <8-9>  MnSymbolE-Bold8
   <9-10> MnSymbolE-Bold9
  <10-12> MnSymbolE-Bold10
  <12->   MnSymbolE-Bold12
}{}

\let\llangle\@undefined
\let\rrangle\@undefined
\DeclareMathDelimiter{\llangle}{\mathopen}%
                     {MnLargeSymbols}{'164}{MnLargeSymbols}{'164}
\DeclareMathDelimiter{\rrangle}{\mathclose}%
                     {MnLargeSymbols}{'171}{MnLargeSymbols}{'171}
\makeatother

\newcommand{\review}[1]{#1}


\begin{document}

\title{Unconditional convergence for discretizations of dynamical optimal transport} 

\date{\today}
\author{Hugo Lavenant}
\address{Department of Mathematics, University of British Columbia, Vancouver BC Canada V6T
1Z2}
\email{lavenant@math.ubc.ca}

\begin{abstract}
The dynamical formulation of optimal transport, also known as Benamou-Brenier formulation or Computational Fluid Dynamic\review{s} formulation, amounts to write the optimal transport problem as the optimization of a convex functional under a PDE constraint, and can handle \emph{a priori} a vast class of cost functions and geometries. Several \review{discretizations} of this problem have been proposed, leading to computations on flat spaces as well as Riemannian manifolds, with extensions to mean field games and gradient flows in the Wasserstein space. 

In this article, we provide a framework which guarantees convergence under mesh refinement of the solutions of the space-time discretized problems to the one of the infinite-dimensional one for quadratic optimal transport. The convergence holds without condition on the ratio between spatial and temporal step sizes, and can handle arbitrary positive measures as input, while the underlying space can be a Riemannian manifold. Both the finite volume discretization proposed by Gladbach, Kopfer and Maas, as well as the discretization over triangulations of surfaces studied by the present author in collaboration with Claici, Chien and Solomon fit in this framework.     
\end{abstract}

\maketitle

\tableofcontents

\section{Introduction}

In this article, we study a problem of calculus of variations of the form 
\begin{equation}
\label{equation_problem_unformal_v}
\min_{\rho, \vbf} \left\{ \iint \frac{1}{2} \rho |\vbf|^2 \, \ddr t \, \ddr x  \ \text{ such that } \ \dr_t \rho + \nabla \cdot ( \rho \vbf ) = 0 \right\},
\end{equation}
where the unknowns are $\rho$ a space-time dependent density and $\vbf$ a space-time dependent velocity field. These two unknowns are linked by the continuity equation $\dr_t \rho + \nabla \cdot ( \rho \vbf ) = 0$, supplemented with no-flux boundary conditions, which says that the mass with density $\rho$ is advected by the velocity field $\vbf$. Provided the velocity field is smooth enough (which is not always the case for the optimizer), $\rho$ is uniquely determined by its temporal boundary values and $\vbf$, hence $\vbf$ could be \review{thought} as a control. The functional which is optimized is the \emph{action}, that is the space-time integral of the density of kinetic energy. The initial and final temporal values of $\rho$ can be either fixed, or a penalization depending on them can be added in the objective functional.  

Instances of this problem appear in at least two situations. If the initial and final values $\rho_0$ and $\rho_1$ of $\rho$ are fixed, the problem \eqref{equation_problem_unformal_v} amounts to compute a geodesic in the Wasserstein space (see \cite{Benamou2000} or \cite[Chapter 5]{SantambrogioOTAM}), also known as the displacement interpolation (or McCann's interpolation \cite{Mccann1997}) between the initial and the final value of $\rho$. In short, the optimal $\rho$ is for some geometry the shortest curve joining its endpoints $\rho_0$ and $\rho_1$, while the value of the problem is the squared Wasserstein distance between $\rho_0$ and $\rho_1$. If the final value is penalized rather than fixed, \eqref{equation_problem_unformal_v} reads as one step of the minimizing movement scheme (also known as the JKO scheme after \cite{Jordan1998}) for gradient flows in the Wasserstein space (see for instance \cite{AGS} or \cite[Chapter 8]{SantambrogioOTAM}). On the other hand, one could also add a running cost depending on the density to the action, and then the problem would read as the variational formulation of an instance of Mean Field Games \cite{Benamou2016}: roughly, $\rho$ describes the evolution in time of a density of a crowd of agents trying to reach a final destination with minimum effort while avoiding other agents. 

In this article, we will mainly focus on the case where the initial and final values of $\rho$ are fixed (that is the computation of a geodesic in the Wasserstein space) because this case concentrates most of the difficulties. In the last section, we will also deal with a final penalization of the density which does not contain additional technicalities, and we will comment on the issues when on tries to add a running cost depending on the density.

\bigskip

Since the seminal work of Benamou and Brenier \cite{Benamou2000}, it is understood that \eqref{equation_problem_unformal_v} can be recast as a convex problem: the idea is to introduce a new unknown $\mbf = \rho \vbf$ as the continuity equation becomes the linear constraint $\dr_t \rho + \nabla \cdot \mbf = 0$, while the density of kinetic energy 
\begin{equation*}
\frac{1}{2} \rho |\vbf|^2 = \frac{|\mbf|^2}{2 \rho}
\end{equation*}
is a jointly convex function of $\rho$ and $\mbf$. Hence, we will rather consider the problem 
\begin{equation}
\label{equation_problem_unformal}
\min_{\rho, \mbf} \left\{ \iint \frac{|\mbf|^2}{2 \rho} \, \ddr t \, \ddr x  \ \text{ such that } \ \dr_t \rho + \nabla \cdot \mbf = 0 \right\},
\end{equation}
whose rigorous formulation is given below in Section \ref{section_framework}. With this convex reformulation, one can hope to compute numerically solutions of the problem \eqref{equation_problem_unformal}: it has been done first in \cite{Benamou2000} and in many subsequent works detailed below. In all of these works, problem \eqref{equation_problem_unformal} is discretized in space and time, yielding a finite dimensional convex optimization problem. Thanks to the general theory of convex optimization, there are strong guarantees that iterative schemes will indeed provide a very good approximation of the finite dimensional problems resulting from the discretization of \eqref{equation_problem_unformal}. However, to the best of the author's knowledge, there are few proofs that solutions of the finite dimensional problems will indeed converge to solutions of \eqref{equation_problem_unformal} as the spatial and temporal grids are refined. 

This problem is delicate and does not fit in classical theories for two (linked) reasons. First, the natural framework for the continuous problem is the set of positive measures: indeed, \eqref{equation_problem_unformal} admits a solution if $\rho_0$ and $\rho_1$ the initial and final values of the density are fixed and equal to arbitrary positive measures sharing the same total mass. Moreover, there is no spatial regularizing effect: if $\rho_0$ is a Dirac mass in $x$ while $\rho_1$ is a Dirac mass is $y$, then one solution is $t \mapsto \rho_t = \delta_{\gamma(t)}$, where $\gamma$ is a constant-speed geodesic joining $x$ to $y$. Thus a satisfying convergence result should work even if the datum are spatially very singular which makes all the theory of finite elements phrased in Hilbert spaces is not available. In addition, the function $(\rho, \mbf) \mapsto |\mbf|^2/(2 \rho)$ is not continuous at the point $(0,0)$, meaning that one has to take care of what happens when there is void, i.e. when the density vanishes. As there is no spatial regularization effect, one really \review{has} to face this issue. In the infinite-dimensional theory this issue is now well-understood (see for instance \citep[Chapter 8]{AGS}), but it is not immediate to see how such theory will apply for discretizations of \eqref{equation_problem_unformal}.

\subsection{Related work}

\begin{figure}
\begin{center}
\begin{tikzpicture}[scale = 0.5]


\draw [<-, line width = 1pt] (1,0) -- (19,0) ;
\draw (10,-0.5) node[below]{$\tau$ temporal step size} ;

\fill [color = black] (19,0)  circle (0.2)  ;
\draw (19,0) node[below]{$\tau = 0$} ;


\draw [<-, line width = 1pt] (0,1) -- (0,11) ;
\draw (-0.5,6) node[above, rotate = 90]{ $\sigma$ spatial step size} ;

\fill [color = black] (0,11)  circle (0.2)  ;
\draw (0,11) node[above]{$\sigma = 0$} ;


\draw (19,11) node[draw]{\makecell[c]{Original problem\\\eqref{equation_problem_unformal}}} ;

\draw (6,3) node[draw]{\makecell[c]{Fully discretized\\problem}} ;

\draw (19,3) node[draw]{\makecell[c]{Semi-discretized\\problem \cite{Maas2011}}} ;


\draw [->, dashed, line width = 1pt] (10,3) -- (15,3) ;
\draw (12.5,3) node[below]{ \small{$\Gamma$-convergence \cite{Erbar2017}} } ;

\draw [->, dashed, line width = 1pt] (19,5) -- (19,9) ;
\draw (19,7) node[right]{ \makecell[l]{\small{Gromov-Hausdorff}\\\small{convergence \cite{Gigli2013, Gladbach2018}} }} ;

\draw [->, dashed, line width = 1pt] (6,5) -- (15,10) ;
\draw (12,6.5) node[right]{ \makecell[l]{\small{$\Gamma$-convergence}\\\small{for finite difference}\\\small{under regularity}\\\small{assumption \cite{Carrillo2019}} }} ;


\draw [->,  line width = 1.5pt] (6,5.5) -- (15,10.5) ;
\draw (10,8) node[left]{ \small{ \textbf{This article} } } ;

\end{tikzpicture}
\end{center}

\caption{Schematic representation of other convergence results present in the literature (dashed lines) and ours (solid line). If one combines the result of \cite{Erbar2017} and \cite{Gladbach2018} then one gets a ``diagonal'' arrow but with conditions on the ratio between temporal and spatial step sizes. Compared to \cite{Carrillo2019} we do not need to assume regularity of the solution of \eqref{equation_problem_unformal} and we can work with more general discretizations than finite difference. ``Convergence'' has a different precise meaning for each arrow, we refer the reader to the original articles to get precise statements.}

\label{figure_previous_results}

\end{figure}
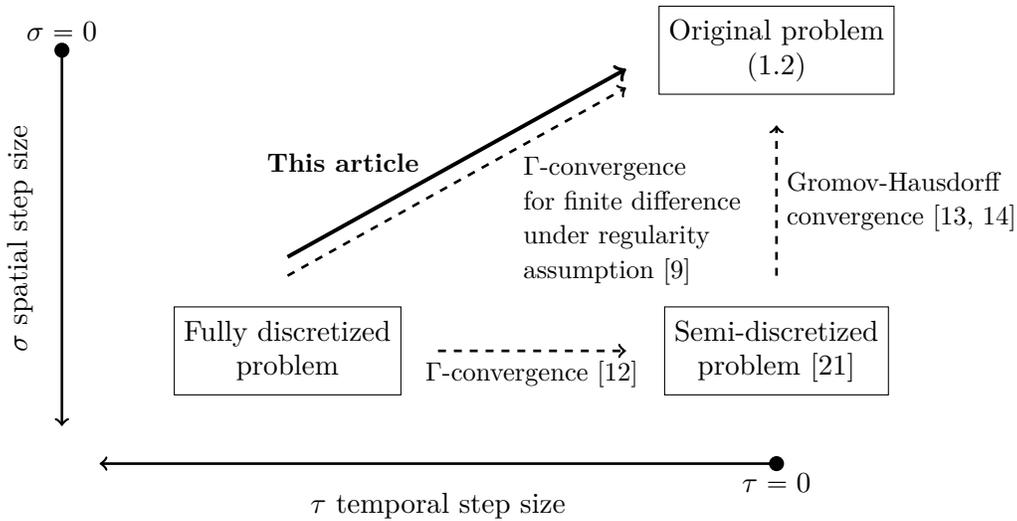

Since the seminal article \cite{Benamou2000}, formulation \eqref{equation_problem_unformal} and its variations have been used and solved numerically after discretization in a variety of context: congested dynamic \cite{Buttazzo2009}, Wasserstein gradient flows \cite{Benamou2016gradient, Carrillo2019}, variational mean field games \cite{Benamou2015, Benamou2016}, unbalanced optimal transport \cite{Chizat2018}, Wasserstein geodesics over surfaces \cite{Lavenant2018} to give a few examples. 

Since the augmented Lagrangian method proposed in \cite{Benamou2000} to solve the (finite-dimensional) problem obtained after discretization, variants have been proposed in the more general framework of \emph{proximal splitting} algorithms \cite{Papadakis2014, Carrillo2019}, and recently with the use of a clever Helmholtz-Hodge decomposition to handle the divergence constraint \cite{Henry2019}. Proofs of convergence of these convex optimization algorithms in the infinite dimensional case were provided by \cite{Guittet2003, Hug2016} even though they are phrased in the framework of Hilbert spaces (that is the authors work with densities being $L^2$ w.r.t. space and time) and not in the one of positive measures. 

Starting with considerations far from numerics, Maas \cite{Maas2011} defined a notion of optimal transport on graphs, which can be read as a semi-discretization of \eqref{equation_problem_unformal}: the time variable is kept continuous, but the spatial variable is replaced by a discrete one, namely the objects are defined on either the vertices or edges of a graph. In this framework, it was later proved in \cite{Gigli2013} and \cite{Gladbach2018} that, if the discretization of space is refined (if one interprets the graphs as a finite volume discretization of the space), then \review{one} gets a finer and finer approximation of \eqref{equation_problem_unformal}: actually it is phrased as a Gromov-Hausdorff convergence of metric spaces (namely, convergence of the spaces of probability measures endowed with the Wasserstein distance). More is said in Subsection \ref{subsection_finite_volumes}, as the discretization proposed by Maas and collaborators fits into our framework. We mention that \cite{Erbar2017} offers a temporal discretization of Problem \eqref{equation_problem_unformal} if the space is already discretized as a graph, and proves convergence of solutions when one refines the temporal discretization, see Figure \ref{figure_previous_results} for a schematic view of these different results. However, if one tries to combine convergence under refinement of temporal discretizations \cite{Erbar2017} and spatial discretization \cite{Gladbach2018}, it leads to a restriction on the ratio between the temporal and spatial step sizes. On the contrary, we will show in the present work that convergence holds without any such restriction. 

Eventually, in \cite{Carrillo2019}, the authors prove convergence of a fully discretized version of \eqref{equation_problem_unformal} to the original problem, though this is not the main point of their article. However, they work only with a finite difference discretization and make strong regularity assumptions on the solution of the problem \eqref{equation_problem_unformal} (density bounded from below by a strictly positive constant, and smooth density and momentum) which are not satisfied if one works with arbitrary positive measures as temporal boundary conditions. As the authors point out \cite[Remark 8]{Carrillo2019}, they observe numerically that their method seems to work even if their strong regularity assumptions are not satisfied.   

\review{After the diffusion of an earlier version of this article, Natale and Todeschi have proposed a finite element discretization of dynamical optimal transport \cite{NataleTodeschi2020} and proved that it fits in the framework of the present work.}

\subsection{Contribution and organization of this article}

The goal of this article is to provide a framework which automatically guarantees \review{convergence} of solutions of a fully discretized version of \eqref{equation_problem_unformal} to the ones of \eqref{equation_problem_unformal}. We will allow initial and final values of the density to be arbitrary positive measures (sharing the same total mass), hence do not require any regularity of the solution of \eqref{equation_problem_unformal}. Moreover, we will give generic conditions on the spatial discretization for this convergence to hold: these conditions will be generic enough to encompass the finite volume discretization of \cite{Gladbach2018} as well as the discretization on triangulations of surfaces proposed by the present author in \cite{Lavenant2018}. The underlying spatial space will be allowed to be either a convex domain in a Euclidean space or a Riemannian manifold as in \cite{Lavenant2018}.

However our result is not quantitative at all: a study of the speed of convergence would likely depend on the regularity of the solution of the infinite dimensional problem but we prefer to focus on the most generic case for the original problem \eqref{equation_problem_unformal}. Moreover, the functionals that we are minimizing are $1$-homogeneous, hence not strictly convex: even a speed of convergence for the value of the problems would not automatically lead to a speed of convergence for the solutions. To the best of our knowledge, this question is completely open.

\bigskip  	
  	
In the rest of this article, we first state precisely the problem \eqref{equation_problem_unformal} we aim to discretize, and sufficient conditions on the spatial discretization to ensure convergence of solutions of \review{the} fully space-time discretized problem: this is the object of Section \ref{section_framework}. The proof of our convergence result is provided in Section \ref{section_proof}. Then in Section \ref{section_examples} we show that our framework applies to spatial discretizations already present in the litterature, namely the one of Gladbach, Kopfer and Maas \cite{Gladbach2018}, as well as the one we proposed with Claici, Chien and Solomon \cite{Lavenant2018}. Eventually, we show in Section \ref{section_extension} that \review{with} little additional work, one can add a penalization of the final density in the functional to be minimized (retrieving one step of the JKO scheme for Wasserstein gradient flows) while still having guarantees of convergence, but that the addition of a running cost depending on the density yields more involved issues that we do not cover.  

\begin{rmk}
We want to emphasize that the methods of proof and general ideas of this article are very much inspired from the aforementioned works \cite{Gladbach2018, Erbar2017, Carrillo2019}. Our techniques are closely related to theirs. However, we think that the framework that we propose shades a new light on their ideas and clarifies between what is really necessary to get this kind of convergence and what was contingent to the specific choice of discretization. 
\end{rmk}

\section{Framework and statement of the result}

If $A$ is a subset of $B$, the function $\1_A$, defined on $B$, takes the value $1$ on $A$ and $0$ elsewhere. In the rest of this article, we will use $C$ to denote a constant independent on some parameters (specified in the context) whose valued may change from line to line. Similarly, we will use $(\varepsilon_\sigma)_\sigma$ to denote a generic function which tends to $0$ as $\sigma \to 0$, but which may change from line to line. 

\label{section_framework}

\subsection{Infinite dimensional problem}

As far as space is concerned, we work with $(X,g)$ a smooth compact Riemannian manifold possibly with a non-empty boundary. The boundary of $X$, denoted by $\dr X$ is also assumed to be infinitely smooth and, importantly, \emph{convex}. For the definition of a convex boundary, we refer to \cite{Wang2009}, see also \cite{Bartolo2002} for a geometric point of view. As an example, the reader can have in mind convex bounded domains in the Euclidean space with smooth boundary or smooth compact Riemannian manifold without boundary. 

\begin{rmk}
The smoothness and convexity assumptions on the boundary will be used only in Proposition \ref{prop_regularization} to regularize curves of measures. However, experiments \cite[Fig. 10]{Lavenant2018} suggest that they probably are not optimal.  
\end{rmk} 

For each point $x \in X$, we denote the tangent space by $T_x X$, this space is equipped with a scalar product $g_x$ and a norm $| \ |_x$. The volume measure is denoted by $\ddr x$, if not specified integration is always performed w.r.t. this measure. The Riemannian distance is $d_g$. The tangent bundle is denoted by $T X$ and $\pi : TX \to X$ is the canonical projection. The space $C(X)$ is the one of continuous functions over $X$, and by a slight abuse of notations we denote by $C(X, TX)$ the space of continuous vector fields over $X$, i.e. of mappings $\vbf : X \to TX$ such that $\pi(\vbf(x)) =x$ for all $x \in X$. We will use bold letters to denote vector fields. These spaces are endowed with the supremum norm, generating the topology of uniform convergence. \review{The field $\nbf_{\dr X} : \dr X \to TX$ is the outward normal to $\dr X$.}

We denote by respectively $\M(X)$ and $\M(TX)$ the (topological) duals of $C(X)$ and $C(X, TX)$ and the duality products will be denoted $\langle \cdot, \cdot \rangle$. We use the same letters for duality products in $C(X)$ and $C(X,TX)$ as it will be clear from context which one is used. The space $\M(X)$, which is nothing else than the space of Borel measures on $X$, contains the convex subspace $\M_+(X)$ made of positive measures. The operator norm of an element $\rho \in \M(X)$ is nothing else than the total variation norm of the measure and is denoted by $\|\rho \|$, and analogously for $\M(TX)$. In particular, if $\rho \in \M_+(X)$, $\|\rho \| = \langle \rho, 1 \rangle$ and such a quantity is called the \emph{mass} of $\rho$. We will often identify a measure with its density w.r.t. the volume measure.

An important particular case (where our results are already new and interesting) is the flat one, namely when $X = \Omega \subset \R^d$ is a closed convex subset of $\R^d$ with smooth boundary. In this case, the volume measure is the Lebesgue measure, $C(X) = C(\Omega)$ and $C(X, TX) = C(\Omega, \R^d)$ is the set of continuous functions valued into $\R^d$. Moreover, $\M(TX) = \M(\Omega)^d$ is nothing else than the set of measures valued in $\R^d$ and defined over $\Omega$.

If $\rho \in \M(X)$ and $\mbf \in \M(TX)$, we define 
\begin{equation}
\label{equation_Benamou_Brenier_x}
A(\rho, \mbf) := \sup_{a,\bbf} \left\{ \langle  \rho, a \rangle + \langle  \mbf, \bbf \rangle \ : \ a \in C(X), \ \bbf \in C(X, TX) \text{ and } \forall x \in X, \  a(x) + \frac{|\bbf(x)|^2_x}{2} \leqslant 0 \right\}.
\end{equation} 
This is the so-called Benamou-Brenier formula: it is clearly a convex lower semi-continuous functional, and (see \cite[Proposition 5.18]{SantambrogioOTAM}) it is finite if and only if $\rho \in \M_+(X)$ and $\mbf$ has a density $\vbf : X \to TX$ w.r.t. $\rho$, and in such a case
\begin{equation*}
A(\rho, \mbf) = A(\rho, \rho \vbf ) = \int_X \frac{1}{2} |\vbf(x)|^2_x ~ \rho(\ddr x).
\end{equation*} 
\review{In other words one can write $A(\rho, \mbf) = \int_X |\mbf|^2/(2 \rho)$ with the convention that for a scalar $\rho \in \R$ and a vector $\mbf \in T_x X$ 
\begin{equation}
\label{equation_convention_b22a}
\frac{|\mbf|^2}{2\rho} := \begin{cases}
+ \infty & \text{if } \rho < 0, \\
+ \infty & \text{if } \rho = 0 \text{ and } |\mbf|_x > 0, \\
0 & \text{if } \rho=|\mbf|_x = 0, \\
\dst{\frac{|\mbf|_x^2}{2\rho}} & \text{otherwise}.
\end{cases}
\end{equation}}
\noindent We will keep such a convention in the rest of the article. Moreover, let us remark that if $\rho \in \M_+(X)$, then clearly one can take \review{$a(x) = - |\bbf(x)|_x^2/2$} in formulation \eqref{equation_Benamou_Brenier_x}. Hence, if one defines $A^\star$ on $\M(X) \times C(X,TX)$ by 
\begin{equation*}
A^\star(\rho, \bbf) := \frac{1}{2} \int_X |\bbf(x)|_x^2 ~ \rho(\ddr x) = \frac{1}{2} \langle \rho, |\bbf|^2 \rangle,
\end{equation*}
then one can simply write 
\begin{equation*}
A(\rho, \mbf) = \sup_{\bbf \in C(X,TX)} \left( \langle \mbf, \bbf \rangle - A^\star(\rho, \bbf) \right)
\end{equation*}
at least if $\rho \in \M_+(X)$, that is if we assume that $\rho$ is positive. The functional $A$ is nothing else than the Legendre transform of $A^\star$ w.r.t. its second variable.  

\begin{rmk}
In the article \cite{Gladbach2018}, the notations for $A$ and $A^\star$ are swapped compared to here. This is because we try to work as much a possible with $(\rho, \mbf)$ as unknowns which are thought as ``primal'' variables, even though rigorously they should be considered as ``dual'' variables.  
\end{rmk}

As far as time is concerned, we assume that the initial time is $t=0$ and the final time $t=1$. Hence the temporal domain is $[0,1]$. This space is endowed with the Lebesgue measure $\ddr t$, and we denote by $\ddr t \otimes \ddr x$ the measure on $[0,1] \times X$ which is the tensorial product of the Lebesgue measure on $[0,1]$ and the volume measure on $X$. We naturally define the spaces $C([0,1] \times X)$ and $C([0,1] \times X, TX)$ of respectively space-time dependent scalar functions and space-time dependent velocity fields. They are endowed with the supremum norm and their topological duals are denoted by $\M([0,1] \times X)$ and $\M([0,1] \times TX)$. Similarly to the previous case, $\M_+([0,1] \times X)$ denotes the subset of $\M([0,1] \times X)$ made of positive measures. The norm (operator norm which is the total variation norm) of $\rho \in \M([0,1] \times X)$ is denoted by $\|\rho \|$, analogously for $\M([0,1] \times TX)$. To make the distinction apparent, we will use double brackets $\llangle \cdot, \cdot \rrangle$ for the duality products of space-time dependent objects. The Benamou-Brenier formula can be extended to space-time dependent objects: if $(\rho, \mbf) \in \M([0,1] \times X) \times \M([0,1] \times TX)$,
\begin{multline*}
\A(\rho, \mbf) := \sup_{a,\bbf} \Bigg\{  \llangle \rho, a \rrangle + \llangle \mbf, \bbf \rrangle \ : \ a \in C([0,1] \times X), \ \bbf \in C([0,1] \times X, TX) \\
\text{ and } \forall (t,x) \in [0,1] \times X, \  a(t,x) + \frac{|\bbf(t,x)|^2_x}{2} \leqslant 0 \Bigg\}.
\end{multline*} 
Similarly, we define $\A^\star$ on $\M([0,1] \times X) \times C([0,1] \times X, X)$ by 
\begin{equation*}
\A^\star(\rho, \bbf) := \frac{1}{2} \iint_{[0,1] \times X} |\bbf(t,x)|_x^2 ~ \rho(\ddr t, \ddr x) = \frac{1}{2} \llangle \rho, |\bbf|^2 \rrangle,
\end{equation*}
and one can simply write, if $\rho \in \M_+([0,1] \times X)$ 
\begin{equation*}
\A(\rho, \mbf) = \sup_{\bbf \in C([0,1] \times X,TX)} \left( \llangle \mbf, \bbf \rrangle - \A^\star(\rho, \bbf) \right).
\end{equation*}
Actually, thanks to a smoothing argument we can even take the supremum over vector fields $\bbf$ in $C^1([0,1] \times X, TX)$ and the identity stays valid. 


The gradient $\nabla$ maps $C^1(X)$ into $C(X,TX)$ while the divergence $\nabla \cdot$ is the adjoint of the gradient w.r.t. the $L^2$ scalar product (weighted by the volume measure). Namely, if $f \in C^1(X)$ and $\mathbf{g} \in C^1(X, TX)$ vanish on a neighborhood of $\dr X$, 
\begin{equation*}
\int_X g_x(\nabla f(x), \mathbf{g}(x)) ~ \ddr x = - \int_X f(x) ( \nabla \cdot \mathbf{g} )(x) ~ \ddr x.
\end{equation*}

If $\rho_0, \rho_1 \in \M_+(X)$, we define $\CE(\rho_0, \rho_1)$ the set of pairs $(\rho, \mbf)$ which satisfies the continuity equation with initial and final values given by respectively $\rho_0$ and $\rho_1$ as follows: 
\begin{align*}
\CE(\rho_0, \rho_1) := \Bigg\{ (\rho, \mbf) \ : \ \forall \phi \in C^1([0,1] \times X), \ 
\llangle \rho, \dr_t \phi \rrangle + \llangle \mbf, \nabla \phi \rrangle = \langle \phi(1, \cdot) , \rho_1 \rangle - \langle \phi(0, \cdot),  \rho_0  \rangle \Bigg\} \\ \subset\M([0,1] \times X) \times \M([0,1] \times TX). 
\end{align*}
Indeed, it is nothing else than a weak formulation of the continuity equation $\dr_t \rho + \nabla \cdot \mbf = 0$ with no-flux boundary conditions \review{$\mbf \cdot \nbf_{\dr X} = 0$} on $\dr X$. In particular, if $\rho_0$ and $\rho_1$ do not have the same total mass, it is easy to see that this set is empty by testing with a constant function.

\begin{rmk}
\label{rmk_rho_curve}
Let $(\rho, \mbf)$ such that $(\rho, \mbf) \in \CE(\bar{\rho}_0, \bar{\rho}_1)$ for some $\bar{\rho}_0, \bar{\rho}_1 \in \M_+(X)$ and $\A(\rho, \mbf) < + \infty$. Thanks to the continuity equation, it is not difficult to see that the temporal marginal of $\rho$ is proportional to the Lebesgue measure on $[0,1]$. In particular, we can define $(\rho_t)_{t \in [0,1]}$ a curve valued in $\M_+(X)$ as the disintegration of $\rho$ w.r.t. its temporal marginal \cite[Theorem 5.3.1]{AGS}. Thanks to the continuity equation and the estimate $\A(\rho, \mbf) < + \infty$, in fact the map $t \mapsto \rho_t \in \M_+(X)$ is continuous for the topology of weak convergence (actually even $1/2$-Hölder w.r.t. the Wasserstein distance) and $\bar{\rho}_0$ (resp. $\bar{\rho}_1$) is the limit as $t \to 0$ (resp. $t \to 1$) of $\rho_t$. In short, $\rho \in \M_+([0,1 \times X])$ can also be seen as a continuous curve valued in $\M_+(X)$, and $\bar{\rho}_0, \bar{\rho}_1$ are the values of the curve at the temporal boundaries.  
\end{rmk}

\bigskip

Let $\rho_0, \rho_1 \in \M_+(X)$ sharing the same total mass. We define the functional $\J_{\rho_0, \rho_1}$ on $\M([0,1] \times X) \times \M([0,1] \times TX)$ by 
\begin{equation*}
\J_{\rho_0, \rho_1}(\rho, \mbf) := \begin{cases}
\A(\rho, \mbf) & \text{if } (\rho, \mbf) \in \CE(\rho_0, \rho_1), \\
+ \infty & \text{otherwise}.
\end{cases}
\end{equation*}
Given what is said above, minimizing $\J_{\rho_0, \rho_1}$ amounts to solve problem \eqref{equation_problem_unformal} described in the introduction. This is a convex optimization problem defined on the space $\M([0,1] \times X) \times \M([0,1] \times TX)$ endowed with the topology of weak convergence.

We will denote by $W_2$ the quadratic Wasserstein distance, see Appendix \ref{section_wasserstein}. The functional $\J_{\rho_0, \rho_1}$ is closely related to this distance as mentioned in the introduction and the next theorem summarizes such a link. However, we will not rely on it in the sequel: the reader unfamiliar with the theory of optimal transport can skip it for a first reading.   

\begin{theo}
\label{theorem_J_W2}
Under the assumption that $\rho_0$ and $\rho_1$ share the same total mass, there exists a minimizer to the functional $\J_{\rho_0, \rho_1}$ and 
\begin{equation}
\label{equation_J_Wasserstein}
\min_{\M([0,1] \times X) \times \M([0,1] \times TX)} \  \J_{\rho_0, \rho_1} = \frac{1}{2} W_2^2(\rho_0, \rho_1).
\end{equation}
Moreover, at least in the case where $X$ is a flat convex domain or a Riemannian manifold without boundary and if $\rho_0$ or $\rho_1$ is absolutely continuous w.r.t. the volume measure then the minimizer is unique. 
\end{theo} 

\begin{proof}
Existence is really easy from the direct method of calculus if variations. The only non trivial thing to check is the existence of at least one competitor, but this it is a byproduct of the proof of \eqref{equation_J_Wasserstein}.

For the proof of \eqref{equation_J_Wasserstein}, one can refer to \cite[Theorem 5.28]{SantambrogioOTAM} for $X$ being a flat convex domain and \cite[Proposition 2.5]{Erbar2010} when $X$ is a Riemannian manifold without boundary. We have not found the case of a Riemannian manifold \emph{with} boundary written explicitly in the literature, but following \cite[Remark 8.3]{Villani2003} the only requirement is to prove that a manifold with boundary can be isometrically embedded in a Euclidean space. Such a property is true, as a Riemannian manifold with boundary can always be seen as a subset of a Riemannian manifold without boundary \cite[Corollary B]{Pigola2016}, and the latter can be embedded isometrically in a Euclidean space thanks to Nash's theorem.

Eventually, uniqueness of the minimizer comes from the proof of uniqueness for the optimal transport problem: see \cite[Proposition 5.32 and Theorem 1.22]{SantambrogioOTAM} for the flat case and \cite[Corollary 7.23]{Villani2008} combined with \cite[Theorems 8 and 9]{Mccann2001} on a manifold without boundary.   
\end{proof}

The goal of this article is to explain how one can approximate $\J_{\rho_0, \rho_1}$ with functionals defined on finite dimensional spaces.

\subsection{Discretization}

The discretizations already proposed in the literature share some similar structure as they try to mimick the one of the infinite dimensional problem. They split the temporal and spatial variables, and although the temporal discretization is rather straightforward, more diverse propositions have been made for the spatial one. As far as the latter is concerned, what is needed is at least an equivalent of the action, and a divergence operator. We assume that we have some finite dimensional approximations of such objects. Namely, we will have a family of models indexed by a parameter $\sigma$ thought as a spatial step size. The parameter $\sigma$ belongs to a subset $\Sigma$ of $(0, + \infty)$ which contains $0$ as an accumulation point and in the sequel ``for all $\sigma > 0$'' means ``for all $\sigma$ in the the set $\Sigma$''.   

\begin{defi}
\label{definition_approximation}
A family of finite dimensional models of dynamical optimal transport is the datum, for any $\sigma > 0$ of 
\begin{equation*}
( \underbrace{\X_\sigma, \Y_\sigma}_{\text{``spaces''}}, 
\underbrace{A_\sigma}_{\text{``action''}},
\underbrace{\Div_\sigma}_{\text{``derivation''}} )
\end{equation*}
where:
\begin{enumerate}
\item $\X_\sigma$ and $\Y_\sigma$ are finite dimensional vector spaces. 
\item $A_\sigma : \X_\sigma \times \Y_\sigma \to [0, + \infty]$ is a proper convex functional $(-1)$-positively homogeneous in its first variable and $2$-homogeneous in its second variable. We require that $A_\sigma$ is non increasing on a cone $\X_{\sigma,+} \subset \X_\sigma$ in the sense that for every $P_1, P_2 \in \X_{\sigma,+}$ and any $\Mbf \in \Y_\sigma$, 
\begin{equation*}
A_\sigma(P_1+ P_2, \Mbf) \leqslant A_\sigma(P_1, \Mbf).
\end{equation*} 
\item $\Div_\sigma : \Y_\sigma \to \X_\sigma$ is a linear operator.
\end{enumerate} 
\end{defi}

\noindent \review{We will think of $\X_\sigma$ as an approximation of $\M(X)$, while $\Y_\sigma$ is an approximation of the space $\M(TX)$ \emph{with no-flux boundary conditions}. The function $A_\sigma$ mimics the action $A$ and $\Div_\sigma$ corresponds the divergence operator.}
The space $\X_{\sigma,+}$ is thought as $\M_+(X)$. We emphasize that we do not require any duality structure on $\X_\sigma$ and $\Y_\sigma$, nor any norm on them. For the moment, we have made no explicit link between these objects and the original ones living on the manifold $X$. As we expect the approximation to be finer and finer as $\sigma \to 0$, the dimension of $\X_\sigma$ and $\Y_\sigma$ are thought as increasing as $\sigma$ decreases. 

By convention, we will use capital letters to denote the counterpart in the spaces $\X_\sigma, \Y_\sigma$ of elements in $\M(X)$ and $\M(TX)$. Hence a generic element of $\X_\sigma$ is $P$ (a capitalized $\rho$), and a generic element in $\Y_\sigma$ is $\Mbf$.   

Next, we take $N+1 \geqslant 2$ the number of spatial time steps, and we use $\tau = 1/N$ to denote the temporal step size.

\begin{defi}
Let $N+1 \geqslant 2$ be given and let us define $\tau=1/N$. Moreover, let $(\X_\sigma, \Y_\sigma, A_\sigma, \Div_\sigma)_{\sigma}$ a family of finite dimensional models of dynamical optimal transport. 

If $\sigma > 0$ is given, and $\bar{P}_0$ and $\bar{P}_1$ are elements in $\X_{\sigma,+}$, we define the functional $\J^{N, \sigma}_{\bar{P}_0, \bar{P}_1}$ on $(\X_\sigma)^{N+1} \times (\Y_\sigma)^{N}$ by
\begin{equation}
\label{equation_discrete_cost}
\J^{N, \sigma}_{\bar{P}_0, \bar{P}_1}( (P_k)_{0 \leqslant k \leqslant N}, (\Mbf_k)_{1 \leqslant k \leqslant N}  ) = \tau \sum_{k=1}^{N} A_\sigma \left( \frac{P_{k-1} + P_{k}}{2}, \Mbf_k \right)
\end{equation}
if $(P_k)_{0 \leqslant k \leqslant N} \in (\X_{\sigma,+})^{N+1}$ and the discrete continuity equation is satisfied, that is
\begin{equation*}
\begin{cases}
\tau^{-1} (P_{k} - P_{k-1})  + \Div_\sigma (\Mbf_k) = 0, & \forall k \in \{1,2, \ldots, N \} \\
P_{0} = \bar{P}_0, & \\
P_{N} = \bar{P}_1, & \\
\end{cases}
\end{equation*}
and $+ \infty$ otherwise. 
\end{defi}

\noindent \review{As already mentioned below Definition \ref{definition_approximation}, $\Y_\sigma$ will be an approximation of $\M(TX)$ with no-flux boundary conditions, hence the apparent absence of spatial boundary conditions in the definition of $\J^{N, \sigma}_{\bar{P}_0, \bar{P}_1}$. Notice that we enforce the condition $P_k \in \X_{\sigma,+}$ for every $k \in \{ 0,1, \ldots, N \}.$}

\begin{rmk}
\label{rmk_averaging_needed}
The function $\J^{N, \sigma}_{\bar{P}_0, \bar{P}_1}$ is mimicking the continuous one $\J_{\rho_0, \rho_1}$. Actually what we are doing, given the continuity equation, is piecewise affine interpolation in time for the densities, and piecewise constant interpolation in time for the momentum. On the other hand, for the cost functional \eqref{equation_discrete_cost}, we transform the piecewise affine interpolation of the density into a piecewise constant one by averaging in time before computing the action. This averaging is necessary because we want to avoid at any cost putting momentum where there is no mass. To understand the issue, let us go back for a moment to the continuous case and let us focus on a single time step: assume that $(\rho, \mbf) \in \M([0,1] \times X) \times \M([0,1] \times TX)$ are respectively affine and constant w.r.t. the time variable. That is, $\rho(t,\cdot) = (1-t) \rho(0,\cdot) + t \rho(1, \cdot)$ and $\mbf(t, \cdot) = \mbf(\cdot)$. To simplify the exposition, assume that they have densities w.r.t. the volume measure. Then
\begin{equation*}
\A(\rho, \mbf) = \iint_{[0,1] \times X} \frac{|\mbf(x)|^2}{2 \rho(t,x)} \, \ddr t \, \ddr x =  \int_X \frac{|\mbf(x)|^2}{2} \underbrace{\left(\int_0^1 \frac{\ddr t}{ (1-t) \rho(0,x) + t \rho(1,x) } \right)}_{= + \infty \text{ if  } \rho(0,x) \text{ or }\rho(1,x) = 0} \, \ddr x.
\end{equation*}
Hence $\A(\rho, \mbf) = + \infty$ if $|\mbf(x)|^2 > 0$ and \emph{either} $\rho(0,x)$ or $\rho(1,x)$ vanishes for $x$ belonging to a set of positive measure. That is, $\mbf$ must have a density w.r.t. both $\rho(0, \cdot)$ and $\rho(1, \cdot)$ for the action to be finite. Such a condition is very restrictive (think that $\rho(0, \cdot)$ could be a Dirac mass), and things are much better with our choice of discretization where $\mbf$ must have a density w.r.t. $\rho(0, \cdot) + \rho(1, \cdot)$. We refer the reader to the proof of Proposition \ref{prop_controllability} where the particular choice of averaging in time for the densities is crucial. 

\review{We do not claim that this choice of staggered grids in time for the density and the momentum is only possible one, but we want to underline that the averaging problem described in this remark has to be taken care of for all temporal discretizations.}
\end{rmk}

Minimizing the function $\J^{N, \sigma}_{\bar{P}_0, \bar{P}_1}$ is a finite dimensional convex optimization problem. As $\J^{N, \sigma}_{\bar{P}_0, \bar{P}_1} \geqslant 0$, it always has a solution if admissible competitors exist. A consequent effort has been devoted to solve it efficiently as recalled earlier. We emphasize that for the actual computation of a minimizer of $\J^{N, \sigma}_{\bar{P}_0, \bar{P}_1}$, one can (and very often does) introduce an additional duality structure on $\X_\sigma$ and $\Y_\sigma$.

\bigskip

Before stating the assumptions sufficient to make the link between these finite dimensional problems and the infinite dimensional one, we need to introduce some additional vocabulary. 

First, in the sequel we will require some estimates to hold uniformly if some test functions are uniformly regular. Hence we need to define some way to measure regularity of test functions on the manifold $X$. To that extent, we use the following definition. 

\begin{defi}
Let $(U^\alpha, \varphi^\alpha)_{\alpha}$ be a \emph{finite} atlas of $X$, where each open set $U^\alpha \subset X$ is mapped by $\varphi_\alpha$ on a open set of either a Euclidean space or a half Euclidean space. 

For $q \geqslant 1$ integer, we say that $\Bc \subset C(X)$ is a \emph{bounded} set of $C^q(X)$ if
\begin{equation*}
\sup_\alpha \sup_{f \in \Bc} \| (f \circ (\varphi^\alpha)^{-1}) \|_{C^q( \varphi^\alpha(U^\alpha)  )} < + \infty 
\end{equation*} 
Similarly, we say that $\Bc \subset C(X, TX)$ is a \emph{bounded} set of $C^q(X, TX)$ if
\begin{equation*}
\sup_\alpha \sup_{f \in \Bc} \| D \varphi^\alpha (f \circ (\varphi^\alpha)^{-1}) \|_{C^q( \varphi^\alpha (U^\alpha)  )} < + \infty  
\end{equation*}
\end{defi}

\noindent In the definition above, the $C^q$ norm of a function defined on a Euclidean space is just the supremum norm of the function and all its partial derivatives up to order $q$. As $X$ is compact and smooth, these definitions do not depend on the choice of the atlas. 

\review{We} will also need to introduce a duality structure on $\Y_\sigma$. Indeed, at least when studying the discretization of \cite{Gladbach2018}, we found this duality structure necessary. We emphasize, however, that to \emph{define} the finite dimensional functional $\J^{N, \sigma}$ this duality structure was not needed, nor in the \emph{consequences} of Theorem \ref{theo_main} below. 

Specifically, let $\Y'_\sigma$ be the dual of $\Y_\sigma$, that is the set of linear form on $\Y_\sigma$. If $L : \Y_\sigma \to \M(TX)$, we can define its \emph{adjoint} operator $L^\top : C(X,TX) \to \Y'_\sigma$ by the following identity: for any $\bbf \in C(X,TX)$ and $\Mbf \in \Y_\sigma$,
\begin{equation*}
\langle \Mbf, L^\top(\bbf)  \rangle_{\Y_\sigma \times \Y'_\sigma} = \langle L(\Mbf), \bbf \rangle_{ \M(TX) \times C(X,TX) }. 
\end{equation*}   
As $\Y_\sigma$ is finite dimensional, $L^\top(\bbf)$ is well-defined for any $\bbf \in C(X, TX)$. Moreover, for a fixed $\sigma > 0$ we define $A^\star_\sigma : \X_\sigma \times \Y'_\sigma \to [0, + \infty)$ to be the Legendre transform of $A_\sigma$ w.r.t. its second variable: if $\Bbf \in \Y'_\sigma$ and $P \in \X_\sigma$ then 
\begin{equation*}
A^\star_\sigma(P, \Bbf) = \sup_{\Mbf \in \Y_\sigma} \left( \langle \Mbf, \Bbf  \rangle_{\Y_\sigma \times \Y'_\sigma} - A_\sigma(P, \Mbf) \right). 
\end{equation*}
As $\Y_\sigma$ is a finite-dimensional and $A_\sigma$ \review{is assumed to be proper and convex (see Definition \ref{definition_approximation})}, the following identity holds \cite[Theorem 12.2]{Rockafellar1970}: for any $P \in \X_{\sigma,+}$ and $\Mbf \in \Y_\sigma$, 
\begin{equation*}
A_\sigma(P, \Mbf) = \sup_{\Bbf \in \Y'_\sigma} \left( \langle \Mbf, \Bbf  \rangle_{\Y_\sigma \times \Y'_\sigma} - A^\star_\sigma(P, \Bbf) \right). 
\end{equation*}
Moreover, one can easily check that $A^\star_\sigma$ is also $2$-homogeneous in its second variable. In particular, a scaling argument leads to, for any $P \in \X_\sigma$, $\Mbf \in \Y_\sigma$ and $\Bbf \in \Y'_\sigma$, 
\begin{equation}
\label{equation_control_A_Astar}
\langle \Mbf, \Bbf  \rangle_{\Y_\sigma \times \Y'_\sigma} \leqslant \review{2} \sqrt{A_\sigma(P, \Mbf)} \sqrt{A^\star_\sigma(P, \Bbf)}.
\end{equation}

Now we can state the assumption necessary to make the link between this finite dimensional problems and the infinite dimensional one. It amounts to give ourselves reconstruction and sampling operators sending $\X_\sigma$ onto $\M(X)$, $\Y_\sigma$ onto $\M(TX)$ and \emph{vice versa}, and assume some sort of \emph{commutation} relations between them.

\begin{defi}
\label{definition_adapated}
Let $(\X_\sigma, \Y_\sigma, A_\sigma, \Div_\sigma)_{\sigma}$ be a family of finite dimensional models of dynamical optimal transport. Such a family is said \emph{adapted} to the manifold $X$ if there exists, for any $\sigma > 0$, ``reconstruction operators'' $R^{\CE}_{\X_\sigma}, R^A_{\X_\sigma}, R_{\Y_\sigma}$ and ``sampling operators'' $S_{\X_\sigma}, S_{\Y_\sigma}$ where: 
\begin{enumerate}
\item $R^{\CE}_{\X_\sigma}, R^A_{\X_\sigma}  : \X_\sigma \to \M(X)$ and $R_{\Y_\sigma} : \Y_\sigma \to \M(TX)$ are linear operators. They are defined everywhere on $\X_\sigma$ and $\Y_\sigma$. Moreover, both $R^{\CE}_{\X_\sigma}$ and $R^{A}_{\X_\sigma}$ map $\X_{\sigma,+}$ into $\M_+(X)$.    
\item $S_{\X_\sigma} : \M(X) \to \X_\sigma$ and $S_{\Y_\sigma} : \M(TX) \to \Y_\sigma$ are linear operators. We require $S_{\X_\sigma}$ to be defined everywhere on $\M(X)$ and to map $\M_+(X)$ into $\X_{\sigma,+}$, while $S_{\Y_\sigma}$ can have a domain $\D(S_{\Y_\sigma})$ different from $\M(TX)$ which must at least contains measures which have a continuous densities (hence is dense in $\M(TX)$).
\end{enumerate}

\noindent Moreover, these operators must satisfy the following properties.

\begin{enumerate}[label=\textnormal{(A\arabic*)}]

\item \label{asmp_interp_sampling} \emph{(Reconstruction is asymptotically the inverse of sampling)} For any $\rho \in \M_+(X)$, there holds 
\begin{equation*}
\lim_{\sigma \to 0} (R^{\CE}_{\X_\sigma} \circ S_{\X_\sigma})(\rho) = \rho
\end{equation*} 
weakly in $\M(X)$.

\item \label{asmp_interps} \emph{(Reconstructions of the density are asymptotically equivalent)} Let $\Bc$ be a bounded set of $C^1(X)$. There exists $(\varepsilon_\sigma)_{\sigma > 0}$ tending to $0$ as $\sigma \to 0$ such that, for any $\sigma > 0$, any $\phi \in \Bc$ and any $P \in \X_\sigma$, 
\begin{equation*}
\langle ( R^{\CE}_{\X_\sigma} - R^A_{\X_\sigma} )(P), \phi  \rangle \leqslant  \varepsilon_\sigma  \|R^{\CE}_{\X_\sigma} (P) \|.
\end{equation*}


\item \label{asmp_interp_der} \emph{(Asymptotic commutation of reconstruction and derivation)} Let $\Bc$ be a bounded set of $C^2(X)$. There exists $(\varepsilon_\sigma)_{\sigma > 0}$ tending to $0$ as $\sigma \to 0$ such that, for any $\Mbf \in \Y_\sigma$ and any $\phi \in \Bc$,
\begin{equation*}
\left| \langle (R^{\CE}_{\X_\sigma} \circ \Div_\sigma) (\Mbf_\sigma), \phi \rangle + \langle R_{\Y_\sigma} (\Mbf_\sigma), \nabla \phi \rangle \right| \leqslant \varepsilon_\sigma   \| R_{\Y_\sigma} (\Mbf )\|.  
\end{equation*}

\item \label{asmp_samp_der} \emph{(Exact commutation of sampling and derivation)} For any $\mbf \in \M(TX)$ which has a density (still denoted by $\mbf$) w.r.t. volume measure in $C^1(X,TX)$ satisfying \review{the no-flux boundary conditions $\mbf(x) \cdot \nbf_{\dr X}(x) = 0$} for all $x \in \dr X$, then for any $\sigma > 0$
\begin{equation*}
S_{\X_\sigma} (\nabla \cdot \mbf) =  (\Div_\sigma \circ S_{\Y_\sigma})(\mbf).
\end{equation*}

\item \label{asmp_interp_action} \emph{(One-sided estimates for the Legendre transform of the action)} Let $\Bc$ be a bounded set of $C^1(X,TX)$. There exists $(\varepsilon_\sigma)_{\sigma > 0}$ tending to $0$ as $\sigma \to 0$ such that the following holds. If $\bbf \in \Bc$ then for any $\sigma > 0$ and any $P \in \X_\sigma$, 
\begin{equation*}
A_\sigma^\star( P, R_{\Y_\sigma}^\top (\bbf) ) \leqslant  A^\star( R^A_{\X_\sigma} (P), \bbf  ) + \varepsilon_\sigma \| R^A_{\X_\sigma}(P) \|. 
\end{equation*} 
Moreover, there exists $C \geqslant 1$ such that for any $\bbf \in C(X,TX)$, there holds
\begin{equation*}
A_\sigma^\star( P, R_{\Y_\sigma}^\top (\bbf) ) \leqslant \frac{C}{2} \| R^A_{\X_\sigma} (P)   \| \left( \sup_{x \in X} |\bbf(x)|_x \right)^2.
\end{equation*}

\item \label{asmp_samp_actio} \emph{(Asymptotic one-sided estimate for the action on smooth densities)} Let $\Bc$ a bounded set of $C^1(X)$ such that all functions in $\Bc$ are uniformly bounded from below by a positive constant. Let $\Bc'$ a bounded subset of $C^1(X,TX)$. Then there exists $(\varepsilon_\sigma)_{\sigma > 0}$ tending to $0$ as $\sigma \to 0$ such that, if $\rho \in \M(X)$ has a density in $\Bc$ and $\mbf \in \M(TX)$ has a density in $\Bc'$ then for any $\sigma > 0$ 
\begin{equation*}
A_\sigma( S_{\X_\sigma} (\rho), S_{\Y_\sigma}(\mbf) )  \leqslant A(\rho, \mbf) + \varepsilon_\sigma. 
\end{equation*}

\item \label{asmp_controllability} \emph{(Controllability)} There exists $(\varepsilon_\sigma)_{\sigma > 0}$ which tends to $0$ as $\sigma \to 0$ and $\omega$ is a continuous functions satisfying $\omega(0) = 0$, such that the following holds. If $x, y$ are points in $X$, then for any $\sigma > 0$ there exists $\hat{P} \in \X^+_\sigma$ and two elements $\hat{\Mbf}_1, \hat{\Mbf}_2 \in \Y_\sigma$ such that 
\begin{equation}
\label{equation_controllability}
\begin{cases}
\Div_\sigma (\hat{\Mbf}_1) = \hat{P} - S_{\X_\sigma} (\delta_x) \\
\Div_\sigma (\Mbf_2) = \hat{P} - S_{\X_\sigma} (\delta_y)
\end{cases} \ \ \text{ and } \ \forall i \in \{ 1,2 \}, \ \ A_\sigma(\hat{P}, \hat{\Mbf}_i) \leqslant \omega(d_g(x,y)) + \varepsilon_\sigma.
\end{equation}

\end{enumerate}
\end{defi}

\noindent Notice that we give ourselves $R^{\CE}_{\X_\sigma}$ and $R^A_{\X_\sigma}$ two different reconstructions operators for the density. Indeed, even if they are assumed to be asymptotically equivalent, each of these operators will be used to pass to the limit different terms in the objective functional (namely the continuity equation with $R^{\CE}_{\X_\sigma}$ and the action with $R^A_{\X_\sigma}$). 

This list of conditions can seem quite long, the reader can take a look at the examples in Section \ref{section_examples} to see how one can check them in practice. Definition \ref{definition_adapated} is only concerned with the spatial discretization, but with its help we will be able to say something about the limit of the functional $\J^{N, \sigma}$ defined for space-time dependent objects. 

Except for \ref{asmp_controllability} which is discussed in Remark \ref{rmk_controllability} below, all requirements are some sort of \emph{commutation} relations. Except for \ref{asmp_samp_der}, all allow for some leeway, that is the presence of an error term. Moreover, although some differential operators are involved in \ref{asmp_interp_der} and \ref{asmp_samp_der}, the remaining properties only deal with $0$-th order quantities: importantly the assumptions \ref{asmp_interp_action} and \ref{asmp_samp_actio} involving the action are only one-sided estimates. 

\begin{rmk}
In practice, one chooses the sampling operators in order for \ref{asmp_samp_der} to be satisfied, namely exact commutation between sampling and derivation. For the reconstruction operators, $R^{A}_{\X_\sigma}$ and $R_{\Y_\sigma}$ are chosen such that the one-sided estimation on the action \ref{asmp_interp_action} holds. Once this is done, one can choose $R^{\CE}_{\X_\sigma}$ to get \ref{asmp_interp_der} asymptotic commutation between reconstruction and derivation. Then, properties \ref{asmp_interp_sampling}, \ref{asmp_interps} and \ref{asmp_samp_actio} are simple \emph{consistency} properties to check, involving no derivatives. 
\end{rmk}

\begin{figure}
\begin{center}
\begin{tikzpicture}[scale = 0.3]

\draw (-4,5) node[left]{Density} ;
\draw (0,-7) node[left]{Momentum} ;

\draw [line width = 1pt] (0,0)  rectangle (10, 10)  ;
\draw [line width = 1pt] (12,0)  rectangle (22, 10)  ;
\draw [line width = 1pt] (24,0)  rectangle (34, 10)  ;

\draw[dashed] (-5,-1) -- (39,-1) ;

\draw [line width = 1pt] (6,-2)  rectangle (17, -12)  ;
\draw [line width = 1pt] (18,-2)  rectangle (28, -12)  ;

\fill [color = black] (2,2)  circle (0.3)  ;
\draw (5,10.5) node[above]{$\delta_x$} ;

\fill [color = black] (32,8)  circle (0.3)  ;
\draw (29,10.5) node[above]{$\delta_y$} ;

\draw [line width = 1pt] (14,2) -- (20,8)  ;
\draw (17,10.5) node[above]{$\hat{\rho}$} ;

\shade[top color = white, bottom color = black] (7.85,-9.9) -- (14.85,-3.9) -- (15,-4.1) -- (8,-10.1) -- cycle  ;
\draw (11,-12.5) node[below]{$|\hat{\mbf}_1|$} ;

\shade[top color = black, bottom color = white] (19.85,-9.9) -- (25.85,-3.9) -- (26,-4.1) -- (20,-10.1) -- cycle  ;
\draw (23,-12.5) node[below]{$|\hat{\mbf}_2|$} ;

\end{tikzpicture}
\end{center}
\caption{Schematic representation of the objects constructed in Remark \ref{rmk_controllability}. Given two Dirac masses $\delta_x$ and $\delta_y$, one constructs an intermediate measure $\hat{\rho}$ and two momenta $\hat{\mbf}_1, \hat{\mbf_2}$ so that $\nabla \cdot \hat{\mbf}_1 = \hat{\rho} - \delta_x$ and $\nabla \cdot \hat{\mbf}_2 = \hat{\rho} - \delta_y$ while the actions $A(\hat{\rho}, \hat{\mbf}_1)$ and $A(\hat{\rho}, \hat{\mbf}_2)$ are controlled by the (Riemannian) distance between $x$ and $y$.}
\label{figure_controllablity}
\end{figure}
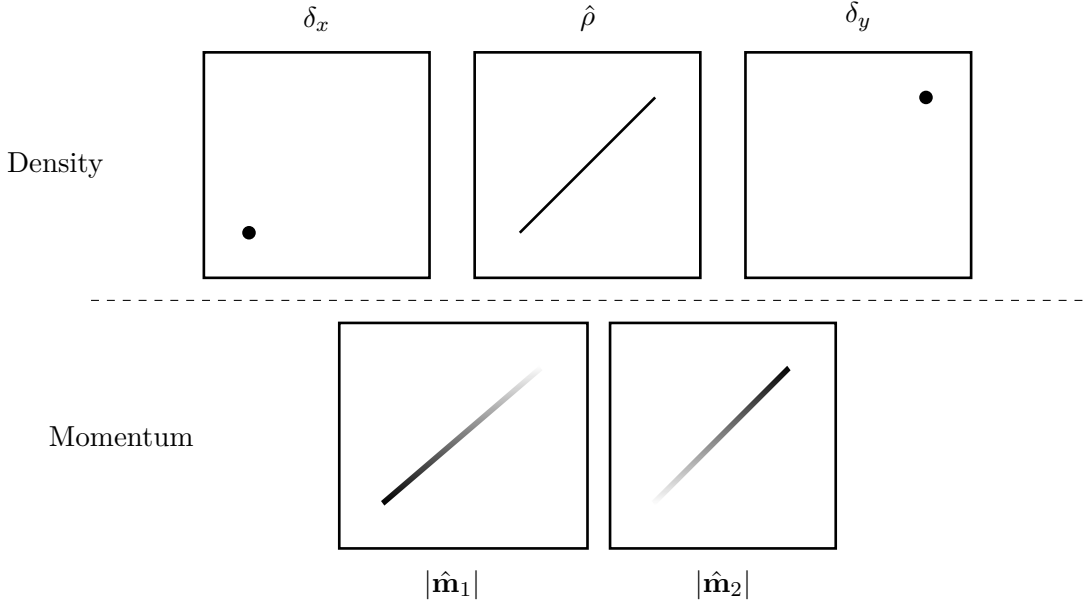

\begin{rmk}
\label{rmk_controllability}
The less standard assumption seems to be \ref{asmp_controllability}. It is used to show how one can control the minimal value of $\J^{N, \sigma}_{S_{\X_\sigma}(\rho_0), S_{\X_\sigma}(\rho_1)}$ in terms of the Wasserstein distance between $\rho_0$ and $\rho_1$ even if $N$ is very small, see Proposition \ref{prop_controllability}. Checking such a condition relies on the particular discretization that is chosen, usually it consists in mimicking the proof of \ref{asmp_controllability} if there is no discretization, that is if $\X_\sigma = \M(X)$ and $\Y_\sigma = \M(TX)$. Indeed, if $x, y \in X$ and $\gamma : [0,1] \to X$ is a constant-speed geodesic joining these two points, let us define $\hat{\rho} \in \M_+(X)$ and $\hat{\mbf}_1, \hat{\mbf}_2 \in \M(TX)$ by, for $\phi \in C(X)$ and $\bbf \in C(X,TX)$,
\begin{multline}
\label{equation_controllability_continuous}
\langle \hat{\rho}, \phi \rangle = \int_0^1 \phi(\gamma(t)) ~ \ddr t, \ \ 
\langle \hat{\mbf}_1, \bbf \rangle =  - \int_0^1 (1-t) ~ g_{\gamma(t)}(\bbf(\gamma(t)), \dot{\gamma}(t)) ~ \ddr t \\
\text{ and } 
\langle \hat{\mbf}_2, \bbf \rangle = \int_0^1 t ~ g_{\gamma(t)} (\bbf(\gamma(t)) , \dot{\gamma}(t)) ~ \ddr t.
\end{multline}
That is, $\hat{\rho}$ is the $\mathcal{H}^1$ Haussdorff measure restricted to $\gamma([0,1])$ (with correct scaling factor to make it a probability measure), while $\hat{\mbf}_1$ and $\hat{\mbf}_2$ have a $L^\infty$ density w.r.t. $\hat{\rho}$ (see Figure \ref{figure_controllablity} for a representation of this construction). One can check that $\nabla \cdot \hat{\mbf}_1 = \hat{\rho} - \delta_x$ while $\nabla \cdot \hat{\mbf}_2 = \hat{\rho} - \delta_y$. Moreover, the density of $\hat{\mbf}_1$ w.r.t. $\hat{\rho}$ at the point $\gamma(t)$ is given by $(1-t) \dot{\gamma}(t)$, whose norm is bounded by $|\dot{\gamma}(t)|_{\gamma(t)} = d_g(x,y)$ the Riemannian distance between $x$ and $y$. Hence we get $A(\hat{\rho}, \hat{\mbf}_1) \leqslant d_g(x,y)^2/2$, and similarly for $A( \hat{\rho}, \hat{\mbf}_2)$. In practice to prove \ref{asmp_controllability} one chooses $\hat{P}$ and $\hat{\Mbf}_i$ that are mimicking $\hat{\rho}$ and $\hat{\mbf}_i$. We cannot use directly the sampling operators $S_{\X_\sigma}$ and $S_{\Y_\sigma}$ as the control that we get on the action thanks \ref{asmp_samp_actio} is only valid for measures with smooth densities, while $\hat{\rho}$ and $\hat{\mbf}_i$ are far from such regularity.
\end{rmk}

\begin{rmk}
\review{\emph{A priori}, the reconstruction and sampling operators for the density do not preserve mass, though they preserve positivity, this is precisely what the space $\X_{\sigma,+}$ is about. On a particular discretization (like the two presented in Section \ref{section_examples}), one can check that indeed mass is preserved by reconstruction and sampling, but this is not needed for the convergence result. Nevertheless, as detailed in the first step of the proof of Theorem \ref{theo_GammaLiminf}, if one reconstructs densities from solutions of the discrete continuity equation, the mass is asymptotically constant in time in the limit $N \to + \infty, \sigma \to 0$. Actually, to get exact mass preservation for solutions of the discrete continuity equation, it is enough to ask that \ref{asmp_interp_der} holds without an error term ($\varepsilon_\sigma = 0$) if $\phi$ is constant.} 
\review{Moreover, the proof of Theorem \ref{theo_GammaLimsup} (see below) shows that if the continuous problem is not empty, then the discretized one isn't either. A close reading of the proof tells us that it relies crucially on Assumptions \ref{asmp_samp_der} and \ref{asmp_controllability}. }  
\end{rmk}

\review{
\begin{rmk}
As mentioned above, the condition of no-flux boundary conditions is built in the choice of the space $\Y_\sigma$ as one can check in the examples of Section \ref{section_examples}. In Definition \ref{definition_adapated} this is only apparent in \ref{asmp_samp_der}, as in the sequel we only use the sampling operator $S_{\Y_\sigma}$ to $\mbf \in \M(TX)$ which are smooth and have no-flux boundary conditions. Actually, one can even define $S_{\Y_\sigma}(\mbf)$ to be zero if $\mbf$ does \emph{not} satisfy no-flux boundary conditions: this does not change the validity of \ref{asmp_samp_der} and \ref{asmp_samp_actio}. 
\end{rmk}}

\begin{rmk}
\review{As we are discretizing a non linear minimization problem with linear constraints, readers familiar with the theory of finite elements can wonder if any kind of $\inf-\sup$ condition is relevant here. As a disclaimer, we are not familiar enough with this condition to give a satisfying answer. We can only say that the present article was written without referring to this $\inf-\sup$ condition, and that the dynamical optimal transport problem as we tackle it is not phrased in the setting of Hilbert spaces, but in the Banach spaces of continuous functions and measures.} 

\review{As far we understand it, the discretizations in Section \ref{section_examples} do not satisfy an $\inf-\sup$ condition. Nevertheless, let us mention that the authors of the recently proposed discretization \cite{NataleTodeschi2020} prove that it fits in our framework and that a discrete $\inf-\sup$ condition is satisfied in their case.} 
\end{rmk}

As said earlier, the introduction of a duality structure in \ref{asmp_interp_action} is not entirely natural. Let us state an other assumption which implies \ref{asmp_interp_action} and which involves only the action and not its Legendre transform. 

\begin{lm}
\label{lemma_interp_action}
Assume that the following property holds: 
\begin{enumerate}[label=\textnormal{(A'\arabic*)}]

\setcounter{enumi}{4}

\item \label{asmp_interp_action_stronger} \emph{(One-sided estimate for the action)} There exists $(\varepsilon_\sigma)_{\sigma}$ tending to $0$ as $\sigma \to 0$ such that, for any $\sigma > 0$ and any $(P, \Mbf) \in \X_\sigma \times \Y_\sigma$, 
\begin{equation*}
A(R^A_{\X_\sigma} (P), R_{\Y_\sigma} (\Mbf)) \leqslant (1 + \varepsilon_\sigma) A_\sigma(P, \Mbf).
\end{equation*}

\end{enumerate}

\noindent Then condition \ref{asmp_interp_action} is satisfied. 

\end{lm} 

\noindent The reason we introduced \ref{asmp_interp_action} instead of \ref{asmp_interp_action_stronger} is that, for the discretization proposed in \cite{Gladbach2018}, only \ref{asmp_interp_action} seems to be provable.

\begin{proof}
Let $\bbf \in C(X,TX)$ and $P \in \X_\sigma$. We know that there exists $\Mbf \in \Y_\sigma$ such that $A^\star_\sigma(P, R_{\Y_\sigma}^\top (\bbf)) = \langle R_{\Y_\sigma}^\top (\bbf), \Mbf \rangle - A_\sigma(P, \Mbf)$. Using the assumption \ref{asmp_interp_action_stronger}, we deduce that 
\begin{multline*}
A^\star_\sigma(P, R_{\Y_\sigma}^\top (\bbf)) = \langle R_{\Y_\sigma}^\top (\bbf), \Mbf \rangle - A_\sigma(P, \Mbf) 
\leqslant \langle R_{\Y_\sigma} (\Mbf) , \bbf \rangle - \frac{1}{1+ \varepsilon_\sigma} A( R^A_{\X_\sigma} (P), R_{\Y_\sigma} (\Mbf)  )  \\
= (1+ \varepsilon_\sigma) \left( \langle \tilde{\mbf}, \bbf \rangle -  A( R^A_{\X_\sigma} (P),\tilde{\mbf} ) \right)
\end{multline*} 
with $\tilde{\mbf} = (1+ \varepsilon_\sigma)^{-1} R_{\Y_\sigma} (\Mbf) $. Using the very definition of the Legendre transform, we see that 
\begin{equation*}
A^\star_\sigma(P, R_{\Y_\sigma}^\top (\bbf)) \leqslant (1+ \varepsilon_\sigma) A^\star ( R^A_{\X_\sigma} (P), \bbf ).
\end{equation*}
Now if $\bbf$ belongs to $\Bc$ a bounded set of $C^1(X,TX)$, we can easily bound the error term $\varepsilon_\sigma A^\star ( R^A_{\X_\sigma} (P), \bbf )$ by $C \varepsilon_\sigma \| R^A_{\X_\sigma}(P) \|$ where $C$ depends only on $\Bc$. Getting the second estimate is easy as in the formula above $\bbf$ can be any continuous function. We then simply use the estimate for the continuous action:
\begin{equation*}
A^\star ( R^A_{\X_\sigma} (P), \bbf ) = \frac{1}{2} \langle R^A_{\X_\sigma} (P), |\bbf|^2  \rangle \leqslant \frac{C}{2} \| R^A_{\X_\sigma} (P)  \|  \left( \sup_{x \in X} |\bbf(x)|_x \right)^2. \qedhere
\end{equation*}
\end{proof}

Before stating the convergence result, we need to explain how to build space-time reconstruction operators from the spatial ones. These operators will be denoted $\Rc^{\CE}_{N, \X_\sigma}, \Rc^{A}_{N, \X_\sigma} : (\X_\sigma)^{N+1} \to \M([0,1] \times X)$ and $\Rc_{N, \Y_\sigma} : (\Y_\sigma)^N \to \M([0,1] \times TX)$. Let $(P,\Mbf) = ((P_k)_{0 \leqslant k \leqslant N}, (\Mbf_k)_{1 \leqslant k \leqslant N})$ a point in $(\X_\sigma)^{N+1} \times (\Y_\sigma)^{N}$. For $\Rc^{\CE}_{N,\X_\sigma}$ we use \review{piecewise} linear interpolations in time on intervals $[(k-1) \tau, k \tau]$, while for $\Rc^A_{N, \X_\sigma}$ as well as $\Rc_{N, \Y_\sigma}$, we use piecewise constant interpolation on these same intervals. Specifically for any test functions $\phi$ and $\bbf$ in respectively $C([0,1] \times X)$ and $C([0,1] \times X, TX)$, 
\begin{align*}
\llangle \Rc^{\CE}_{N, \X_\sigma}(P), \phi \rrangle & = \sum_{k=1}^N \int_{(k-1) \tau}^{k \tau} \left\langle \frac{k \tau - t}{\tau} R^{\CE}_{\X_\sigma}(P_{k-1}) + \frac{t - (k-1) \tau}{\tau} R^{\CE}_{\X_\sigma}(P_k), \phi(t, \cdot)  \right\rangle ~ \ddr t, \\
\llangle \Rc^A_{N, \X_\sigma}(P), \phi \rrangle & = \sum_{k=1}^N \left\langle R^A_{\X_\sigma} \left( \frac{P_{k-1} + P_k}{2} \right), \int_{(k-1)\tau}^{k \tau} \phi(t, \cdot) ~ \ddr t \right\rangle, \\ 
\llangle \Rc_{N, \Y_\sigma }(\Mbf), \bbf \rrangle & = \sum_{k=1}^N \left\langle R_{\Y_\sigma}(\Mbf_k), \int_{(k-1) \tau}^{k \tau} \bbf(t, \cdot) ~ \ddr t \right\rangle.    
\end{align*}

\subsection{Statement of the convergence result}

The main result of the present work is that, properly reconstructed, the minimizers of $\J^{N, \sigma}_{S_{\X_\sigma}(\rho_0), S_{\X_\sigma}(\rho_1)}$ will converge to the ones of $\J_{\rho_0, \rho_1}$ in the limit $N \to + \infty, \sigma \to 0$. In the sequel, by the limit $N \to + \infty, \sigma \to 0$, it means that we can take $(N_n)_{n \in \N}$ and $(\sigma_n)_{n \in \N}$ two sequences, tending respectively to $+ \infty$ and $0$ with \emph{no} specification of the speed at which these convergences occur, and the result holds when we study the limit $n \to + \infty$ of minimizers of $\J^{N_n, \sigma_n}_{S_{\X_{\sigma_n}}(\rho_0), S_{\X_{\sigma_n}}(\rho_1)}$.  

\begin{theo}
\label{theo_main}
Let $(\X_\sigma, \Y_\sigma, A_\sigma, \Div_\sigma)_\sigma$ a family of finite dimensional models of dynamical optimal transport which is adapted to the Riemannian manifold $(X,g)$ in the sense of Definition \ref{definition_adapated}, with reconstruction and sampling operators $(R^{\CE}_{\X_\sigma}, R^A_{\X_\sigma}, R_{\Y_\sigma}, S_{\X_\sigma}, S_{\Y_\sigma})_\sigma$. 

Let $\rho_0, \rho_1 \in \M_+(X)$ sharing the same total mass be given. For any $N$ and $\sigma > 0$, let $(P^{N, \sigma}, \Mbf^{N, \sigma}) \in (\X_\sigma)^{N+1} \times (\Y_\sigma)^N$ be minimizers of $\J^{N, \sigma}_{S_{\X_\sigma}(\rho_0), S_{\X_\sigma}(\rho_1)}$. 

Then, in the limit $N \to + \infty$ and $\sigma \to 0$, up to the extraction of a subsequence, $\Rc^{\CE}_{N, \X_\sigma}(P^{N, \sigma})$ and $\Rc^{A}_{N, \X_\sigma}(P^{N, \sigma})$ will converge weakly to the same limit $\rho \in \M([0,1] \times X)$, and $\Rc_{N, \Y_\sigma}(\Mbf^{N, \sigma})$ will converge weakly to $\mbf \in \M([0,1] \times TX)$, in addition the pair $(\rho, \mbf)$ is a minimizer of $\J_{\rho_0, \rho_1}$ and 
\begin{equation*}
\lim_{N \to + \infty, \sigma \to 0} \J^{N, \sigma}_{S_{\X_\sigma}(\rho_0), S_{\X_\sigma}(\rho_1)} ( P^{N, \sigma}, \Mbf^{N, \sigma} ) = \J_{\rho_0, \rho_1} (\rho, \mbf). 
\end{equation*} 
\end{theo}

\noindent In particular, if the minimizer of $\J_{\rho_0, \rho_1}$ is unique (which can be guaranteed in some situation, see Theorem \ref{theorem_J_W2}), we can remove the ``up to the extraction of a subsequence'' in the statement.  

\begin{rmk}
We emphasize that the reconstructed densities $\Rc^{\CE}_{N, \X_\sigma}(P^{N, \sigma})$, $\Rc^{A}_{N, \X_\sigma}(P^{N, \sigma})$ and the reconstructed momentum $\Rc_{N, \Y_\sigma}(\Mbf^{N, \sigma})$ do not necessarily satisfy a continuity equation, nor have a finite action. Only their weak limit do.  
\end{rmk}

Such a result is in fact obtained by combining two sharper results, which look like a $\Gamma$-convergence result.  

\begin{theo}
\label{theo_GammaLiminf}
Let $(\X_\sigma, \Y_\sigma, A_\sigma, \Div_\sigma)_\sigma$ and $(R^{\CE}_{\X_\sigma}, R^A_{\X_\sigma}, R_{\Y_\sigma}, S_{\X_\sigma}, S_{\Y_\sigma})_\sigma$ be as in Theorem \ref{theo_main}.

Let $\rho_0, \rho_1 \in \M_+(X)$ sharing the same total mass be given. For any $N$ and $\sigma > 0$, let $(P^{N, \sigma}, \Mbf^{N, \sigma}) \in (\X_\sigma)^{N+1} \times (\Y_\sigma)^N$ such that 
\begin{equation*}
\sup_{N, \sigma} \J_{S_{\X_\sigma}(\rho_0), S_{\X_\sigma}(\rho_1)}^{N, \sigma}(P^{N, \sigma}, \Mbf^{N, \sigma}) < + \infty.
\end{equation*}
Then, in the limit $N \to + \infty$ and $\sigma \to 0$, up to the extraction of a subsequence, $\Rc^{\CE}_{N, \X_\sigma}(P^{N, \sigma})$ and $\Rc^{A}_{N, \X_\sigma}(P^{N, \sigma})$ will converge weakly to the same limit $\rho \in \M([0,1] \times X)$, and $\Rc_{N, \Y_\sigma}(\Mbf^{N, \sigma})$ will converge weakly to $\mbf \in \M([0,1] \times TX)$, in addition
\begin{equation*}
\J_{\rho_0, \rho_1}(\rho, \mbf) \leqslant \liminf_{N \to + \infty, \sigma \to 0} \J_{S_{\X_\sigma}(\rho_0), S_{\X_\sigma}(\rho_1)}^{N, \sigma}(P^{N, \sigma}, \Mbf^{N, \sigma}). 
\end{equation*}
\end{theo}

\begin{theo}
\label{theo_GammaLimsup}
Let $(\X_\sigma, \Y_\sigma, A_\sigma, \Div_\sigma)_\sigma$ and $(R^{\CE}_{\X_\sigma}, R^A_{\X_\sigma}, R_{\Y_\sigma}, S_{\X_\sigma}, S_{\Y_\sigma})_\sigma$ be as in Theorem \ref{theo_main}.

Let $\rho_0, \rho_1 \in \M_+(X)$ sharing the same total mass be given, and let $(\rho, \mbf) \in \M([0,1] \times X) \times \M([0,1] \times TX)$ be such that $\J_{\rho_0, \rho_1}(\rho, \mbf) < + \infty$. For any $N$ and $\sigma > 0$, we can build $(P^{N, \sigma}, \Mbf^{N, \sigma}) \in (\X_\sigma)^{N+1} \times (\Y_\sigma)^N$ such that 
\begin{equation*}
\limsup_{N \to + \infty, \sigma \to 0}  \J_{S_{\X_\sigma}(\rho_0), S_{\X_\sigma}(\rho_1)}^{N, \sigma}(P^{N, \sigma}, \Mbf^{N, \sigma}) \leqslant \J_{\rho_0, \rho_1} (\rho, \mbf).
\end{equation*}
\end{theo}

\noindent It is of course very simple to obtain Theorem \ref{theo_main} from Theorems \ref{theo_GammaLiminf} and \ref{theo_GammaLimsup}, the delicate part being the proof of the latter.

\begin{rmk}
\review{This result looks like a $\Gamma$-convergence result, but we did not succeeded in writing it like one mainly for technical issues. The main one is that there is no guarantee in Theorem \ref{theo_GammaLimsup} that the momentum $\Rc_{N, \Y_\sigma}(\Mbf^{N, \sigma})$ converges to $\mbf$ in the limit $N \to + \infty, \sigma \to 0$ because our ``recovery sequence'' is not built by regularizing the momentum, see Remark \ref{rmk_m_not_optimal}. Moreover, in our approach there is no canonical way to inject $\X_\sigma$ and $\Y_\sigma$ into $\M(X)$ and $\M(TX)$, only different ways which are all useful and asymptotically equivalent.}  

\review{The primary motivation of our study is to give guarantees about convergence of numerical discretizations of dynamical optimal transport, and we see Theorem \ref{theo_main} as providing a reasonable answer even though it is not a $\Gamma$-convergence result.} 
\end{rmk}


In the sequel, we first prove Theorems \ref{theo_GammaLiminf} and \ref{theo_GammaLimsup}, and then show how previously proposed discretizations satisfy the assumptions of Definition \ref{definition_adapated}: this is the object of Section \ref{section_proof} and Section \ref{section_examples} respectively. Given the way our statements are phrased, these two sections are entirely independent.

\section{Proof of the convergence result}
\label{section_proof}

The goal of this section is to prove Theorems \ref{theo_GammaLiminf} as well as \ref{theo_GammaLimsup}. As explained in the introduction, the main issue is the discontinuity of the function $(\rho, \mbf) \mapsto |\mbf|^2 /(2 \rho)$ at the point $(0,0)$. As this function is nevertheless lower semi-continuous, Theorem \ref{theo_GammaLiminf} is still quite straightforward. However, for the proof \review{of} Theorem \ref{theo_GammaLimsup} we rely on careful approximation arguments. 

\subsection{Proof of Theorem \ref{theo_GammaLiminf}}

Recall that we assume that $\rho_0, \rho_1 \in \M_+(X)$ sharing the same total mass are given. For any $N, \sigma$, we have $(P^{N, \sigma}, \Mbf^{N, \sigma}) \in (\X_\sigma)^{N+1} \times (\Y_\sigma)^N$ such that 
\begin{equation*}
\sup_{N, \sigma} \J_{S_{\X_\sigma}(\rho_0), S_{\X_\sigma}(\rho_1)}^{N, \sigma}(P^{N, \sigma}, \Mbf^{N, \sigma}) < + \infty.
\end{equation*}

\medskip

\emph{First step: uniform estimates on the mass of the measures}. First, we need to prove that, as measures, $\Rc^{\CE}_{N, \sigma}(P^{N, \sigma})$, $\Rc^{A}_{N, \sigma}(P^{N, \sigma})$ and $\Rc_{N, \Y_\sigma}(\Mbf^{N, \sigma})$ are uniformly bounded. To that extent, we rely both on the continuity equation and the one-sided estimate on the action \ref{asmp_interp_action}. Given a particular discretization (like the ones in Section \ref{section_examples}), a more direct proof may exist, but we want to stress out that it can already be proved only with the assumptions in Definition \ref{definition_adapated} (even though it leads to a more technical estimates).

Let us use the continuity equation to get an estimate on $\| \Rc^{\CE}_{N, \X_\sigma}(P^{N, \sigma}) \|$ and $\|\Rc^A_{N, \X_\sigma}(P^{N, \sigma}) \|$.
\review{We recall that $P^{N, \sigma}_k \in \X_{\sigma, +}$ so $R^{\CE}_{\X_\sigma}(P^{N, \sigma}_{k})$ is a positive measure: to evaluate its mass $\| R^{\CE}_{\X_\sigma}(P^{N, \sigma}_{k}) \|$ we just have to take the duality production with a constant function equal to $1$. So}
let us take the discrete continuity equation, first apply $R^{\CE}_{\X_\sigma}$ and then take the duality product with $\phi$ the constant function equal to $1$. We get, for any $k \in \{ 1,2, \ldots, N \}$,
\begin{equation*}
\langle R^{\CE}_{\X_\sigma}(P^{N, \sigma}_{k}), \phi \rangle = \langle R^{\CE}_{\X_\sigma}(P^{N, \sigma}_{k-1}), \phi \rangle + \tau \langle (R^{\CE}_{\X_\sigma} \circ \Div_\sigma) (\Mbf^{N, \sigma}_k), \phi \rangle.
\end{equation*} 
On the other hand, because of the assumption \ref{asmp_interp_der} and as $\nabla \phi = 0$, up to an error $\varepsilon_\sigma$ which goes to $0$ as $\sigma \to 0$, 
\begin{equation*}
\left| \|R^{\CE}_{\X_\sigma}(P^{N, \sigma}_{k}) \| - \|R^{\CE}_{\X_\sigma}(P^{N, \sigma}_{k-1})\| \right| \leqslant \varepsilon_\sigma \tau \|R_{\Y_\sigma}(\Mbf^{N, \sigma}_k) \|.
\end{equation*} 
Moreover, we know that $R^{\CE}_{\X_\sigma}(P^{N, \sigma}_{0}) = (R^{\CE}_{\X_\sigma} \circ S_{\X_\sigma})(\rho_0)$ converges to $\rho_0$, which has a given mass. Summing the estimates above, we see that for any $k \in \{ 0,1, \ldots, N \}$, 
\begin{equation*}
\left| \|R^{\CE}_{\X_\sigma}(P^{N, \sigma}_{k}) \| - \| \rho_0 \| \right| \leqslant \varepsilon_\sigma + \varepsilon_\sigma \tau \sum_{l=1}^N \|R_{\Y_\sigma}(\Mbf^{N, \sigma}_k) \| =  \varepsilon_\sigma +  \varepsilon_\sigma \| \Rc_{N, \Y_\sigma}(\Mbf^{N, \sigma}) \|. 
\end{equation*}
Using \ref{asmp_interps} tested against the constant function \review{equal to 1, and again as $R^{A}_{\X_\sigma} (P^{N, \sigma}_k) $ and $R^{\CE}_{\X_\sigma} (P^{N, \sigma}_k)$ are positive measures}, we deduce that $\| R^{A}_{X_\sigma} (P^{N, \sigma}_k) \| \leqslant \| R^{\CE}_{X_\sigma} (P^{N, \sigma}_k) \| + \varepsilon_\sigma$ for all $k \in \{ 0,1, \ldots N \}$.
Summing these estimates over $k$, we get that 
\begin{equation}
\label{zz_aux_2}
\|\Rc^{A}_{N,\X_\sigma}(P^{N, \sigma}) \| + \| \Rc^{\CE}_{N,\X_\sigma}(P^{N, \sigma}) \|  \leqslant 2 \| \rho_0 \| + \varepsilon_\sigma + \varepsilon_\sigma \| \Rc_{N, \Y_\sigma}(\Mbf^{N, \sigma}) \|,
\end{equation} 
where $\varepsilon_\sigma$ may have changed from one line to another but keeps the property of tending to $0$ as $\sigma \to 0$.

Now we have to control the norm of the reconstructed momentum $\| \Rc_{N, \Y_\sigma}(\Mbf^{N, \sigma}) \|$. Let $\bbf \in C([0,1] \times X,TX)$ such that $|\bbf(t,x)| \leqslant 1$ for all $(t,x) \in [0,1] \times X$. Not to overburden the equations, let us introduce for $k \in \{ 1,2, \ldots, N \}$ the functions $\bbf_k \in C(X,TX)$ defined by 
\begin{equation*}
\bbf_k(x) = \frac{1}{\tau} \int_{(k-1) \tau}^{k \tau} \bbf(t,x) ~ \ddr t,
\end{equation*}
they also satisfy the pointwise estimate $|\bbf_k(x)|_x \leqslant 1$ for all $k$ and $x$. Using estimate \eqref{equation_control_A_Astar} followed by Cauchy-Schwarz, we get:
\review{
\begin{align*}
\llangle  \Rc_{N, \Y_\sigma}(\Mbf^{N, \sigma}), \bbf \rrangle 
&= \sum_{k=1}^N \tau \langle R_{\Y_\sigma} (\Mbf^{N, \sigma}_k), \bbf_k \rangle \\
& \leqslant \sum_{k=1}^N 2 \tau  \sqrt{ A_\sigma \left( \frac{P^{N, \sigma}_{k-1} + P^{N, \sigma}_{k}}{2} , \Mbf^{N, \sigma}_k  \right)} \sqrt{A^\star_\sigma \left(  \frac{P^{N, \sigma}_{k-1} + P^{N, \sigma}_{k}}{2}, R_{\Y_\sigma}^\top (\bbf_k)   \right)   }  \\
& \leqslant 2 \sqrt{ \sum_{k=1}^N \tau A_\sigma \left( \frac{P^{N, \sigma}_{k-1} + P^{N, \sigma}_{k}}{2} , \Mbf^{N, \sigma}_k  \right)  } \sqrt{ \sum_{k=1}^N \tau A^\star_\sigma \left(  \frac{P^{N, \sigma}_{k-1} + P^{N, \sigma}_{k}}{2}, R_{\Y_\sigma}^\top (\bbf_k)   \right)    }.
\end{align*}
The first term in the product is controlled by $\J_{S_{\X_\sigma}(\rho_0), S_{\X_\sigma}(\rho_1)}^{N, \sigma}(P^{N, \sigma}, \Mbf^{N, \sigma})$ which is assumed to be uniformly bounded in $N, \sigma$, while for the second we use the assumption \ref{asmp_interp_action} and the pointwise estimates on $\bbf_k$. Denoting by $C > 0$ a constant which may change from one equation to another, we end up with 
\begin{equation*}
\llangle \Rc_{N, \Y_\sigma}(\Mbf^{N, \sigma}), \bbf \rrangle \leqslant \sqrt{C} \sqrt{ \frac{C}{2} \sum_{k=1}^N \tau \left\| \frac{R^A_{\X_\sigma} (P^{N, \sigma}_{k-1})+R^A_{\X_\sigma} (P^{N, \sigma}_k)}{2} \right\|   } \leqslant C \sqrt{ \| \Rc^A_{N, \X_\sigma} (P^{N, \sigma}) \|  }.
\end{equation*}}  
Now we can take the supremum over all $\bbf$ and use \eqref{zz_aux_2} that we got from the discrete continuity equation to end up with 
\begin{equation*}
\| \Rc_{N, \Y_\sigma}(\Mbf^{N, \sigma}) \| \leqslant C \sqrt{C + C  \| \Rc_{N, \Y_\sigma}(\Mbf^{N, \sigma}) \| },
\end{equation*} 
with $C > 0$ a constant independent on $N$ and $\sigma$. 

It allows us to conclude that $\| \Rc_{N, \Y_\sigma}(\Mbf^{N, \sigma}) \| $ is bounded uniformly in $N$ and $\sigma$. Using \eqref{zz_aux_2}, we see that $\| \Rc^{A}_{N,\X_\sigma}(P^{N, \sigma}) \|$ and $ \| \Rc^{\CE}_{N,\X_\sigma}(P^{N, \sigma}) \|$ are also bounded uniformly in $N$ and $\sigma$. Hence, up to extraction, $\Rc^{\CE}_{N,\X_\sigma}(P^{N, \sigma})$, $\Rc^A_{N,\X_\sigma}(P^{N, \sigma})$ and $\Rc_{N, \Y_\sigma}(\Mbf^{N, \sigma})$ converge weakly in respectively $\M([0,1] \times X), \M([0,1] \times X)$ and $\M([0,1] \times TX)$ to some limits that we call respectively $\rho, \tilde{\rho}$ and $\mbf$.

\medskip

\emph{Second step: the limits $\rho$ and $\tilde{\rho}$ are the same}. Thanks to the assumption \ref{asmp_interps}, this is easy to see. Let us take $\phi \in C^1([0,1] \times X)$ a smooth function. In particular, $\{ \phi(t, \cdot) \ : \ t \in [0,1] \}$ is a bounded set of $C^1(X,TX)$. Moreover, up to an error controlled by $C \tau$ in $L^\infty$ norm, $\phi(t, \cdot)$ can be replaced by $\phi(k \tau, \cdot)$ for all $t \in [(k-1) \tau, k \tau]$. Given the uniform bound on the mass of $\Rc^{\CE}_{N,\X_\sigma}(P^{N, \sigma})$ and $\Rc^A_{N,\X_\sigma}(P^{N, \sigma})$, we conclude that 
\begin{equation*}
\left| \left\llangle ( \Rc^{\CE}_{N, \X_\sigma} - \Rc^A_{N, \X_\sigma} )(P^{N, \sigma}), \phi \right\rrangle \right| \leqslant \sum_{k=1}^N \tau \left| \left\langle (R^{\CE}_{\X_\sigma} - R^{A}_{\X_\sigma}) \left( \frac{P_{k-1} + P_k}{2} \right), \phi(k \tau, \cdot) \right\rangle \right| + C \tau.
\end{equation*}  
Now, given \ref{asmp_interps} and the uniform bound on $\| \Rc^{\CE}_{N, \X_\sigma}(P^{N, \sigma}) \| $ we see easily that, for some error $\varepsilon_\sigma$ which goes to $0$ as $\sigma \to 0$, 
\begin{equation*}
\left| \left\llangle ( \Rc^{\CE}_{N, \X_\sigma} - \Rc^A_{N, \X_\sigma} )(P^{N, \sigma}), \phi \right\rrangle \right| \leqslant \varepsilon_\sigma + C \tau.
\end{equation*}
Sending $N \to + \infty$ (hence $\tau \to 0$) and $\sigma \to 0$, we end up with $\langle \rho - \tilde{\rho}, \phi \rangle = 0$. As $\phi$ was an arbitrary smooth function, it gives us the equality between $\rho$ and $\tilde{\rho}$.

\medskip

\emph{Third step: passing to the limit in the continuity equation}. Let $\phi \in C^2([0,1] \times X)$ a smooth test function. We will test the continuity equation against $(\Rc^{\CE}_{N, \X_\sigma}(P^{N, \sigma}), \Rc_{N, \Y_\sigma}(\Mbf^{N, \sigma}))$ and pass to the limit. Indeed, an integration by parts followed by the use of the discrete continuity equation leads to 
\begin{align*}
& \llangle  \Rc^{\CE}_{N, \X_\sigma} (P^{N, \sigma}), \dr_t \phi \rrangle  
= \sum_{k=1}^N \int_{(k-1) \tau}^{k \tau} \left\langle \frac{k \tau - t}{\tau} R^{\CE}_{\X_\sigma}(P^{N, \sigma}_{k-1}) + \frac{t - (k-1) \tau}{\tau} R^{\CE}_{\X_\sigma}(P^{N, \sigma}_k), \dr_t \phi(t, \cdot)  \right\rangle \, \ddr t \\
&= \langle R^{\CE}_{\X_\sigma}(P^{N, \sigma}_N), \phi(1, \cdot) \rangle - \langle R^{\CE}_{\X_\sigma}(P^{N, \sigma}_0), \phi(0, \cdot) \rangle - \sum_{k=1}^N \left\langle R^{\CE}_{\X_\sigma} \left( \frac{P^{N, \sigma}_k - P^{N, \sigma}_{k-1}}{\tau} \right), \int_{(k-1) \tau}^{k \tau} \phi(t, \cdot) \,\ddr t  \right\rangle \\
&=\langle ( R^{\CE}_{\X_\sigma} \circ S_{\X_\sigma})(\rho_1), \phi(1, \cdot) \rangle - \langle ( R^{\CE}_{\X_\sigma} \circ S_{\X_\sigma})(\rho_0), \phi(0, \cdot) \rangle + \sum_{k=1}^N \left\langle (R^{\CE}_{\X_\sigma} \circ \Div_\sigma)( \Mbf^{N, \sigma}_k ), \int_{(k-1) \tau}^{k \tau} \phi(t, \cdot) \, \ddr t  \right\rangle.
\end{align*} 
As $\{ \phi(t, \cdot) \ : \ t \in [0,1] \}$ is a bounded set of $C^2(X)$, thanks to \ref{asmp_interp_der}, with an error $\varepsilon_\sigma$ tending to $0$ as $\sigma \to 0$, we can write that 
\begin{multline*}
\left| \underbrace{\sum_{k=1}^N \left\langle R_{\Y_\sigma}(\Mbf^{N, \sigma}_k), \int_{(k-1) \tau}^{k \tau} \nabla \phi(t, \cdot) \, \ddr t \right\rangle}_{ = \llangle \Rc_{N, \Y_\sigma}(\Mbf^{N, \sigma}), \nabla \phi \rrangle }  + \sum_{k=1}^N \left\langle (R^{\CE}_{\X_\sigma} \circ \Div_\sigma)( \Mbf^{N, \sigma}_k ), \int_{(k-1) \tau}^{k \tau} \phi(t, \cdot) \, \ddr t  \right\rangle \right| \\
\leqslant \varepsilon_\sigma \| \Rc_{N,\Y_\sigma} (\Mbf^{N, \sigma}) \|.
\end{multline*} 
Plugging back this information, we see that 
\begin{multline*}
\left| \llangle \Rc^{\CE}_{N, \X_\sigma} (P^{N, \sigma}), \dr_t \phi \rrangle + \llangle \Rc_{N, \Y_\sigma}(\Mbf^{N, \sigma}), \nabla \phi \rrangle - \langle ( R^{\CE}_{\X_\sigma} \circ S_{\X_\sigma})(\rho_1), \phi(1, \cdot) \rangle + \langle ( R^{\CE}_{\X_\sigma} \circ S_{\X_\sigma})(\rho_0), \phi(0, \cdot) \rangle   \right| \\
\leqslant  \varepsilon_\sigma \| \Rc_{N,\Y_\sigma} (\Mbf^{N, \sigma}) \|.
\end{multline*}
As $\| \Rc_{N,\Y_\sigma} (\Mbf^{N, \sigma}) \|$ is uniformly bounded in $N$ and $\sigma$, sending $N \to + \infty$ and $\sigma \to 0$, with the help of \ref{asmp_interp_sampling} to handle the boundary terms, we end up with
\begin{equation*}
\llangle \rho, \dr_t \phi \rrangle + \llangle \mbf, \nabla \phi \rrangle - \langle \rho_1, \phi(1, \cdot) \rangle + \langle \rho_0, \phi(0, \cdot) \rangle = 0. 
\end{equation*}
As $\phi$ is an arbitrary smooth function, it means that $(\rho, \mbf) \in \CE(\rho_0, \rho_1)$, that is it satisfies the continuity equation with boundary conditions $(\rho_0, \rho_1)$.

\medskip

\emph{Fourth step: passing to the limit the action}. In this abstract setting, this is simple. Even though \ref{asmp_interp_action_stronger} implies \ref{asmp_interp_action}, let us first show the result if \ref{asmp_interp_action_stronger} holds as it is a one-line estimate. Indeed, given \ref{asmp_interp_action_stronger} and the way $\Rc^{A}_{N, \X_\sigma}, \Rc_{N, \Y_\sigma}$ are defined, 
\begin{multline*}
\J^{N, \sigma}_{S_{\X_\sigma}(\rho_0), S_{\X_\sigma}(\rho_1)}( P^{N, \sigma}, \Mbf^{N, \sigma}  ) = \tau \sum_{k=1}^{N} A_\sigma \left( \frac{P^{N, \sigma}_{k-1} + P^{N, \sigma}_{k}}{2}, \Mbf^{N, \sigma}_k \right) \\ 
\geqslant \frac{1}{1+ \varepsilon_\sigma} \tau \sum_{k=1}^N A \left( R^{A}_{\sigma} \left( \frac{P^{N, \sigma}_{k-1} + P^{N, \sigma}_{k}}{2} \right), R_{\Y_\sigma}(\Mbf^{N, \sigma}_k)  \right) 
= \frac{1}{1+ \varepsilon_\sigma} \A( \Rc^{A}_{N, \X_\sigma}(P^{N, \sigma}), \Rc_{N, \Y_\sigma}(\Mbf^{N, \sigma}) ).
\end{multline*}
As the action $\A$ is lower semi-continuous on $\M([0,1] \times X) \times \M([0,1] \times TX)$, we can easily pass to the limit. 

On the other hand, let us just assume \ref{asmp_interp_action}. Let us fix $\eta > 0$, we take $\bbf \in C(^1[0,1] \times X, TX)$ smooth such that
\begin{equation*}
\A(\tilde{\rho}, \mbf) \leqslant \llangle \mbf, \bbf \rrangle - \A^\star( \tilde{\rho}, \bbf  ) + \eta.
\end{equation*} 
Now, by continuity of $\A^\star(\cdot, \bbf )$ (w.r.t. weak convergence) and the weak convergence of $\Rc_{N,\Y_\sigma}( \Mbf^{N, \sigma} )$ and $\Rc^A_{N, \X_\sigma}(P^{N, \sigma})$, one sees that 
\begin{multline*}
\llangle \mbf, \bbf \rrangle - \A^\star( \tilde{\rho}, \bbf  )
= \lim_{N \to + \infty, \sigma \to 0} \left(  \llangle \Rc_{N,\Y_\sigma}( \Mbf^{N, \sigma} ), \bbf \rrangle - \A^\star( \Rc^A_{N, \X_\sigma}(P^{N, \sigma}), \bbf  ) \right) \\
=  \lim_{N \to + \infty, \sigma \to 0}  \sum_{k=1}^N \int_{(k-1) \tau}^{k \tau} \left(  \langle \Mbf^{N, \sigma}_k, R_{\Y_\sigma}^\top ( \bbf(t, \cdot) ) \rangle - A^\star \left( R^A_{\X_\sigma} \left( \frac{P_{k-1}^{N, \sigma} + P_{k}^{N, \sigma}}{2} \right)  , \bbf(t, \cdot)  \right) \right) \,\ddr t. 
\end{multline*}
As $\{ \bbf(t, \cdot), \ t \in [0,1] \}$ is a bounded set of $C^1(X,TX)$, we can use \ref{asmp_interp_action} to write that
\begin{multline*}
\llangle \mbf,  \bbf \rrangle - \A^\star( \tilde{\rho}, \bbf  )  
\leqslant  \liminf_{N \to + \infty, \sigma \to 0}  \sum_{k=1}^N \int_{(k-1) \tau}^{k \tau} \Bigg( \langle \Mbf^{N, \sigma}_k, R_{\Y_\sigma}^\top ( \bbf(t, \cdot) ) \rangle   -  A^\star_\sigma \Bigg(  \frac{P_{k-1}^{N, \sigma} + P_{k}^{N, \sigma}}{2}  ,   R_{\Y_\sigma}^\top ( \bbf(t, \cdot))  \Bigg) \\  
+ \frac{\varepsilon_\sigma}{2} \| R^A_{\X_\sigma}(P^{N, \sigma}_{k-1}) + R^A_{\X_\sigma}(P^{N, \sigma}_k) \| \Bigg) \, \ddr t  
\end{multline*}
Then, using the definition of the Legendre transform, we conclude that
\begin{multline*}
\A(\tilde{\rho} , \mbf) - \eta \leqslant \liminf_{N \to + \infty, \sigma \to 0}  \left[ \left\{ \tau \sum_{k=1}^N  A_\sigma \left( \frac{P_{k-1}^{N, \sigma} + P_{k}^{N, \sigma}}{2}, \Mbf_k^{N, \sigma}  \right) \right\} + \varepsilon_\sigma \| \Rc^A_{N, \X_\sigma}(P^{N,\sigma}) \| \right]    \\ 
\leqslant \liminf_{N \to + \infty, \sigma \to 0}  \J_{S_{\X_\sigma}(\rho_0), S_{\X_\sigma}(\rho_1)}^{N, \sigma}(P^{N, \sigma}, \Mbf^{N, \sigma}). 
\end{multline*}
Notice that the error term has disappeared as $\| \Rc^A_{N, \X_\sigma}(P^{N,\sigma}) \|$ is uniformly bounded in $N$ and $\sigma$. As $\eta$ can be taken arbitrary small, we end up with
\begin{equation}
\label{equation_GammaLiminf_A}
\liminf_{N \to + \infty, \sigma \to 0} \J_{S_{\X_\sigma}(\rho_0), S_{\X_\sigma}(\rho_1)}^{N, \sigma}(P^{N, \sigma}, \Mbf^{N, \sigma}) \geqslant \A(\tilde{\rho}, \mbf).
\end{equation}
We have seen $\tilde{\rho} = \rho$ and $(\rho, \mbf) \in \CE(\rho_0, \rho_1)$: it makes the right hand side equal to $\J_{\rho_0, \rho_1}(\rho, \mbf)$. \qed

\begin{rmk}
As one can see in the proof of the first and third steps, in fact $(P^{N, \sigma}, \Mbf^{N, \sigma})$ do not need to satisfy exactly the discrete continuity equation, one could allow for some leeway. As pointed out in \cite[Remark 2]{Carrillo2019}, allowing for some leeway can help to speed up the numerical computation of a minimum of $\J^{N, \sigma}$. 
\end{rmk}

\subsection{Proof of Theorem \ref{theo_GammaLimsup}}

Such a proof is more involved as it relies on a careful regularization procedure, as well as the use of the controllability assumption \ref{asmp_controllability} to handle what is happening near the temporal boundaries. In the sequel, by an abuse of notations, we will identify a measure with its density w.r.t. the volume measure. Let us take $\rho_0, \rho_1$ sharing the same total mass and $(\rho, \mbf) \in \M([0,1] \times X) \times \M([0,1] \times TX)$ such that $\J_{\rho_0, \rho_1}(\rho, \mbf) < + \infty$. To prove our result, we will first regularize $(\rho, \mbf)$ and then sample via $S_{\X_\sigma}$ and $S_{\Y_\sigma}$. Because of the regularization we lose the temporal boundary conditions: this is to remedy to this problem that we use the assumption \ref{asmp_controllability}, see Proposition \ref{prop_controllability} below.

We recall that $W_2$ denotes the quadratic Wasserstein distance, see Appendix \ref{section_wasserstein}.

\medskip

\emph{First tool: regularization of $(\rho, \mbf)$}. The first tool is to regularize the continuous pair $(\rho, \mbf)$. This is object of the following proposition.  

\begin{prop}
\label{prop_regularization}
Let $(\rho, \mbf) \in \M([0,1] \times X) \times \M([0,1] \times TX)$ and $\rho_0, \rho_1 \in \M_+(X)$ such that $\J_{\rho_0, \rho_1}(\rho, \mbf) < + \infty$. For any $\eta > 0$, there exists $(\tilde{\rho}, \tilde{\mbf}) \in \M([0,1] \times X) \times \M([0,1] \times TX)$ such that: 
\begin{enumerate}
\item The densities of $\tilde{\rho}$ and $\tilde{\mbf}$ w.r.t. $\ddr t \otimes \ddr x$ (still denoted by $\tilde{\rho}$ and $\tilde{\mbf}$) are smooth in the sense that $\tilde{\rho} \in C^1([0,1] \times X)$ and is bounded from below uniformly on $[0,1] \times X$ by a strictly positive constant; while $\tilde{\mbf} \in C^1([0,1] \times X, TX)$. 
\item The pair $(\tilde{\rho}, \tilde{\mbf})$ belongs to $\CE(\tilde{\rho}(0, \cdot) , \tilde{\rho}(1, \cdot))$ and $W_2(\tilde{\rho}(0, \cdot) , \rho_0) \leqslant \eta$, as well as $W_2(\tilde{\rho}(1, \cdot), \rho_1) \leqslant \eta$. 
\item The following estimate on the action holds: 
\begin{equation*}
\A(\tilde{\rho}, \tilde{\mbf}) \leqslant (1+ \eta) \A(\rho, \mbf). 
\end{equation*}
\end{enumerate}
\end{prop}

\noindent Though we have not found this result phrased like this in the literature, closely related ones are available. As it can be considered as standard and it is not the core of the argument, we delay the proof of such a result until Appendix \ref{section_regularization_curves}. Notice that we have not mentioned in the Proposition whether $\tilde{\rho}$ and $\tilde{\mbf}$ are closed to $\rho$ and $\mbf$ as it is irrelevant for the rest of the analysis, more is said in Remark \ref{rmk_m_not_optimal}.

\medskip

\emph{Second tool: controllability}. The issue with the regularization procedure is that it does not keep the boundary values. Assumption \ref{asmp_controllability} will help us to do a little surgery near the temporal boundaries.

\begin{prop}
\label{prop_controllability}
There exists a continuous function $\hat{\omega} : [0, + \infty) \to [0, + \infty)$ independent on $N$ and $\sigma$, such that $\hat{\omega}(0) = 0$, and an error $\varepsilon_{N, \sigma}$ depending only on $N, \sigma$ and going to $0$ in the limit $N \to + \infty, \sigma \to 0$ such that for any $\rho_0, \rho_1 \in \M_+(X)$ sharing the same total mass, there holds
\begin{equation*}
\min_{(\X_\sigma)^{N+1} \times (\Y_\sigma)^N} \J^{N, \sigma}_{S_{\X_\sigma}(\rho_0), S_{\X_\sigma}(\rho_1)} \ \leqslant  \hat{\omega}(W_2(\rho_0, \rho_1)) + \varepsilon_{N, \sigma} \| \rho_0 \|.  
\end{equation*}
\end{prop}

\noindent Here the key point is that $\rho_0, \rho_1$ are \emph{arbitrary} positive measures and there is no condition on the ratio between the temporal and spatial step sizes. 

\begin{proof}
For simplicity (the only effect would be the appearance of errors of order $\tau$ which can be absorbed in $\varepsilon_{N, \sigma}$), we assume that $N$ is even. Let us also fix $\sigma > 0$.

We first consider the case where $\rho_0 = \delta_x$ and $\rho_1 = \delta_y$ with $x,y \in X$. Given assumption \ref{asmp_controllability}, there exists $\hat{\Mbf}_1, \hat{\Mbf}_2 \in \Y_\sigma$ and $\hat{P} \in \X_{\sigma,+}$ satisfying \eqref{equation_controllability}. Now let us take $\chi : [0,1] \to [0,1]$ the continuous function defined by 
\begin{equation*}
\chi(t) = \begin{cases}
4 t^2 & \text{if } t \leqslant 1/2, \\ 
4 (1-t)^2 & \text{if } t \geqslant 1/2.
\end{cases}
\end{equation*}
In particular, and it was chosen for that, $\int_0^1 \dot{\chi}^2/\chi < + \infty$ while $\chi(0) = \chi(1) = 0$ and $\chi(1/2) = 1$. We define $((P_k)_{0 \leqslant k \leqslant N}, (\Mbf_k)_{1 \leqslant k \leqslant N}) \in \X_\sigma^{N+1} \times \Y_\sigma^N$ by 
\begin{multline*}
P_k := \begin{cases} 
(1 - \chi(k \tau)) S_{\X_\sigma}(\delta_x) + \chi(k \tau) \hat{P} & \text{if } k \leqslant N/2, \\
(1-\chi(k \tau)) S_{\X_\sigma}(\delta_y) + \chi(k \tau) \hat{P} & \text{if } k \geqslant N/2, \\
\end{cases} \\
\text{ and } \Mbf_k = 
\begin{cases}
\dst{ \frac{ ( \chi(k \tau) - \chi((k-1) \tau))  }{\tau}} \hat{\Mbf}_1  & \text{if } k \leqslant N/2, \\
\dst{ \frac{ ( \chi(k \tau) - \chi((k-1) \tau))  }{\tau}} \hat{\Mbf}_2 & \text{if } k > N/2.
\end{cases}
\end{multline*}   
Notice that, for $k \leqslant N/2$, $P_k$ is just a convex combination of $S_{\X_\sigma}(\delta_x)$ and $\hat{P}$, while all the $\Mbf_k$ are proportional to $\hat{\Mbf}_1$; a similar situation occurs for $k \geqslant N/2$. By construction and in particular thanks to \eqref{equation_controllability}, it is quite straightforward to see that $((P_k)_{0 \leqslant k \leqslant N}, (\Mbf_k)_{1 \leqslant k \leqslant N})$ satisfies the discrete continuity equation. We need to estimate the action: for the sake of the exposition, let us just estimate on the part $k \leqslant N/2$, the other part being completely symmetric. We use first that $S_{\X_\sigma}(\delta_x) \in \X_{\sigma,+}$ and the action only decreases if we add an element of $\X_{\sigma,+}$ to its first variable, and then the $(-1,2)$ homogeneity of the action $A_\sigma$:
\begin{align*}
& \A_{N, \sigma} ((P_k)_{0 \leqslant k \leqslant N/2}, (\Mbf_k)_{1 \leqslant k \leqslant N/2}) 
\\ 
& = \tau \sum_{k=1}^{N/2} A_\sigma \left( \left[ 1 - \frac{\chi((k-1) \tau) + \chi(k \tau)}{2} \right] S_{\X_\sigma}(\delta_x) + \frac{\chi((k-1) \tau) + \chi(k \tau)}{2} \hat{P}, \frac{ ( \chi(k \tau) - \chi((k-1) \tau))  }{\tau} \hat{\Mbf}_1  \right) \\
& \leqslant  \tau \sum_{k=1}^{N/2} A_\sigma \left( \frac{\chi((k-1) \tau) + \chi(k \tau)}{2} \hat{P}, \frac{ ( \chi(k \tau) - \chi((k-1) \tau))  }{\tau} \hat{\Mbf}_1  \right) \\
& =  A_\sigma(\hat{P}, \hat{\Mbf}_1) \sum_{k=1}^{N/2} \frac{(\chi(k \tau) - \chi((k-1) \tau))^2}{\tau (\chi(k \tau) + \chi((k-1) \tau))}.
\end{align*} 
By assumption we can control $A_\sigma(\hat{P}, \Mbf_1)$ by $\omega(d_g(x,y)) + \varepsilon_\sigma$. About the remaining sum, notice that it is nothing else than a discretization of  $\int_0^1 \dot{\chi}^2/\chi$ which is finite. More precisely given the explicit expression of $\chi$, 
\begin{equation*}
\frac{(\chi(k \tau) - \chi((k-1) \tau))^2}{\tau (\chi(k \tau) + \chi((k-1) \tau))} = 4 \frac{\tau (2k-1)^2}{k^2 + (k-1)^2} \leqslant 16 \tau
\end{equation*}
hence the sum will be bounded independently on $N$. \review{The reader can take a look back at Remark \ref{rmk_averaging_needed} and check that in the present computation it is crucial that the denominator is $\chi(k \tau) + \chi((k-1) \tau)$ and is non zero if $k=1$.} 

A similar computation can be performed exactly in the same way for $k \geqslant N/2$. In short the sum is bounded independently on $N$, which translates in  
\begin{equation*}
\J^{N, \sigma}_{S_{\X_\sigma}(\delta_x), S_{\X_\sigma}(\delta_y)}(P, \Mbf) \leqslant C \omega(d_g(x,y)) + \varepsilon_{N,\sigma}. 
\end{equation*}
and it is exactly what we wanted to prove. 

Next, let us consider the general case: we take $\rho_0, \rho_1 \in \M(X)$ with the same total mass. Let $\pi \in \M(X \times X)$ an optimal transport plan between them (see Appendix \ref{section_wasserstein}). For each $(x,y) \in X \times X$, we consider the pair $(P^{xy}, \Mbf^{xy})$ built as above which has an energy $\J^{N, \sigma}_{S_{\X_\sigma}(\delta_x), S_{\X_\sigma}(\delta_y)}(P^{xy}, \Mbf^{xy})$ bounded by $C \omega(d_g(x,y)) + \varepsilon_{N,\sigma}$. Then we simply set 
\begin{equation*}
(P, \Mbf) := \iint_{X \times X} (P^{xy}, \Mbf^{xy}) \, \pi(\ddr x, \ddr y) \in (\X^+_\sigma)^{N+1} \times (\Y_\sigma)^N. 
\end{equation*}
By linearity of the discrete continuity equation, the latter is still satisfied with boundary conditions $(S_{\X_\sigma}(\rho_0), S_{\X_\sigma}(\rho_1))$ as $\pi$ has appropriate marginals. Moreover, by convexity of the action $A_\sigma$ and its $1$-homogeneity, 
\begin{multline*}
\J^{N, \sigma}_{S_{\X_\sigma}(\rho_0), S_{\X_\sigma}(\rho_1)}(P, \Mbf) \leqslant \iint_{X \times X} \J^{N, \sigma}_{S_{\X_\sigma}(\delta_x), S_{\X_\sigma}(\delta_y)}(P^{xy}, \Mbf^{xy}) \, \pi(\ddr x, \ddr y) \\ 
\leqslant \iint_{X \times X} ( \omega( d_g(x,y) ) + \varepsilon_{N, \sigma}) \, \pi(\ddr x, \ddr y).  
\end{multline*}
Now, as $\omega$ is continuous and $\omega(0) = 0$, it is clear that if $\iint d_g(x,y)^2 \, \pi(\ddr x, \ddr y)$ goes to $0$, so does $\iint_{X \times X} \omega( d_g(x,y) ) \, \pi(\ddr x, \ddr y)$, hence the right hand side can be written $\hat{\omega}(W_2(\rho_0, \rho_1))$ with $\hat{\omega}$ continuous and $\hat{\omega}(0) = 0$, up to an error $\varepsilon_{N, \sigma} \| \pi \| = \varepsilon_{N, \sigma} \| \rho_0 \|$.
\end{proof}

\begin{figure}
\begin{center}
\begin{tikzpicture}[scale = 1]

\draw[line width = 1pt] (0,0) -- (14,0) ;
\fill [color = black] (0,0)  circle (0.1)  ;
\fill [color = black] (14,0)  circle (0.1)  ;

\draw (3,0) node[above]{$(\rho, \mbf)$} ; 
\draw (0,0) node[left]{$\rho_0$} ; 
\draw (14,0) node[right]{$\rho_1$} ;


\draw [->, line width = 1pt] (4, -0.3) -- (4,-1.7) ;
\draw (4.3,-1) node[right]{\makecell[l]{\small{Regularization}\\~~\small{(Proposition \ref{prop_regularization})}}} ;
\draw (3,-2) node[above]{($\tilde{\rho}, \tilde{\mbf})$} ;

\draw[line width = 1pt] (0,-2) -- (14,-2) ;
\fill [color = black] (0.1,-1.9)  rectangle (-0.1, -2.1)  ;
\fill [color = black] (14.1,-1.9)  rectangle (13.9, -2.1)  ;

\draw (0,-2) node[left]{$\tilde{\rho}_0$} ; 
\draw (14,-2) node[right]{$\tilde{\rho}_1$} ;


\draw [->, line width = 1pt] (4, -2.3) -- (4,-3.7) ;
\draw (4.3,-3) node[right]{\small{\review{Squeezing}}} ;

\draw[line width = 1pt] (2,-4) -- (12,-4) ;
\fill [color = black] (2.1,-3.9)  rectangle (1.9, -4.1)  ;
\fill [color = black] (12.1,-3.9)  rectangle (11.9, -4.1)  ;

\draw (2,-4) node[above]{$\tilde{\rho}_0$} ; 
\draw (12,-4) node[above]{$\tilde{\rho}_1$} ;

\draw[dotted, line width = 1pt] (0,-4) -- (2,-4) ;
\draw[dotted, line width = 1pt] (11,-4) -- (14,-4) ;

\fill [color = black] (0,-4)  circle (0.1)  ;
\fill [color = black] (14,-4)  circle (0.1)  ;

\draw (0,-3.9) node[left]{$\rho_0$} ; 
\draw (14,-3.9) node[right]{$\rho_1$} ;


\fill [color = black] (0,-6)  circle (0.1)  ;
\fill [color = black] (14,-6)  circle (0.1)  ;
\foreach \k in {1,2,...,27}
	{\fill [color = black] ({\k*0.5},-6)  circle (0.05)  ;}

\draw (0,-6) node[left]{$S_{\X_\sigma}(\rho_0)$} ; 
\draw (14,-6) node[right]{$S_{\X_\sigma}(\rho_1)$} ;

\fill [color = black] (2.1,-5.9)  rectangle (1.9, -6.1)  ;
\fill [color = black] (12.1,-5.9)  rectangle (11.9, -6.1)  ;

\draw (2,-6.1) node[below]{$S_{\X_\sigma}(\tilde{\rho}_0)$} ; 
\draw (12,-6.1) node[below]{$S_{\X_\sigma}(\tilde{\rho}_1)$} ;

\draw [->, line width = 1pt] (0.3, -4.3) -- (0.3,-5.7) ;
\draw (0.4,-5) node[right]{\makecell[l]{\small{Controllability}\\~~\small{(Proposition \ref{prop_controllability})}}} ;

\draw [->, line width = 1pt] (12.3, -4.3) -- (12.3,-5.7) ;
\draw (12.4,-5) node[right]{\makecell[l]{\small{Controllability}\\~~\small{(Proposition \ref{prop_controllability})}}} ;


\draw [->, line width = 1pt] (5, -4.3) -- (5,-5.7) ;
\draw (5.1,-5) node[right]{\makecell[l]{\small{Sampling}\\~~\small{(using Assumptions \ref{asmp_samp_der} and \ref{asmp_samp_actio})}}} ;

\end{tikzpicture}
\caption{Outline of the proof of Theorem \ref{theo_GammaLimsup}. Given a pair $(\rho, \mbf)$ we regularize it, then squeeze it into a shorter time interval, and we use the controllability property to adjust the temporal boundary conditions while sampling in the interior is easy thanks to \ref{asmp_samp_der} and \ref{asmp_samp_actio} as everything is regular.}
\label{figure_sketch_proof}
\end{center}
\end{figure}
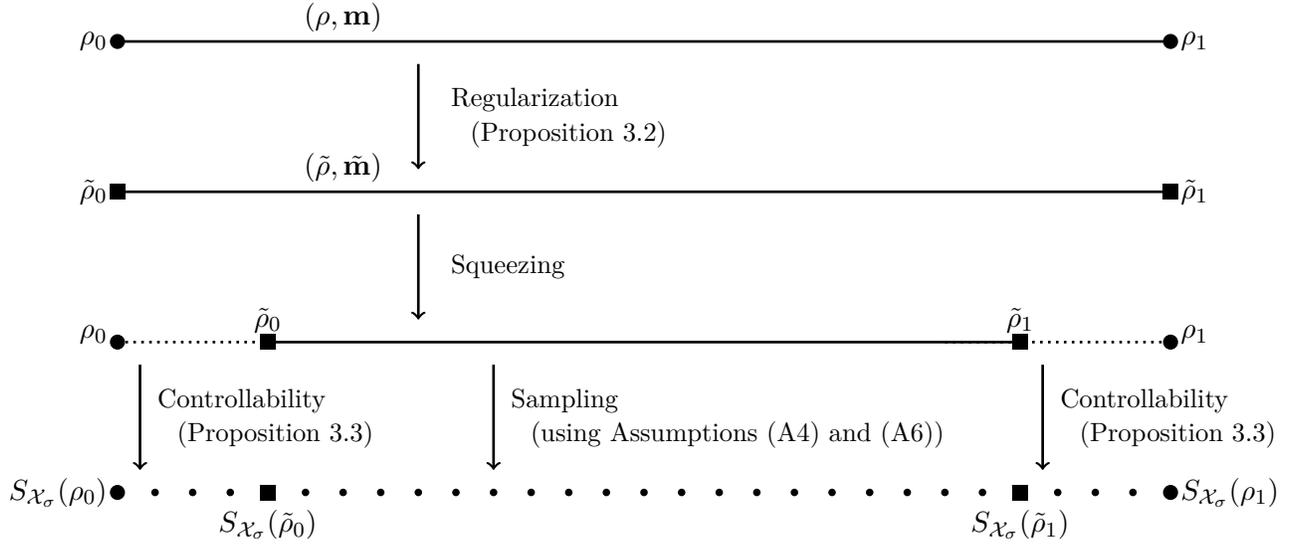

\emph{Proof of Theorem \ref{theo_GammaLimsup}}. We have now all the tools to prove the desired result. The idea is simple: we regularize the curve, use controllability to adjust the temporal endpoints, while in the interior we can sample quite easily as we have something regular. Figure \ref{figure_sketch_proof} represents graphically how we proceed. More specifically, we take $(\rho, \mbf) \in \M([0,1] \times X) \times \M([0,1] \times TX)$ a given pair such that $\J_{\rho_0, \rho_1}(\rho, \mbf) < + \infty$. Let us fix $T > 0$ and $\eta > 0$. By simplicity, we will assume that $N$ is always chosen such that $T$ is a multiple of $\tau = 1/N$. 

Let $(\tilde{\rho}, \tilde{\mbf})$ the regularized curved given by Proposition \ref{prop_regularization}. \review{In particular, $\tilde{\mbf}$ satisfies the no-flux boundary conditions.} On the other hand, let $(\hat{P}^{N, \sigma}, \hat{\Mbf}^{N, \sigma}) \in \X_\sigma^{T/\tau +1} \times \Y_\sigma^{T/\tau}$ the curve joining $S_{\X_\sigma}(\rho_0)$ onto $S_{\X_\sigma}(\tilde{\rho}_0)$ in $T/ \tau$ time steps with a controlled cost given by Proposition \ref{prop_controllability}. Similarly with $(\bar{P}^{N, \sigma}, \bar{\Mbf}^{N, \sigma}) \in \X_\sigma^{T/\tau +1} \times \Y_\sigma^{T/\tau}$ joining $S_{\X_\sigma}(\tilde{\rho}_1)$ onto $S_{\X_\sigma}(\rho_1)$ in $T/ \tau$ time steps with a controlled cost. 

\review{For $k$ such that $k \tau \in [T, 1-T]$ we define the affine rescaling $t_k = T + (k \tau - T)/(1-2T)$ such that $t_{T/\tau} = 0$ and $t_{(1-T)/\tau} = 1$.}

Eventually, we set, for $N, \sigma$ and $k \in \{ 0,1, \ldots, N \}$ given 
\begin{equation*}
P^{N, \sigma}_k := \begin{cases}
\hat{P}^{N, \sigma}_k & \text{if } k \tau \in [0,T] \\
\dst{S_{\X_\sigma} \left( \tilde{\rho}\left(  \review{t_k} , \cdot \right)  \right)} & \text{if } k \tau \in [T, 1-T] \\
\bar{P}^{N, \sigma}_{k - (1-T)/\tau} & \text{if } k \tau \in [1-T,T].
\end{cases}
\end{equation*} 
We chose the momentum accordingly: if $k \in \{ 1,2, \ldots, N \}$, 
\begin{equation*}
\Mbf^{N, \sigma}_k := \begin{cases}
T^{-1}\hat{\Mbf}^{N, \sigma}_k & \text{if } (k+1/2) \tau \in [0,T] \\
\review{ \dst{ \tau^{-1} S_{\Y_\sigma} \left(  \int_{t_{k-1}}^{t_k}  \tilde{\mbf}\left( t , \cdot \right) \, \ddr t  \right)} } & \text{if } (k + 1/2) \tau \in [T, 1-T] \\
T^{-1} \bar{\Mbf}^{N, \sigma}_{k  - (1-T)/\tau}  & \text{if } (k+1/2) \tau \in [1-T,T].
\end{cases}
\end{equation*} 
Notice that we have inserted a factor $1/T$ in the momentum $\hat{\Mbf}^{N, \sigma}$ and $\bar{\Mbf}^{N, \sigma}$ because of the difference of temporal scaling.
With our choice, $(P^{N, \sigma}, \Mbf^{N, \sigma})$ satisfies the discrete continuity equation. Indeed, it is clear if $k \tau \notin [T, 1-T]$  as $\hat{\Mbf}^{N, \sigma}$, $\bar{\Mbf}^{N, \sigma}$ are chosen for that. On the other hand, for $k$ such that \review{$(k-1) \tau$ and $k\tau$ are in $[T, 1-T]$},
\begin{multline*}
P^{N, \sigma}_{k} - P^{N, \sigma}_{k-1} = S_{\X_\sigma} \left( \int_{t_{k-1}}^{t_k} \dr_t \tilde{\rho}(t, \cdot) \, \ddr t \right) = S_{\X_\sigma} \left(-  \int_{t_{k-1}}^{t_k} \nabla \cdot \tilde{\mbf}(t, \cdot) \, \ddr t \right) \\
= - (\Div_\sigma \circ S_{\Y_\sigma}) \left( \int_{t_{k-1}}^{t_k} \tilde{\mbf}(t, \cdot) \, \ddr t \right) =  - \tau \Div_\sigma (\Mbf^{N, \sigma}_k),
\end{multline*}  
where the third identity comes from \ref{asmp_samp_der} and the last one is \review{the definition of $\Mbf^{N, \sigma}_k$}.

It remains to estimate the action of $(P^{N, \sigma}, \Mbf^{N, \sigma})$. Given the way we have built it (see Proposition \ref{prop_controllability}), and taking in account the change in the temporal scaling, the contribution of the action for $k \tau \notin [T, 1-T]$ does not exceed $T^{-1}(\hat{\omega}(\eta) + \varepsilon_{N, \sigma} \| \rho_0 \|)$ (recall that $W_2(\rho_0, \tilde{\rho}(0, \cdot)) \leqslant \eta$ and similarly for the final value). On the other hand, for the part on $[T, 1-T]$ we will use the consistency property \ref{asmp_samp_actio}. As $\tilde{\rho}$ and $\tilde{\mbf}$ are smooth (they are $C^1$ and $\tilde{\rho}$ is uniformly bounded from below) we can write, up to an error $\varepsilon_\sigma$ which depends only on $\sigma$ and which tends to $0$ as $\sigma \to 0$, that for all $N, \sigma, k$, 
\begin{equation*}
 A_\sigma \left( \frac{P^{N, \sigma}_{k-1} + P^{N, \sigma}_{k}}{2}, \Mbf^{N, \sigma}_k \right) \leqslant A \left( \frac{\tilde{\rho}(t_{k-1}, \cdot) + \tilde{\rho}(t_k, \cdot)}{2}, \frac{1}{\tau} \int_{t_{k-1}}^{t_k} \tilde{\mbf}(t, \cdot) \, \ddr t \right)  + \varepsilon_\sigma
\end{equation*} 
On the other hand, as the action $A$ depends smoothly of its input when restricted to the set of measures which have a smooth density, and as $\tilde{\rho}(t, \cdot)$ and $\tilde{\mbf}(t, \cdot)$ depend smoothly on the temporal variable, 
\begin{equation*}
\left| A \left( \frac{\tilde{\rho}(t_{k-1}, \cdot) + \tilde{\rho}(t_k, \cdot)}{2}, \frac{1}{\tau} \int_{t_{k-1}}^{t_k} \tilde{\mbf}(t, \cdot) \, \ddr t \right)  - \frac{1}{\tau(1-2T)} \int_{t_{k-1}}^{t_k} A(\tilde{\rho}(t, \cdot), \tilde{\mbf}(t, \cdot)) \, \ddr t  \right| \leqslant \frac{C}{N^2},
\end{equation*}
where $C$ depends on $\tilde{\rho}$ and $\tilde{\mbf}$, but not on $N$. \review{Notice that we have used $t_{k} - t_{k-1} = \tau/(1 - 2T)$, but also that $A$ is $2$-homogeneous in its second variable.} Hence, summing over $k$, one can see that 
\begin{equation*}
\sum_{k \ : \ k \tau \in [T, 1-2T]} \tau A_\sigma \left( \frac{P^{N, \sigma}_{k-1} + P^{N, \sigma}_{k}}{2}, \Mbf^{N, \sigma}_k \right) \leqslant \frac{1}{1-2T} \A(\tilde{\rho}, \tilde{\mbf}) + \varepsilon_\sigma + \frac{C}{N}.
\end{equation*}
Now, given that $\A(\tilde{\rho}, \tilde{\mbf}) \leqslant (1 + \eta)\A(\rho, \mbf)$ by construction, putting all these information together,
\begin{equation*}
\limsup_{N \to + \infty, \sigma \to 0} \ \J^{N, \sigma}_{S_{\X_\sigma}(\rho_0), S_{\X_\sigma}(\rho_1)}(P^{N, \sigma}, \Mbf^{N, \sigma}) \leqslant \frac{2}{T} \hat{\omega}(\eta) + \frac{1+ \eta}{1- 2T} \A(\rho, \mbf).   
\end{equation*}
We recall that $\eta$ and $T$ are arbitrary. If we choose first $T$ very small, and then $\eta$ small enough, we can see that we reach the desired conclusion. \qed

\section{On the relation with already proposed discretizations}
\label{section_examples}

Now that we have proved Theorems \ref{theo_GammaLiminf} and \ref{theo_GammaLimsup}, to justify the interest of our result, we will show how previous works can be embedded in this framework. This is the case for the the discretization on triangulations of surfaces used by the present author in \cite{Lavenant2018} with Claici, Chien and Solomon, and our general framework was mainly designed for this case. Moreover, we show that it can also be used to analyze the finite volume discretization proposed by Gladbach, Kopfer and Maas \cite{Gladbach2018}.

\subsection{Triangulations of surfaces}
\label{subsection_finite_elements}

\begin{figure}
\begin{center}
\includegraphics[width = 0.7 \textwidth]{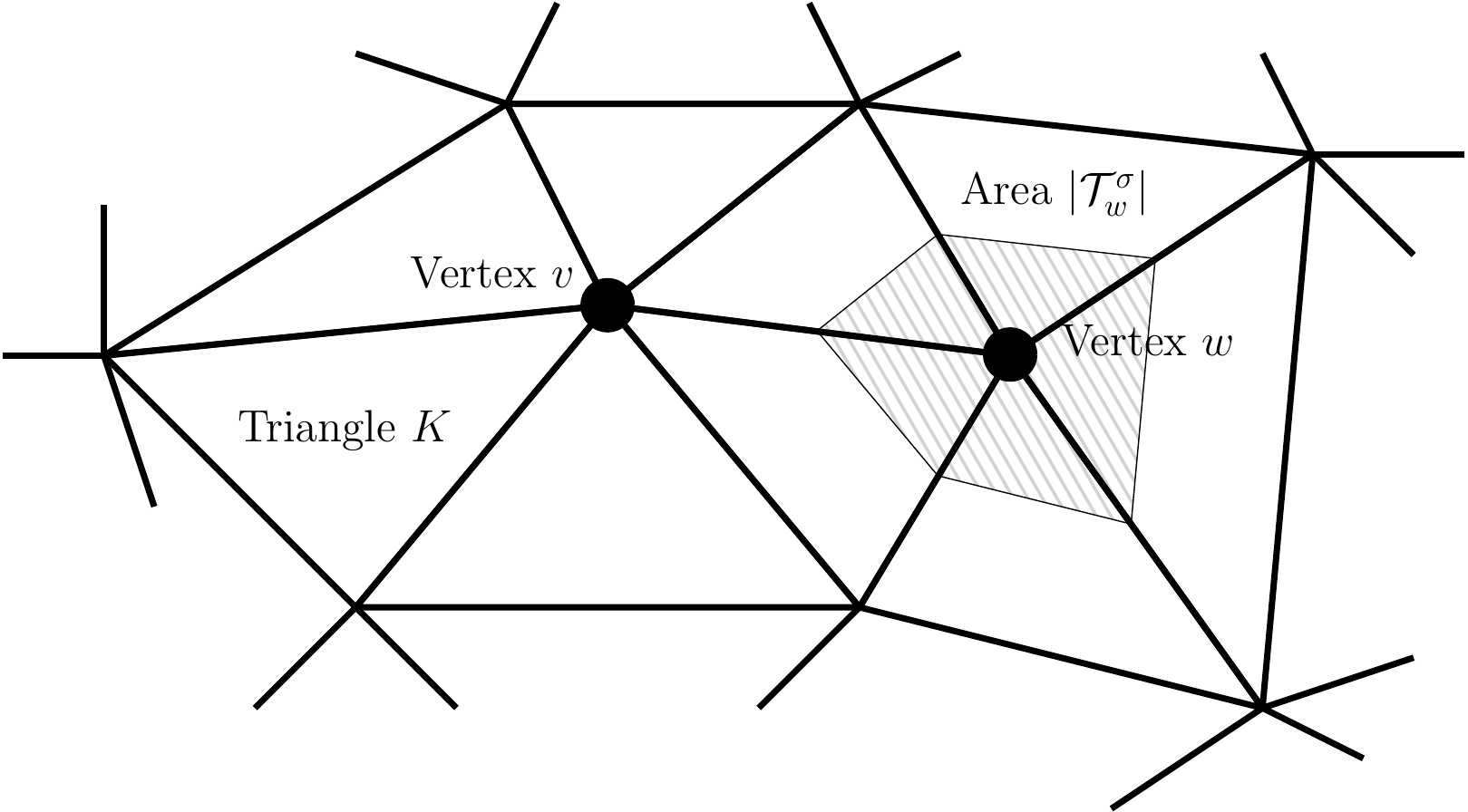}
\end{center}
\caption{Notations used for the triangulations. A triangulation is made of triangles which only intersect at their edges or their vertices. The ``area'' of a vertex $w$, denoted by $|\T^\sigma_w|$ is one third of the the area of all triangles to which $w$ belong: it is the area of the shaded region on the figure. Roughly speaking, the density is attached to vertices while the momentum is attached to triangles.}
\label{figure_finite_elements}
\end{figure}

\begin{figure}
\begin{center}
\includegraphics[width = 0.7 \textwidth]{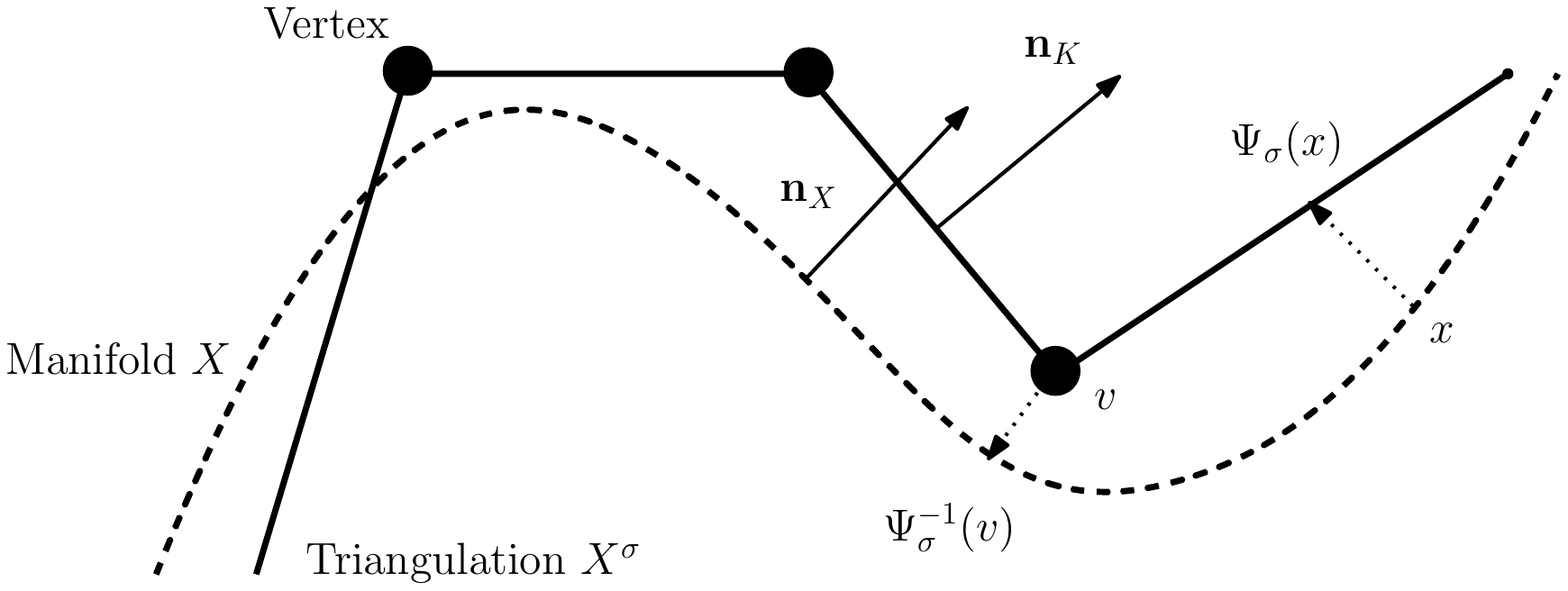}
\end{center}
\caption{\review{Two dimensional view of the link between the triangulation $X^\sigma$ and the manifold $X$. To each triangle is associated a normal $\nbf_K$ while $\nbf_X : X \to \R^3$ is the normal to the manifold. The map $\Psi_\sigma : X \to X^\sigma$ is the inverse of the nearest neighbor projection} \review{$\Psi_\sigma^{-1} : X^\sigma \to X$ onto $X$. This construction is the same as in \cite{Hildebrandt2006}.}}
\label{figure_projection}
\end{figure}

In this subsection, we want to show how our method applies to the discretization that we proposed in \cite{Lavenant2018}. We stronly encourgage the reader to take a look at the figures of this article, as we have numerically solved the fully discretized problem. The underlying space $X$ is now a smooth surface in $\R^3$ and we assume that we have a triangulation that approximates it in the $C^1$-sense, see \cite{Hildebrandt2006} and some explanations below. We use a discretization that is reminiscent of finite elements as velocity fields are defined over the triangles. The triangulation will be assumed to be \emph{regular} but not necessarily \emph{uniformly regular} (see below). In particular, it could allow one to try a multiscale discretization, with a refinement of the mesh only where needed. Moreover, compared to the finite volume discretization described in Subsection \ref{subsection_finite_volumes} below, no isotropy condition is required. 

\begin{rmk}
Before diving into the details of the spatial discretization, let us do a remark (that can be skipped at first reading) about the temporal discretization used in \cite{Lavenant2018} which differs from the one of the present article. Indeed, in \cite{Lavenant2018} we rather focused on discretizing the dual problem, hence the resulting functional (see \cite[Equation (16)]{Lavenant2018}) is slightly different, but can be expressed as follows by a duality argument. With the notation of the present article, if $\bar{P}_0, \bar{P}_1 \in \X_{\sigma,+}$ are given then we define $\tilde{\J}^{N, \sigma}_{\bar{P}_0, \bar{P}_1}$ on $(\X_\sigma)^N \times (\Y_\sigma)^{N+1}$ by 
\begin{equation*}
\tilde{\J}^{N, \sigma}_{\bar{P}_0, \bar{P}_1}( (P_k)_{1 \leqslant k \leqslant N}, (\Mbf_k)_{0 \leqslant k \leqslant N}  ) = \frac{\tau}{2} A_\sigma( P_1, \Mbf_0  ) + \frac{\tau}{2} A_\sigma(P_{N}, \Mbf_N) + \tau \sum_{k=1}^{N-1} A_\sigma \left( \frac{P_{k} + P_{k+1}}{2}, \Mbf_k \right)
\end{equation*}
if $(P_k)_{1 \leqslant k \leqslant N} \in (\X_{\sigma,+})^N$ and the discrete continuity equation is satisfied, that is
\begin{equation*}
\begin{cases}
\tau^{-1} (P_{k+1} - P_{k})  + \Div_\sigma (\Mbf_{k}) = 0, & \forall k \in \{1,2, \ldots, N-1 \} \\
(\tau/2)^{-1} (P_{1} - \bar{P}_0) + \Div_\sigma( \Mbf_0 ) = 0, & \\
(\tau/2)^{-1} (\bar{P}_1 - P_{N}) + \Div_\sigma( \Mbf_N ) = 0 , & \\
\end{cases}
\end{equation*}
and $+ \infty$ otherwise. As the reader can see, the difference appears near the temporal boundary conditions (in some sense the first and last temporal step sizes are $\tau/2$ and not $\tau$), but the techniques of this article still work if one replaces $\J^{N, \sigma}$ by $\tilde{\J}^{N, \sigma}$. 
\end{rmk}

Let us go through the spatial discretization. We take $X \subset \R^3$ a smooth compact $2$-dimensional submanifold of $\R^3$ without boundary. We approximate $X$ with a triangulation. Namely, we assume that for each $\sigma > 0$ small enough we have a pair $(\T^\sigma, \V^\sigma)$ where $\V^\sigma \subset \R^3$ is a finite subset of $\R^3$ (the vertices), and $\T^\sigma$ is a set of triangles in $\R^3$. 

Each $K \in \T^\sigma$ is a triangle in $\R^3$ whose vertices belong to $\V^\sigma$. We denote $X^\sigma = \bigcup_{K \in \T^\sigma} K$ the polyhedral surface that it generates. Notice that both $X^\sigma$ and $X$ are subsets of $\R^3$, but their intersection may be empty. We assume that $\sigma$ is the maximum diameter of elements in $\T^\sigma$. Following \cite[Definition (4.4.13)]{Brenner2007}, we assume that the family of triangulations is regular: it means that there exists a constant $c > 0$ (independent on $\sigma$) such that for every $\sigma > 0$ and every $K \in \T^\sigma$, one can fit a disk of radius $c ~ \mathrm{diam}(K)$ in $K$. It implies that all angles in the triangles are uniformly bounded from below independently on $\sigma$; that for every $\sigma > 0$ and every $K \in \T^\sigma$, the number of triangles sharing a vertex with $K$ is bounded independently on $\sigma$, and all these neighboring triangles have a diameter comparable to the one of $K$. However, contrary to the next section we do not assume that the triangulation is \emph{uniformly regular} (that is, quasi regular in the sense of \cite[Definition (4.4.13)]{Brenner2007}) which means that the minimal diameter of triangles in $\T^\sigma$ can be much smaller than $\sigma$. 

Following \cite{Hildebrandt2006}, we assume that the family of triangulations approximate the manifold $X$ in the $C^1$ sense, \review{see Figure \ref{figure_projection} for a geometric picture}. Specifically $X^\sigma$ is assumed to be a graph over $X$: it means that $X^\sigma$ always lie in the set of points in $\R^3$ for which the projection onto $X$ is smooth and uniquely defined, and that the projection onto $X$ is one to one once restricted to $X^\sigma$. In particular we can define $\Psi_\sigma : X \to X^\sigma$ as the inverse of the projection map $\Psi^{-1}_\sigma : X^\sigma \to X$, and both $\Psi_\sigma$ and $\Psi_\sigma^{-1}$ are continuous and one-to-one. The differential of $\Psi_\sigma$ at the point $x \in X$ is denoted by $D \Psi_\sigma(x) : T_x X \to \nbf_K^\perp \subset \R^3$, where $K$ is the triangle in which $\Psi_\sigma(x)$ lies and $\nbf_K$ is the normal to $K$. Notice that $D \Psi_\sigma$ is well defined up to an negligible set on $X$ (the image by $\Psi_\sigma^{-1}$ of the edges of the triangles of $\T^\sigma$). We assume that $X^\sigma$ converges in Hausdorff distance to $X$, that is $0$-th order convergence. We also assume that the convergence holds at first order in the following sense: \review{if we denote by $\nbf_X : X \to \R^3$ the normal mapping to the manifold $X$ (not to be confused with $\nbf_{\dr X}$ the normal to $\dr X$, and the latter is empty in this case), and we also denote by $\nbf_K$ the normal to the triangle $K$, and we assume} 
\begin{equation*}
\lim_{\sigma \to 0} \left( \sup_{K \in \T^\sigma} \sup_{x \in K} | \mathbf{n}_K - \mathbf{n}_X( \Psi^{-1}_\sigma(x)  )  | \right) = 0.
\end{equation*}  
We refer to \cite[Theorem 2]{Hildebrandt2006} for alternative (but equivalent) formulations of such a convergence.

Our result will be already interesting if $X$ were flat: the reader not so familiar with differential geometry can think as $\Psi_\sigma$ being the identity mapping to translate the results in the flat case.

\bigskip

Let us set some additional notations first, see Figure \ref{figure_finite_elements} for an explanation about some of them. The area (i.e. the $2$-dimensional Haussdorff measure) of a triangle $K \in \T^\sigma$ is denoted by $|K|$. The set of vertices of a triangle $K$ is $\V^\sigma_K \subset \V^\sigma$, while the set of triangles to which a vertex $v$ belongs is $\T^\sigma_v \subset \T^\sigma$. If $v \in \V^\sigma$, we denote by $\hat{\phi}_v$ the function defined on $X_\sigma$, piecewise linear on each triangle, such that $\hat{\phi}_v(v) = 1$ and $\hat{\phi}_v$ is $0$ everywhere else on $\V^\sigma$. In particular, $\hat{\phi}_v$ has a gradient constant over each triangle, and non zero only on $\T^\sigma_v$. We define $|\T^\sigma_v| = \frac{1}{3} \sum_{K \in \T^\sigma_v} |K|$ the ``area'' of the vertex $v$. It can also be expressed by $|\T^\sigma_v| = \int_{X^\sigma} \hat{\phi}_v$, where integration is performed w.r.t. the $2$-dimensional Haussdorff measure.  

We define our finite dimensional model of dynamical optimal transport as follows. 
\begin{itemize}
\item[•] $\X_\sigma = \R^{\V^\sigma}$ and $\X_{\sigma,+} = (\R_+)^{\V^\sigma}$ while 
\begin{equation*}
\Y_\sigma = \left\{ \Mbf = (\Mbf_K)_{K \in \T^\sigma} \in (\R^3)^{\T^\sigma} \ : \ \forall K \in \T^\sigma, \ \Mbf_K \cdot \mathbf{n}_K = 0  \right\}. 
\end{equation*}
That is, density is defined by one scalar per vertex, while momentum is one vector by triangle, lying in the vector space parallel to the triangle.  
\item[•] The divergence operator, for $\Mbf \in \Y_\sigma$ and $v \in \V^\sigma$, is given by 
\begin{equation*}
(\Div_\sigma \Mbf)_v = - \frac{1}{|\T^\sigma_v|}\sum_{K \in \T^\sigma_v} |K| \left( \left. \nabla \hat{\phi}_v \right|_{K} \cdot \Mbf_K \right).
\end{equation*}
This divergence is, for proper scalar products, the dual operator to the gradient restricted to the set of piecewise affine functions. 
\item[•] The action $A_\sigma$ is defined for $(P, \Mbf) \in \X_\sigma \times \Y_\sigma$ as 
\begin{equation*}
A_\sigma(P, \Mbf) = \sum_{K \in \T^\sigma} \frac{1}{2} \frac{ |\Mbf_K|^2  }{ (\sum_{v \in \V^\sigma_K} P_v)/3} |K|.
\end{equation*}
The discrete action $A_\sigma$ easily satisfies the homogeneity and monotonicity requirements of Definition \ref{definition_approximation}. In the formula above, we follow the convention of \eqref{equation_convention_b22a} to define the quotient: such a choice makes $A_\sigma$ convex and lower semi-continuous. 
\end{itemize}
Up to now, we have copied the definitions of \cite{Lavenant2018}. Both the values of the density and the momentum on vertices or triangles are thought as \emph{intensive}, that is they are normalized by the volume of the vertex or the triangle they live on. Of course, all of these definitions do not require the knowledge of $X$, only the one of the triangulation. Let us now switch to reconstruction and sampling. 

\begin{itemize}
\item[•] To sample, we assume that $\mbf \in \M(TX)$ is absolutely continuous w.r.t. volume measure, and $\rho \in \M(X)$ is arbitrary. We define 
\begin{equation*}
(S_{\X_\sigma} (\rho))_v = \frac{1}{|\T^\sigma_v|} \langle  \rho, \hat{\phi}_v \circ \Psi_\sigma \rangle  \ \text{ and } \ (S_{\Y_\sigma} (\mbf))_K = \frac{1}{|K|} \int_{ \Psi_\sigma^{-1}(K)} D \Psi_\sigma(x) \mbf(x) \, \ddr x .
\end{equation*}
Notice that $S_{\X_\sigma} (\rho)$ is some sort of projection onto the set of piecewise affine functions on $X_\sigma$ of  the image measure of $\rho$ by $\Psi_\sigma$. 
\item[•] There are two ways to reconstruct the density, and one way for the momentum. Roughly, for $R^{\CE}_{\X_\sigma}$ we put a Dirac mass at every vertex, while for $R^A_{\X_\sigma}$ we take a density constant over each triangle, equal to the mean of the densities of the vertices of the triangle. On the other hand, for $R_{\Y_\sigma}$ we take a momentum constant over each triangle. Then, we pull everything back onto $X$ via the mapping $\Psi_\sigma$. Specifically, if $\phi \in C(X)$ and $\bbf \in C(X,TX)$ are test functions defined over $X$,
\begin{multline*}
\langle R^A_{\X_\sigma} (P), \phi \rangle = \sum_{K \in \T^\sigma}  \left( \frac{1}{3} \sum_{v \in \V^\sigma_K} P_v \right) \int_{\Psi_\sigma^{-1}(K)} \phi(x) \,\ddr x,
 \ \  \ \langle \ R^{\CE}_{\X_\sigma} (P), \phi \rangle = \sum_{v \in \V^\sigma} P_v |\T^\sigma_v| \phi( \Psi_\sigma^{-1}(v) ), \\
\text{ and } \ \langle R_{\Y_\sigma} (\Mbf), \bbf \rangle =  \sum_{K \in \T^\sigma} \Mbf_K \cdot \left( \int_{\Psi_\sigma^{-1}(K)}  D \Psi_\sigma(x)  \bbf(x) \, \ddr x \right).  
\end{multline*} 
\end{itemize} 

\noindent It is really straightforward to check that $R^A_{\X_\sigma}$ and $R^{\CE}_{\X_\sigma}$ map $\X_{\sigma,+}$ into $\M_+(X)$ while $S_{\X_\sigma}(\M_+(X)) \subset \X_{\sigma,+}$. 


\begin{prop}
\label{proposition_finite_elements_adapated}
Let $X \subset \R^3$ be a smooth compact $2$-dimensional submanifold of $\R^3$ without boundary. For each $\sigma > 0$, assume that we have $(\T^\sigma, \V^\sigma)_\sigma$ a triangulation in $\R^3$, and that this family of triangulations is regular and converges as $\sigma \to 0$ to $X$ in the $C^1$ sense defined in \cite{Hildebrandt2006}.  

Then, provided that $(\X_\sigma, \Y_\sigma, A_\sigma, \Div_\sigma)_\sigma$ are defined as above, we get a family of finite dimensional models of dynamical optimal transport which is adapted to $X$ (in the sense of Definition \ref{definition_adapated}), with reconstruction and sampling operators  $(R^{\CE}_{\X_\sigma}, R^A_{\X_\sigma}, R_{\Y_\sigma}, S_{\X_\sigma}, S_{\Y_\sigma})_\sigma$ defined as above.   
\end{prop}

\begin{rmk}
The reconstruction operators still require the knowledge of $X$, which could be disappointing. However, we can define reconstruction operators valued in the set of measures over $\R^3$ (namely, concentrated on the polyhedral surface $X^\sigma$): densities and momenta are taken for instance taken constant on each triangle. Using the notations of Theorem \ref{theo_main}, in the limit $N \to + \infty, \sigma \to 0$ it is easy to see (because of the $C^1$ convergence of the surface) that the reconstructed densities and momenta on the surfaces $X_\sigma$ will converge to the same limit as $\Rc^A_{N,\X_\sigma}(P^{N, \sigma})$ and $\Rc_{N,\Y_\sigma}(\Mbf^{N, \sigma})$, provided the latter exist. As Theorem \ref{theo_main} combined with Proposition \ref{proposition_finite_elements_adapated} justify that these objects have a limit, so do the reconstructions which are not pulled back on $X$.  
\end{rmk}

\begin{proof}
Before doing computations, let us define some distortion coefficients for areas: if $\sigma > 0$ and $v \in \V^\sigma$ while $K \in \T^\sigma$,
\begin{equation*}
\alpha^\sigma_v := \frac{\int_{X} \hat{\phi}_v (\Psi_\sigma(x)) \, \ddr x  }{ |\T^\sigma_v|} - 1 \ \text{ and } \beta^\sigma_K = \frac{\int_{\Psi_\sigma^{-1}(K)} \, \ddr x}{|K|} - 1.
\end{equation*} 
Thanks to the $C^1$ convergence of the surface which implies convergence of areas \cite[Theorem 2]{Hildebrandt2006}, $\sup_{v \in \V^\sigma} |\alpha^\sigma_v|$ tends to $0$ as $\sigma \to 0$, as well as $\sup_{K \in \T^\sigma} |\beta^\sigma_K|$. Analogously, we define the distortion coefficients for the metric by, for $\sigma > 0$ and $K \in \T^\sigma$, 
\begin{equation*}
\theta^\sigma_K = \sup_{ x \in \Psi^{-1}_\sigma(K) }\sup_{w \in T_x X, w \neq 0} \left|  \frac{|D\Psi_\sigma(x) w|}{|w|} - 1 \right|,
\end{equation*}  
and again still by the $C^1$ convergence of the surface $\sup_K \theta^\sigma_K$ tends to $0$ as $\sigma \to 0$.

We first check \ref{asmp_samp_der}, that is the exact commutation between derivation and sampling. If $\mbf \in \M(TX)$ has a smooth density then for $v \in \V^\sigma$, 
\begin{align*}
(S_{\X_\sigma} (\nabla \cdot \mbf) )_v & = \frac{1}{|\T^\sigma_v|} \int_{X} \hat{\phi}_v(\Psi_\sigma(x)) (\nabla \cdot \mbf)(x) \, \ddr x \\
&= - \frac{1}{|\T^\sigma_v|}  \int_X \nabla \left( \hat{\phi}_v \circ \Psi_\sigma \right) \cdot \mbf(x) \, \ddr x \\
& = - \sum_{K \in \T^\sigma_v} \frac{1}{|\T^\sigma_v|} \int_{\Psi_\sigma^{-1}(K)} \left( (D \Psi_\sigma)^\top(x) \left. \nabla \hat{\phi}_v \right|_K \right) \cdot \mbf(x) \,  \ddr x \\ 
& = - \frac{1}{|\T^\sigma_v|} \sum_{K \in \T^\sigma_v} |K| \left( \left. \nabla \hat{\phi}_v \right|_K \right) \cdot \frac{1}{|K|} \int_{\Psi_\sigma^{-1}(K)}  \left( D \Psi_\sigma (x) \mbf(x) \right) \, \ddr x 
= (( \Div_\sigma \circ S_{\Y_\sigma} ) (\mbf))_v. 
\end{align*} 
Next we check \ref{asmp_samp_actio}, which amounts to say that we use a legit quadrature formula for the action. Let $\Bc$ and $\Bc'$ be bounded sets of $C^1(X)$ and $C^1(X,TX)$ respectively, with functions in $\Bc$ uniformly bounded from below by a positive constant. Let $\rho \in \M(X)$ and $\mbf \in \M(TX)$ whose densities w.r.t. volume measure (still denoted by $\rho$ and $\mbf$) belong to $\Bc$ and $\Bc'$ respectively. Thanks to the boundedness of $\Bc$, and given the definition of the coefficients $\alpha^\sigma_v$, we see that $|(S_{\X_\sigma} (\rho))_v - \rho(\Psi_\sigma^{-1}(v))| \leqslant C ( \sigma + |\alpha^\sigma_v|)$ where $C$ depends only on $\Bc$. For sampling of the momentum, we have to use also the distortions coefficients $\theta^\sigma_K$ for the metric tensor to find $\left| |(S_{\Y_\sigma} (\mbf))_K| - |\mbf( \Psi_\sigma^{-1}(z) )| \right| \leqslant C ( \sigma + |\beta^\sigma_K| + \theta^\sigma_K  )$ provided $z$ is any point of $K$. Putting these estimates together and using that $\rho$ is bounded from below,     
\begin{equation*}
\left| \frac{ |\Mbf_K|^2  }{2 (\sum_{x \in \V^\sigma_K} P_x)/3} |K| - \frac{|\mbf(z)|^2}{2\rho(z)} |K| \right| \leqslant C |K| \varepsilon_\sigma
\end{equation*}  
where $z$ is any point of $\Psi_\sigma^{-1}(K)$ and $\varepsilon_\sigma$ depends on $\sigma, \alpha^{\sigma}_v$, $\beta^\sigma_K$ and $\theta^\sigma_K$ and tends to $0$ uniformly in $K$ as $\sigma \to 0$. On the other hand, as $\rho$ and $\mbf$ have uniformly smooth densities, if $z \in \Psi_\sigma^{-1}(K)$ 
\begin{equation*}
\left| \frac{|\mbf(z)|^2}{2\rho(z)} |K| - \int_{\Psi_\sigma^{-1}(K)} \frac{|\mbf(x)|^2}{2 \rho(x)} \,\ddr x  \right| \leqslant C |K| (\sigma + |\alpha^\sigma_K| ). 
\end{equation*} 
Putting these two estimates together and summing over $K$ we get \ref{asmp_samp_actio}. 

Now we turn to reconstruction. The one-sided estimate for the action is clear. Actually, thanks to Lemma \ref{lemma_interp_action}, we will check \ref{asmp_interp_action_stronger} rather than \ref{asmp_interp_action}: the reconstruction operator $R^A_{\X_\sigma}$ is built for that. Let us take $P \in \X_\sigma$ and $\Mbf \in \Y_\sigma$. For every $K \in \T^\sigma$, the measure $R^A_{\X_\sigma} (P)$ has a density constant over $\Psi_\sigma^{-1}(K)$ and equal to $\frac{1}{3} \sum_{v \in \V^\sigma_K} P_v$, while the measure $R_{\Y_\sigma} (\Mbf)$ has a density given by $(D \Psi_\sigma)^\top \Mbf_K$. As the linear operator $D \Psi_\sigma$ is almost an isometry (up to an error $\theta^\sigma_K$), $|(D \Psi_\sigma)^\top \Mbf_K|^2 \leqslant (1 + \theta^\sigma_K)^2 |\Mbf_K|^2$. Hence, 
\begin{align*}
A(R^A_{\X_\sigma} (P),R_{\Y_\sigma} (\Mbf) ) & = \sum_{K \in \T^\sigma} \int_{\Psi_\sigma^{-1}(K)} \frac{|(R_{\Y_\sigma} \Mbf) (x)|^2}{ 2 (R^A_{\X_\sigma} P)(x) } \,\ddr x \\
& \leqslant \sum_{K \in \T^\sigma} (1+ \theta^\sigma_K)^2   \frac{ |\Mbf_K|^2  }{2 (\sum_{v \in \V^\sigma_K} P_v)/3}  \int_{\Psi_\sigma^{-1}(K)} \,\ddr x \\ 
& \leqslant  \sum_{K \in \T^\sigma} (1+ \theta^\sigma_K)^2 (1+ |\beta^\sigma_K|)  \frac{ |\Mbf_K|^2  }{2 (\sum_{x \in \V^\sigma_K} P_x)/3} |K| \\
&\leqslant \left( \sup_{K \in \T^\sigma}  (1+ \theta^\sigma_K)^2 (1+ |\beta^\sigma_K|) \right) A_\sigma(P, \Mbf),
\end{align*}
and the factor in front of $ A_\sigma(P, \Mbf)$ in the last line tends to $1$ as $\sigma \to 0$. We emphasize that we have used the convention \eqref{equation_convention_b22a} in the computation above. On the other hand, to check \ref{asmp_interp_der} commutation between reconstruction and derivation, let us take $\Bc$ a bounded set of $C^2(X)$ and $\phi \in \Bc$ and compute: 
\begin{multline*}
\langle (R^{\CE}_{\X_\sigma} \circ \Div_\sigma)(P), \phi \rangle 
= - \sum_{v \in \V^\sigma} \left[ \sum_{K \in \T^\sigma_v} |K| \left( \left. \nabla \hat{\phi}_v \right|_{K} \cdot \Mbf_K \right) \right] \phi( \Psi_\sigma^{-1}(v) ) \\
= - \sum_{K \in \T^\sigma} |K| \Mbf_K \cdot \left( \sum_{v \in \V^\sigma_K} \phi(\Psi^{-1}_\sigma(v)) \left. \nabla \hat{\phi}_v \right|_K  \right).
\end{multline*} 
Now we notice that $\sum_{v \in \V^\sigma_K} \phi(\Psi^{-1}_\sigma(v)) \left. \nabla \hat{\phi}_v \right|_K$ is the gradient restricted to $K$ of the function $\hat{\phi} : X^\sigma \to \R$ which is continuous, piecewise affine on each triangle, and whose value on the vertex $v \in \V^\sigma$ coincide with $\phi( \Psi_\sigma^{-1}(v) )$. Thus, 
\begin{equation*}
\langle (R^{\CE}_{\X_\sigma} \circ \Div_\sigma)(P), \phi \rangle  = - \sum_{K \in \T^\sigma} |K| \Mbf_K \cdot \left. \nabla \hat{\phi} \right|_K.
\end{equation*} 
As $\hat{\phi}$ coincides with $\phi \circ \Psi^{-1}_\sigma : X_\sigma \to \R$ on $\V^\sigma$, thanks to the regularity of the mesh we know \cite[Corollary (4.4.24)]{Brenner2007} that at least on a triangle $K$, 
\begin{equation*}
\sup_{y \in K} \left| \left. \nabla \hat{\phi} \right|_K - \nabla \left( \phi \circ \Psi_\sigma^{-1} \right)(y) \right| \leqslant C \sigma,
\end{equation*}
where the constant $C$ depends on the second derivatives of $\phi \circ \Psi_\sigma^{-1}$ over $K$. As $K$ is flat, the second derivatives of $\Psi_\sigma^{-1}$ restricted to $K$ depend only on the derivatives of the metric tensor of the underlying space $X$. The latter is assumed to be smooth, while the derivatives of $\phi$ are bounded by a constant that depends only on $\Bc$. Hence the constant $C$ depends only on $X$ and $\Bc$, but does not depend on $\sigma$. On the other hand, by definition of $R_{\Y_\sigma}$, there holds 
\begin{equation*}
\langle R_{\Y_\sigma} (\Mbf), \nabla \phi \rangle =  \sum_{K \in \T^\sigma} \Mbf_K \cdot \int_{\Psi_\sigma^{-1}(K)}  D \Psi_\sigma(x) \nabla \phi(x) \,\ddr x.
\end{equation*}
Putting these two information together,  
\begin{multline*}
\left| \langle (R^{\CE}_{\X_\sigma} \circ \Div_\sigma)(P), \phi \rangle + \langle (R_{\Y_\sigma} \Mbf), \nabla \phi \rangle  \right| \\
\leqslant \left| \sum_{K \in \T^\sigma} \Mbf_K \cdot \left( - \int_K \nabla \left( \phi \circ \Psi_\sigma^{-1} \right)(y) \,\ddr y + \int_{\Psi_\sigma^{-1}(K)}  D \Psi_\sigma(x) \nabla \phi(x) \,\ddr x  \right) \right| 
+ C \sigma \left( \sum_{K \in \T^\sigma}  |K| |\Mbf_K| \right)  . 
\end{multline*}
By $C^1$ convergence of the triangulations, we can deduce that each term in front of $\Mbf_K$ is bounded by an error $|K| C \varepsilon_\sigma$, where $\varepsilon_\sigma$ depends on $\Psi_\sigma$ but tends to $0$ as $\sigma$ tends to $0$, while $C$ depends on $\Bc$. This is enough to yield \ref{asmp_interp_der}.

The first point \ref{asmp_interp_sampling} is very straightforward and left to the reader. On the other hand, for the second one \ref{asmp_interps}, that is to check that $R^A_{\X_\sigma}$ and $R^{\CE}_{\X_\sigma}$ are asymptotically equivalent, let us take $\Bc$ a bounded set of $C^1(X)$ and $\phi \in \Bc$, then
\begin{equation*}
\langle ( R^{\CE}_{\X_\sigma} - R^A_{\X_\sigma} )(P), \phi  \rangle = \sum_{v \in \V^\sigma} P_v \left( |\T^\sigma_v| \phi(\Psi_\sigma^{-1}(v)) - \sum_{K \in \T^\sigma_v} \frac{1}{3} \int_{\Psi_\sigma^{-1}(K)} \phi  \right).
\end{equation*}
As $\phi \in \Bc$, for each $v \in \V^\sigma$ the discrepancy between $\fint_{\Psi_\sigma^{-1}(K)} \phi$ and $\phi(\Psi_\sigma^{-1}(v))$ is smaller than $C \sigma$. Moreover, using the fact that the distortions coefficients $\alpha^\sigma_v$ are tending to $0$ when $\sigma \to 0$, we get the result. 

Eventually we need to prove controllability \ref{asmp_controllability}. The idea is to start with controllability at the continuous level: if $x,y \in X$, there exists $\gamma : [0,1] \to X$ a constant-speed joining them. Up to a small modification of the curve, we can assume that $\Psi_\sigma \circ \gamma$ intersects edges of the triangulation $\T^\sigma$ in a finite number of points while $|\dot{\gamma}(t)| \leqslant 2 d_g(x,y)$ for all $t \in [0,1]$. Let $\hat{\rho} \in \M([0,1] \times X)$ and $\hat{\mbf}_1, \hat{\mbf}_2 \in \M(TX)$ the continuous objects built according to \eqref{equation_controllability_continuous}. We set $\hat{P} = S_{\X_\sigma}(\hat{\rho})$ and $\hat{\Mbf}_i = S_{\Y_\sigma}(\hat{\mbf}_i)$ for $i \in \{1,2\}$. The reader can check, as $\Psi_\sigma \circ \gamma$ intersects finitely many edges, that $S_{\Y_\sigma}(\hat{\mbf}_i)$ is well defined and, by \ref{asmp_samp_der}, that $\Div_\sigma( \hat{\Mbf}_1 ) = \hat{P} - S_{\X_\sigma}(\delta_x)$, and $\Div_\sigma( \hat{\Mbf}_2 ) = \hat{P} - S_{\X_\sigma}(\delta_y)$. We need to estimate the action (and we cannot use \ref{asmp_samp_actio} because of a lack of smoothness), let us do it for $A_\sigma(\hat{P}, \hat{\Mbf}_1)$. Let $K$ be one of the triangle that $\Psi_\sigma \circ \gamma$ crosses and $\hat{\rho}_K = \hat{\rho}( \Psi_\sigma^{-1}(K))$. As the mesh is regular, all the triangles sharing one vertex with $K$ have an area comparable to the one of $K$, hence $|\T^\sigma_v|$ is also comparable to $|K|$ for $v$ any vertex of $K$. Thus we can estimate, up to a constant $C$ independent on $\sigma$ and changing from one equation to an other,  
\begin{equation*}
\frac{1}{3} \sum_{v \in \V^\sigma_K} \hat{P}_v \geqslant \frac{1}{C |K|} \sum_{v \in \V^\sigma_K} \langle \hat{\rho}, \hat{\phi}_v \circ \Psi_\sigma  \rangle 
\geqslant  \frac{1}{C |K|} \sum_{v \in \V^\sigma_K} \langle \hat{\rho}, \1_{K} \circ \Psi_\sigma \rangle \geqslant \frac{\hat{\rho}_K}{C|K|} 
\end{equation*}  
where we have first used that $\sum_{v \in \V^\sigma_K} \hat{\phi}_v \geqslant \1_K$ and then the definition of $\hat{\rho}_K$ as well as the low distortion of areas induced by $\Psi_\sigma$. On the other hand, recall that $\hat{\mbf}_1$ has a density bounded by $2 d_g(x,y)$ w.r.t. $\hat{\rho}$. Hence, still on this triangle $K$ and starting from the definition of $\hat{\mbf}_1$ and $S_{\Y_\sigma}$ one can easily estimate $|\hat{\Mbf}_1| \leqslant C \hat{\rho}_K d_g(x,y) / |K|$, where $C$ is independent on $\sigma$ and depends on the operator norm of $D \Psi_\sigma$. Putting these estimates together, 
\begin{equation*}
A_\sigma(\hat{P}, \hat{\Mbf}_1) \leqslant \sum_{K \in \V^\sigma} \frac{1}{2 ( \hat{\rho}_{K} / C|K| )  } \left( C \frac{\hat{\rho}_{K} d_g(x,y)}{|K|} \right)^2 |K| \leqslant C d_g(x,y)^2 \sum_{K \in \V^\sigma}  \hat{\rho}_{K} = C d_g(x,y)^2, 
\end{equation*}   
where the last identity comes from the definition of the $\hat{\rho}_K$ and the fact that $\hat{\rho}$ has unit mass.

\end{proof}

\subsection{Finite volumes}
\label{subsection_finite_volumes}

\begin{figure}
\begin{center}
\includegraphics[width = 0.7 \textwidth]{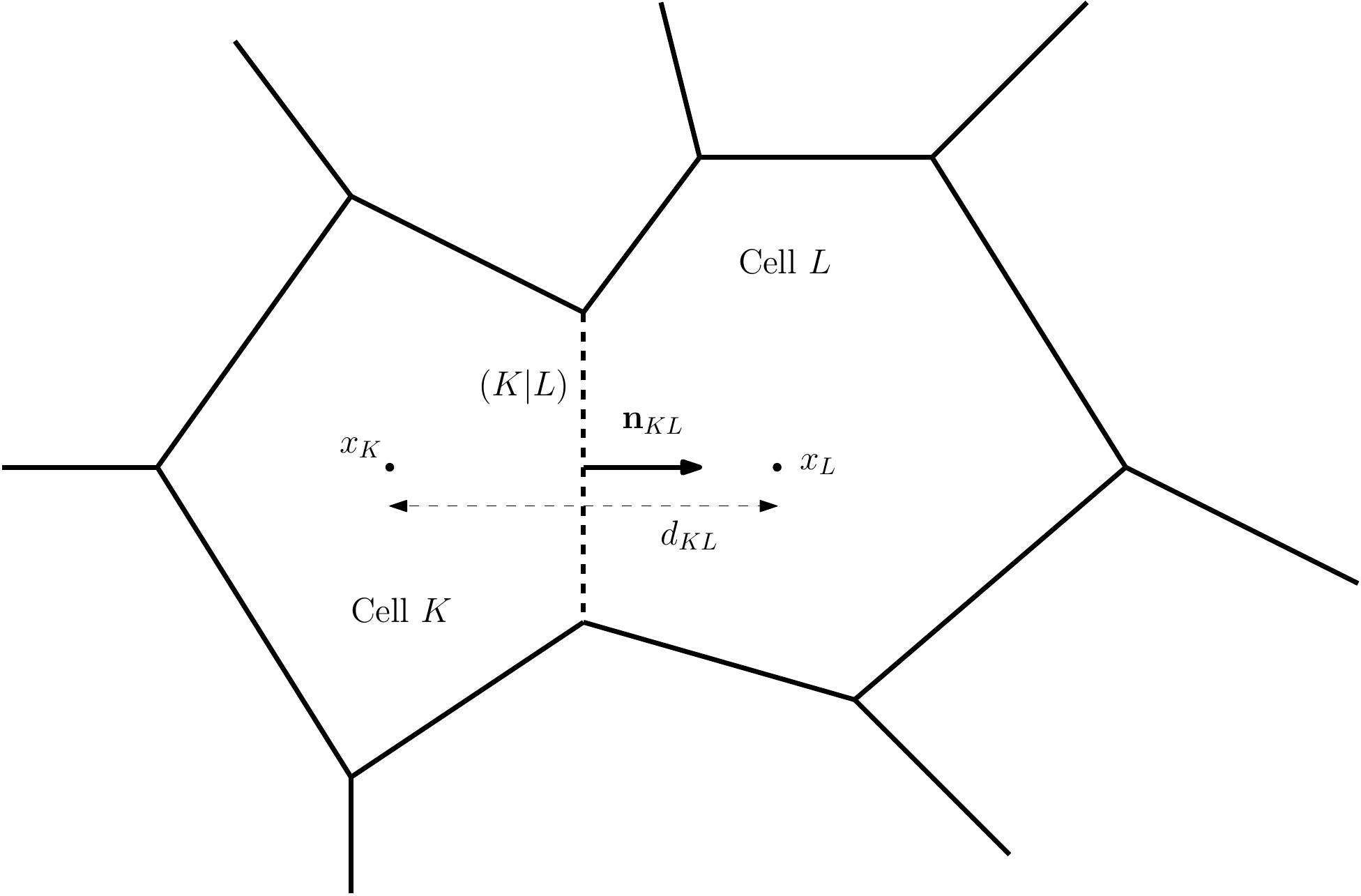}
\end{center}
\caption{Notations used for finite volumes. Two cells (i.e. convex polytopes) $K$ and $L$ on the left and right respectively. Each cell contains a specific point, here $x_K$ and $x_L$. The boundary $(K|L)$ between the two cells is the dashed line. The distance between the centers of cells $K$ and $L$ is $d_{KL}$, while $\nbf_{KL}$ is the unit vector orthogonal to $(K|L)$ oriented from $K$ to $L$. By assumption the mesh is admissible, that is the vector joining $x_K$ to $x_L$ orthogonal to $(K|L)$. Roughly speaking, the density is attached to the cells while the momentum is attached to the boundaries between cells.}
\label{figure_finite_volumes}
\end{figure}

This part focuses on the discretization proposed by Gladbach, Kopfer and Maas in \cite{Gladbach2018}. We strongly advise the reader to read the article and reference therein to get a better understanding of this discretization and the results known about it. We take $X = \Omega \subset \R^d$ a convex compact subset of $\R^d$ with smooth boundary. In particular $\M(TX) = \M(\Omega)^d$ is the space of $d$-dimensional vectorial measures. 

\medskip

For each $\sigma > 0$, we assume that we have a mesh $(\T^\sigma, (x_K)_{K \in \T^\sigma})$ that is: $\T^\sigma$ is a partition of $\Omega$ in convex sets (called cells) with non empty interior; and for each $K \in \T^\sigma$, we have a point $x_K$ located in the interior of the cell $K$. For the sake of clarity we will work in the simplest framework proposed by the aforementioned article: we will only consider symmetric means. As identified by \cite{Gladbach2018}, we will need to make an \emph{isotropy} condition for the convergence to hold, see below.  

We refer to Figure \ref{figure_finite_volumes} to visualize some notations. The volume of each cell $K$ is denoted by $|K|$. For each pair $(K,L) \in (\T^\sigma)^2$, we assume that $\bar{K} \cap \bar{L} =: (K|L)$ is a flat polytope with $d-1$-dimensional Haussdorff measure $|(K|L)|$. Although cells included in the interior of $\Omega$ are polytopes, we allow for the boundary of a cell to coincide with $\dr \Omega$, hence be curved. If $|(K|L)| > 0$, we write $K \sim L$ and denote by $\E^\sigma = \{ (K, L) \ : \ K \sim L \} \subset \T^\sigma \times \T^\sigma$ the set of ``edges''. For a cell $K$, we denote by $\T^\sigma_K \subset \T^\sigma$ the set of neighboring cells, i.e. $\T^\sigma_K = \{ L \in \T^\sigma \ : \ K \sim L \}$. We assume that if $K \sim L$ then $x_K - x_L$ is orthogonal to $(K|L)$. We denote by $d_{KL}$ the distance between $x_K$ and $x_L$ and $\nbf_{KL} = (x_L -x_K)/d_{KL}$ the normal to $(K|L)$. 
 
We assume that the mesh is \emph{uniformly regular} in the following sense. There exists a constant $c > 0$ independent on $\sigma$ such that, for any $\sigma > 0$ the following holds: the diameter of each cell is bounded by $c^{-1} \sigma$; additionaly for any $K \in \T^\sigma$ the ball of center $x_K$ and radius $c \sigma$ is contained in $K$; and if $(K,L) \in \E^\sigma$ then $|(K|L)| \geqslant c \sigma^{d-1}$. The terminology slightly differs from \cite{Gladbach2018}: we have added the word ``uniformly'' to make the distinction with the regularity assumption of the discretization presented in the previous section.
 
Eventually, following \cite[Definition 1.3]{Gladbach2018} we make an isotropy assumption. Specifically there exists $(\varepsilon_\sigma)_\sigma$, tending to $0$ as $\sigma \to 0$ such that for every $\sigma > 0$, every cell $K \in \T^\sigma$ and every vector $v \in \R^d$, there holds
\begin{equation}
\label{equation_isotropy}
\frac{1}{2 |K|} \sum_{L \in \T^\sigma_K} |(K|L)| d_{KL} (v \cdot \nbf_{KL})^2 \leqslant (1+ \varepsilon_\sigma) |v|^2. 
\end{equation}
We refer to \cite{Gladbach2018} to get a better understanding of this condition, which the authors have identified as being necessary for convergence to hold. As they indicate, if the mesh satisfies the ``center of mass'' condition, that is if for $(K,L) \in \E^\sigma$ the mean between $x_K$ and $x_L$ is also the center of mass of $(K|L)$, then the isotropy condition holds.

We give also ourselves a \emph{symmetric} mean $\theta : \R_+ \times \R_+  \to \R_+$ admissible in the sense of \cite[Definition 2.3]{Gladbach2018}, that is: a smooth, positively $1$-homogeneous, symmetric, jointly concave function such that $\theta(1,1) = 1$. It implies in particular that $\theta(a,b) \leqslant (a+b)/2$.

\medskip

Then, for $\sigma$ given, that is for a given mesh, we define the approximation of dynamical optimal transport as follows. 
\begin{itemize}
\item[•] $\X_\sigma = \R^{\T^\sigma}$ while $\X_{\sigma,+} = (\R_+)^{\T^\sigma}$ and $\Y_\sigma \subset \R^{\E^\sigma}$ is the set of $(M_{KL})_{(K,L) \in \E^\sigma}$ such that $M_{KL} = - M_{LK}$ for any $(K,L) \in \E^\sigma$. As elements of $\Y_\sigma$ are scalar fields rather than vectorial ones, a generic element in $\Y_\sigma$ will be denoted by $M$ and not $\Mbf$.
\item[•] The divergence operator is defined by, if $M \in \Y_\sigma$ and $K \in \T^\sigma$, 
\begin{equation*}
(\Div_\sigma M)_K := \sum_{L \in \T^\sigma_K} \frac{|(K|L)|}{|K|} M_{KL}.
\end{equation*} 
\item[•] The discrete action is defined on $\X_\sigma \times \Y_\sigma$ as
\begin{equation*}
A_\sigma(P, M) := \sum_{(K,L) \in \E^\sigma} \frac{M_{KL}^2}{ 4 \theta(P_K, P_L)} |(K|L)| d_{KL} 
\end{equation*}
We have a supplemental $1/2$ factor as each edge appears twice in $\E^\sigma$. The discrete action $A_\sigma$ easily satisfies the homogeneity and monotonicity requirements of Definition \ref{definition_approximation}. In the formula above we follow the convention of \eqref{equation_convention_b22a} to define the quotient: such a choice makes $A_\sigma$ convex and lower semi-continuous. 
\end{itemize}

\noindent In other words, the density is defined on the cells and the momentum on the edges, i.e. at interfaces between cells. Both the values of the density and the momentum on cells or edges are thought as \emph{intensive}, that is they are normalized by the volume or area of the cell or edge they live on. Up to now, we have just copied the definitions of \cite{Gladbach2018}.

As we will have to consider assumption \ref{asmp_interp_action}, we need to specify a representation of $\Y'_\sigma$ the dual of $\Y_\sigma$. To that extent we consider $\Y_\sigma$ as a (finite dimensional) Hilbert space with scalar product define by, for $B,M \in \Y_\sigma$, 
\begin{equation*}
\langle M, B \rangle = \sum_{(K,L) \in \E^\sigma} \frac{M_{KL} B_{KL}}{2} d_{KL} |(K|L)|.
\end{equation*} 
With that choice which allows us to identify $\Y'_\sigma$ with $\Y_\sigma$, the Legendre transform of the action is simply given by, for $P \in \X_{\sigma, +}$ and $B \in \Y_\sigma$, 
\begin{equation*}
A^\star_\sigma (P, B) = \sum_{(K,L) \in \E^\sigma} \frac{\theta(P_K, P_L)}{ 4 } (B_{KL})^2 d_{KL} |(K|L)|.
\end{equation*}

\medskip

Now we need to define reconstruction and sampling operators. 
\begin{itemize}
\item[•] To sample, we assume that $\mbf$ has a continuous density w.r.t. the $d$-dimensional Lebesgue measure, still denoted by $\mbf$, while $\rho \in \M(X)$ is arbitrary. We define 
\begin{equation*}
(S_{\X_\sigma} (\rho))_K = \frac{\rho(K)}{|K|} = \fint_K \ddr \rho \ \text{ and } \ (S_{\Y_\sigma} (\mbf))_{KL} = \frac{1}{|(K |L)|} \int_{(K|L)} \mbf \cdot \mathbf{n}_{KL} \, \ddr \mathcal{H}^{d-1} = \fint_{(K|L)} \mbf \cdot \mathbf{n}_{KL},
\end{equation*}
where $\mathcal{H}^{d-1}$ is the $(d-1)$-dimensional Hausdorff measure. 

\review{We have to assume that each boundary $(K|L)$ belongs to either $K$ or $L$, in such a way that $S_{\X_\sigma}(\rho)$ is still unambiguously defined if $\rho((K|L)) > 0$. It is enough to assume that the cells form a partition of $\Omega$, and to impose no openness or closedness assumption on the cells. Assumptions \ref{asmp_interp_sampling}, \ref{asmp_interp_der} and \ref{asmp_samp_actio} clearly hold whatever the choice which is made, as well as Assumption \ref{asmp_controllability} as the reader can see below.}

\item[•] On the other hand, reconstruction operators are defined as follows. Let $P \in \X_\sigma$ and $M \in \Y_\sigma$. We define
\begin{multline*}
R^A_{\X_\sigma} (P) = \sum_{K \in \T^\sigma} P_K \1_K,
 \ \ R^{\CE}_{\X_\sigma} (P) = \sum_{K \in \T^\sigma} P_K |K| \delta_{x_K}, \\
\text{ and } \ R_{\Y_\sigma} (M) = \frac{1}{2} \sum_{(K,L) \in \E^\sigma}  M_{KL} d_{KL} \nbf_{KL} \1_{(K|L)}.  
\end{multline*}
More specifically, the measure $R_{\Y_\sigma} (M)$ is a $(d-1)$-dimensional measure concentrated on the interfaces between cells: if $\bbf \in C(\Omega, \R^d)$ then 
\begin{equation*}
\langle R_{\Y_\sigma} (M), \bbf \rangle = \frac{1}{2} \sum_{(K,L) \in \E^\sigma} M_{KL} d_{KL} \int_{(K|L)} \bbf(x) \cdot \nbf_{KL}  \, \mathcal{H}^{d-1}(\ddr x). 
\end{equation*}
Again notice the factor $1/2$ as each edge is counted twice. 
\end{itemize} 

\noindent It is really straightforward to check that $R^A_{\X_\sigma}, R^{\CE}_{\X_\sigma}$ map $\X_{\sigma,+}$ into $\M_+(X)$ while $S_{\X_\sigma}(\M_+(X)) \subset \X_{\sigma,+}$. It is also very easy to check that in this case, given the scalar product that we put on $\Y_\sigma$, that there holds $R_{\Y_\sigma}^\top = S_{\Y_\sigma}$ (provided one identifies measures with their density w.r.t. Lebesgue measure). 

\begin{prop}
Let $\Omega \subset \R^d$ compact convex with a smooth boundary. For each $\sigma > 0$, assume that we have $(\T^\sigma, (x_K)_{K \in \T^\sigma})$ an admissible uniformly regular mesh of diameter $\sigma$ and which satisfies the isotropy condition. Let $\theta$ be an admissible symmetric mean in the sense of of \cite{Gladbach2018}. 

Then, provided that $(\X_\sigma, \Y_\sigma, A_\sigma, \Div_\sigma)_\sigma$ are defined as above, we get a family of finite dimensional models of dynamical optimal transport which is adapted to $\Omega$ (in the sense of Definition \ref{definition_adapated}), with reconstruction and sampling operators  $(R^{\CE}_{\X_\sigma}, R^A_{\X_\sigma}, R_{\Y_\sigma}, S_{\X_\sigma}, S_{\Y_\sigma})_\sigma$ defined as above.   
\end{prop}

\begin{proof}
Let us show that what we have defined satisfies all the requirements of Definition \ref{definition_adapated}. We will not check the different requirements in the order they are stated, but rather highlight the way the objects are chosen. 

We first check commutation between derivation and sampling \ref{asmp_samp_der}. This is just the application of the divergence theorem. Indeed, if $\mbf \in \M(\Omega)^d $ has a smooth density and satisfies $\mbf \cdot \nbf_\Omega = 0$ on $\dr \Omega$ then for $\sigma > 0$ and $K \in \T^\sigma$,  
\begin{equation*}
(S_{\X\sigma} (\nabla \cdot \mbf))_K = \fint_K (\nabla \cdot \mbf) = \sum_{L \sim K} \frac{1}{|K|} \int_{(K|L)} \mbf \cdot \mathbf{n}_{KL} = ((\Div_\sigma \circ S_{\Y_\sigma})(\mbf))_K.  
\end{equation*}
The asymptotic exactitude for the action $A_\sigma$ when evaluated on samples of smooth densities and momenta \ref{asmp_samp_actio} amounts to check that the quadrature formula is correct. Let $\Bc$ and $\Bc'$ be bounded sets of $C^1(\Omega)$ and $C^1(\Omega,\R^d)$ respectively, with functions in $\Bc$ uniformly bounded from below by a positive constant. Let $\rho \in \M(\Omega)$ and $\mbf \in \M(\Omega)^d$ whose density w.r.t. Lebesgue measure (still denoted by $\rho$ and $\mbf$) belong to $\Bc$ and $\Bc'$ respectively. In particular, up to an error controlled by $C \sigma$, we can replace $(S_{\Y_\sigma}(\mbf))_{KL}$ by $\mbf(x_K) \cdot \nbf_{KL}$ and $(S_{\X_\sigma}(\rho))_{K}$ by $\rho(x_K)$ hence 
\begin{align*}
A_\sigma(S_{\X_\sigma}(P), S_{\Y_\sigma}(\mbf)) & = \sum_{(K,L) \in \E^\sigma} \frac{ |(S_{\Y_\sigma}(\mbf))_{KL}|^2  }{ 4 \theta( (S_{\X_\sigma}(\rho))_{K} , (S_{\X_\sigma}(\rho))_{L} )} d_{KL} |(K|L)| \\
& \leqslant \sum_{K \in \T^\sigma} \frac{1}{4 \rho(x_K)} \sum_{L \in \T^\sigma_K} |\mbf(x_k) \cdot \mathbf{n}_{KL}|^2 d_{KL} |(K|L)| + C \sigma \\
& \leqslant (1+ \varepsilon_\sigma) \sum_{K \in \T^\sigma} |K| \frac{|\mbf(x_k)|^2}{2 \rho(x_k)}  + C \sigma 
= (1+\varepsilon_\sigma)(A(\rho, \mbf) + C \sigma) + C \sigma ,  
\end{align*}
where the third inequality comes from \eqref{equation_isotropy} and the last one is just a quadrature formula. As $A(\rho, \mbf) $ is bounded by a constant which depends only on $\Bc$ and $\Bc'$, we get the result. Although we have used \eqref{equation_isotropy} it was not necessary and we could have relied only on \cite[Lemma 5.4]{Gladbach2018} at the price of slightly more involved estimates.   

Now we turn to reconstruction. For the one-sided estimation on the action, we use \ref{asmp_interp_action} rather than \ref{asmp_interp_action_stronger}, as we have not been able to build reconstruction operators such that the latter holds. We rely crucially on the isotropy assumption for this step. Let $\Bc$ a bounded set of $C^1(\Omega, \R^d)$ and $\bbf \in \Bc$, as well as $P \in \X_\sigma$. Using the property $\theta(a,b) \leqslant (a+b)/2$, we get easily that 
\begin{multline*}
A_\sigma^\star( P, R_{\Y_\sigma}^\top (\bbf) ) 
= \sum_{(K,L) \in \E^\sigma} \frac{\theta(P_K, P_L)}{ 4 } ((R_{\Y_\sigma}^\top (\bbf))_{KL})^2 d_{KL} |(K|L)| \\
\leqslant  \frac{1}{4} \sum_{K \in \T^\sigma} P_K \sum_{L \in \T^\sigma_K} ((R_{\Y_\sigma}^\top (\bbf))_{KL})^2 d_{KL} |(K|L)|.  
\end{multline*}  
Given the regularity of $\bbf$, it is clear that $(R_{\Y_\sigma}^\top (\bbf))_{KL}$ can be replaced by $\bbf(x_K) \cdot \nbf_{KL}$ up to an error controlled by $\sigma$. Hence we can write, using the isotropy condition \eqref{equation_isotropy}, that 
\begin{multline*}
A_\sigma^\star( P, R_{\Y_\sigma}^\top (\bbf) ) 
\leqslant \frac{1}{4} \sum_{K \in \T^\sigma} P_K \sum_{L \in \T^\sigma_K} ( \bbf(x_K) \cdot \nbf_{KL} )^2 d_{KL} |(K|L)| + C \sigma \| R^{\CE}_{\X_\sigma} (P) \|  \\ 
\leqslant  (1+ \varepsilon_\sigma) \frac{1}{2} \sum_{K \in \T^\sigma} P_K |K| \left| \bbf(x_K) \right|^2 + C \sigma \| R^{\CE}_{\X_\sigma} (P) \| . 
\end{multline*}  
On the other hand, by smoothness of $\bbf$, it is clear that $|K| |\bbf(x_K)|^2 \leqslant \int_{K} |\bbf|^2 + C |K| \sigma$, and that $\bbf$ is uniformly bounded by a constant that depends only on $\Bc$. Hence we conclude with the definition of $R^A_{\X_\sigma}$ that 
\begin{equation*}
A_\sigma^\star( P, R_{\Y_\sigma}^\top (\bbf) )  \leqslant \frac{1}{2} \sum_{K \in \T^\sigma} P_K \int_K |\bbf|^2 + \varepsilon_\sigma \| R^{\CE}_{\X_\sigma}(P) \| = A^\star( R^A_{\X_\sigma}(P), \bbf ) + \varepsilon_\sigma \| R^{\CE}_{\X_\sigma}(P) \|.
\end{equation*}
To check the second estimate in \ref{asmp_interp_action}, let us take $\bbf \in C(\Omega,\R^d)$, by homogeneity of the formula w.r.t. $\bbf$ we assume that $|\bbf(x)| \leqslant 1$ for all $x \in \Omega$. In particular, $((R_{\Y_\sigma}^\top (\bbf))_{KL})^2 \leqslant 1$ for all $(K,L) \in \E^\sigma$, hence 
\begin{equation*}
A_\sigma^\star( P, R_{\Y_\sigma}^\top (\bbf) ) 
\leqslant \frac{1}{2} \sum_{K \in \T^\sigma} P_K \sum_{L \in \T^\sigma_K} d_{KL} |(K|L)|.  
\end{equation*}  
By uniform regularity of the mesh, up to a constant independent on $\sigma$ the quantity $\sum_{L \in \T_K^\sigma} d_{KL} |(K|L)|$ can be controlled by $|K|$, hence we get the second estimate in \ref{asmp_interp_action}.

On the other hand, to check \ref{asmp_interp_der} the commutation between reconstruction and derivation, let us take $\Bc$ a bounded set of $C^2(\Omega)$ and $\phi \in \Bc$, as well as $M \in \Y_\sigma$. Then  
\begin{align*}
\langle (R^{\CE}_{\X_\sigma} & \circ \Div_\sigma)  (M), \phi \rangle + \langle R_{\Y_\sigma} (M), \nabla \phi \rangle \\ 
& = \sum_{K \in \T^\sigma} \sum_{L \in \T^\sigma_K} \phi(x_K) |(K|L)| M_{KL} + \frac{1}{2} \sum_{(K,L) \in \E^\sigma} \left(  M_{KL} d_{KL} \fint_{(K|L)} \nabla \phi \cdot \nbf_{KL}  \right) \\
& = \sum_{(K,L) \in \E^\sigma} |(K|L)| \frac{M_{KL}}{2} \left(  ( \phi(x_K) - \phi(x_L) ) + d_{KL}  \fint_{(K|L)} \nabla \phi \cdot \nbf_{KL} \right).
\end{align*}
By smoothness of $\phi$, one can say that 
\begin{equation*}
\left| ( \phi(x_K) - \phi(x_L) ) + d_{KL}  \fint_{(K|L)} \nabla \phi \cdot \nbf_{KL} \right| \leqslant C \sigma^2,
\end{equation*}
where $C$ depends only on $\Bc$, hence we deduce that the total error can be bounded by  
\begin{equation*}
\left| \langle (R^{\CE}_{\X_\sigma}  \circ \Div_\sigma)  M, \phi \rangle + \langle R_{\Y_\sigma} (M), \nabla \phi \rangle \right| \leqslant C \sigma \sum_{(K,L) \in \E^\sigma} |M_{KL}| |(K | L)| d_{KL} \leqslant C \sigma \| R_{\Y_\sigma} (M) \|,
\end{equation*}
which is the desired result.

The first and second points \ref{asmp_interp_sampling} and \ref{asmp_interps} are very straightforward to check, this is left to the reader.

Lastly, we need to check \ref{asmp_controllability}, that is controllability. As explained in Remark \ref{rmk_controllability}, we mimic the construction in the continuous case. Let $x, y \in \Omega$ and $K_x, K_y \in \T^\sigma$ the cells to which they respectively belong. By regularity of the mesh, there exists a path $K_x = K_1, K_2,  \ldots , K_Q = K_y$ of length $Q$ such that $K_{i-1} \sim K_{i}$ for every $i$ and $C^{-1} (Q-1) \sigma \leqslant \sum_i |x_{K_{i-1}} - x_{K_i}| \leqslant C (|x-y| + \sigma)$. The density $\hat{P} \in \X_\sigma$ is chosen as 
\begin{equation*}
\hat{P}_{K_i} = \frac{1}{Q |K_i|},
\end{equation*}
and $\hat{P}_K = 0$ if $K$ is not in the chain $\{ K_1, \ldots, K_Q \}$. On the other hand, the momentum is chosen to be $0$ on every edge not connecting a two consecutive cells in $\{ K_1, \ldots, K_R \}$ and for every $i \in \{ 2, \ldots, Q \}$ we set 
\begin{equation*}
(\hat{M}_1)_{K_{i-1} K_i} = - \frac{Q+1-i}{Q |(K_{i-1}|K_i)|} \  \text{ and } \ (\hat{M}_2)_{K_{i-1} K_i} = \frac{i-1}{Q |(K_{i-1}|K_i)|},
\end{equation*}  
and of course the opposite on the edge $(K_i,K_{i-1})$. It is just a matter of a straightforward computation to check that $\Div_\sigma (\Mbf_1) = \hat{P} - S_{\X_\sigma} (\delta_x)$ and $\Div_\sigma (\Mbf_2) = \hat{P} - S_{\X_\sigma} (\delta_y)$. On the other hand, as the mesh is $\sigma$-regular, we know that $|K| \sim C \sigma^d$ and $|(K|L)| \sim C \sigma^{d-1}$ for any $(K,L) \in \E^\sigma$. We can thus estimate 
\begin{multline*}
A_\sigma(\hat{P}, \hat{M}_1) = \sum_{i=2}^{Q} \frac{\left( \frac{Q+1-i}{Q |(K_{i-1}|K_i)|} \right)^2}{2 \theta \left( \frac{1}{Q|K_{i-1}|}, \frac{1}{Q|K_i|} \right) } |(K_{i-1}|K_i)| d_{K_{i-1} K_i} \\ \leqslant \frac{C}{Q} \left( \sum_{i=2}^Q (Q+1-i)^2 \right) \frac{\sigma}{\sigma^{d-1} \sigma^{-d}} \leqslant C ((Q-1) \sigma)^2.
\end{multline*}
Given the estimate $(Q-1) \sigma \leqslant C( |x-y| + \sigma )$, we indeed get the announced result. The estimate for $A_\sigma(\hat{P}, \hat{M}_2)$ is entirely analogous. 

\end{proof}

\section{Extensions}
\label{section_extension}

In this section we investigate two natural extensions: instead of fixing the final value of the density, we add a cost depending on it in the objective functional; alternatively we can also add a running cost depending on the density. Although the first extension is rather easy, the second one is more involved and we only explain what the additional difficulties are without fully addressing them. 


\begin{rmk}
One alternative that we \review{do} not explore is the use of other Lagrangians, that is to replace $|\mbf|^2 / (2 \rho)$ by $\rho L(\cdot, \mbf/\rho)$ where the Lagrangian $L : TX \to [0, + \infty]$ is convex in its second variable but not necessarily quadratic. In such a situation, applying the heat flow does not necessarily decrease the action hence the techniques of proof for Proposition \ref{prop_regularization} would break down. Nevertheless, if $X$ is the torus, one could hope to use convolution instead of applying the heat flow and rely on the convexity of the Lagrangian. We have preferred to explore complicated geometries (that is Riemannian manifolds like in \cite{Lavenant2018}) restricted to quadratic Lagrangians rather than simple geometries (taking $X$ to be a torus) with more general Lagrangians: we find the results more interesting in the first case rather than in the second one, but this is mainly a matter of personal taste.  
\end{rmk}

\subsection{Final penalization of the density}

Given the \emph{minimizing movement scheme} for gradient flows \cite{Jordan1998} in the Wasserstein, it is natural to take $G : \M_+(X) \to [0, + \infty]$ and to consider the following problem, whose unknowns are $\rho \in \M([0,1] \times X)$, $\mbf \in \M([0,1] \times TX)$ and $\rho_1 \in \M_+(X)$ (but $\rho_0 \in \M_+(X)$ is given): 
\begin{equation}
\label{equation_problem_JKO}
\min_{\rho, \mbf, \rho_1} \J_{\rho_0, \rho_1}(\rho, \mbf) + G(\rho_1). 
\end{equation}
In other words, the final value of the density is no longer fixed but penalized with the help of $G$. Of course to approximate the problem the idea is to use $G_\sigma$ an approximation of $G$ which is defined over $\X_\sigma$. The natural requirements for it to be a good approximation of $G$ are the followings. 

\begin{defi}
\label{definition_adapted_final}
Let $(\X_\sigma, \Y_\sigma, A_\sigma, \Div_\sigma)_\sigma$ be a finite dimensional model of dynamical optimal transport with reconstruction and sampling operators $(R^{\CE}_{\X_\sigma}, R^A_{\X_\sigma}, R_{\Y_\sigma}, S_{\X_\sigma}, S_{\Y_\sigma})_\sigma$.

Assume that $G : \M_+(X) \to [0, + \infty]$ is a given functional over the space of measures. For every $\sigma > 0$ we give ourselves $G_\sigma : \X_\sigma \to [0, + \infty]$. The family of functions $(G_\sigma)_\sigma$ is an adapted discretization of $G$ if: 
\begin{enumerate}[label=\textnormal{(A\arabic*)}]
\setcounter{enumi}{7}
\item 
\label{asmp_G_GammaLiminf}
There exists $(\varepsilon_\sigma)_\sigma$ which tends to $0$ when $\sigma \to 0$ such that, for any $\sigma > 0$ and any $P \in \X_\sigma$, 
\begin{equation*}
G( R^A_{\X_\sigma} (P)  ) \leqslant (1+ \varepsilon_\sigma) G_\sigma( P). 
\end{equation*}
\item 
\label{asmp_G_GammaLimsup}
For any $\rho \in \M_+(X)$, there holds 
\begin{equation*}
\limsup_{\sigma \to 0} \, G_\sigma( S_{\X_\sigma} (\rho) ) \leqslant G( \rho ). 
\end{equation*}
\end{enumerate} 
\end{defi}

This definition was chosen for the following result to hold. 

\begin{theo}
Let $(\X_\sigma, \Y_\sigma, A_\sigma, \Div_\sigma)_\sigma$ a family of finite dimensional models of dynamical optimal transport which is adapted to $X$ in the sense of Definition \ref{definition_adapated}, with reconstruction and sampling operators $(R^{\CE}_{\X_\sigma}, R^A_{\X_\sigma}, R_{\Y_\sigma}, S_{\X_\sigma}, S_{\Y_\sigma})_\sigma$. Assume also that $(G_\sigma)_\sigma$ is an adapted discretization of a lower semi-continuous functional $G : \M_+(X) \to [0, + \infty]$ in the sense of Definition \ref{definition_adapted_final}. 

Let $\rho_0 \in \M_+(X)$ be given.

For any $N$ and any $\sigma > 0$, let $(P^{N, \sigma}, \Mbf^{N, \sigma})$ a minimizer of the functional 
\begin{equation*}
(P, \Mbf) \in (\X_\sigma)^{N+1} \times (\Y_\sigma)^N \mapsto \J^{N, \sigma}_{ S_{\X_\sigma}(\rho_0), P_N } (P, \Mbf) + G_\sigma(P_N).
\end{equation*}
Then, in the limit $N \to + \infty$ and $\sigma \to 0$, up to the extraction of a subsequence, $\Rc^{\CE}_{N, \X_\sigma}(P^{N, \sigma})$ and $\Rc^{A}_{N, \X_\sigma}(P^{N, \sigma})$ will converge weakly to the same limit $\rho \in \M([0,1] \times X)$, the measure $R^{A}_{\X_\sigma}(P_N^{N, \sigma})$ will converge weakly to $\rho_1 \in \M_+(X)$ and $\Rc_{N, \Y_\sigma}(\Mbf^{N, \sigma})$ will converge weakly to $\mbf \in \M([0,1] \times TX)$, in addition the triplet $(\rho, \mbf, \rho_1)$ is a solution of the problem \eqref{equation_problem_JKO} and
\begin{equation*}
\lim_{N \to + \infty, \sigma \to 0} \left( \J^{N, \sigma}_{ S_{\X_\sigma}(\rho_0), P^{N, \sigma}_N } (P^{N, \sigma}, \Mbf^{N, \sigma}) + G_\sigma(P^{N, \sigma}_N) \right) =  \J_{\rho_0, \rho_1} ( \rho, \mbf ) + G(\rho_1).
\end{equation*}
\end{theo}

\begin{proof}
As the proof closely follows the proofs of Section \ref{section_proof}, we will only sketch it. 

We can copy almost line by line the proof of Theorem \ref{theo_GammaLiminf}. Indeed, the uniform control on the masses of the measures $\Rc^{\CE}_{N, \sigma}(P^{N, \sigma})$, $\Rc^{A}_{N, \sigma}(P^{N, \sigma})$ and $\Rc_{N, \Y_\sigma}(\Mbf^{N, \sigma})$ is obtained only through the assumption that $(R^{\CE}_{\X_\sigma} \circ S_{\X_\sigma} )(\rho_0)$ is bounded uniformly in $\sigma$, which is still the case. In particular, we deduce that $R^A_{\X_\sigma}( P^{N, \sigma}_N )$ is bounded in $\M_+(X)$, hence converges to a limit $\rho_1$. Passing to the limit in the continuity equation is done exactly in the same way. Then to pass to the limit in the objective functional, we use the estimate \ref{asmp_G_GammaLiminf} as well as the lower semi-continuity of $G$ to see that  
\begin{equation*}
\liminf_{N \to + \infty, \sigma \to 0}  G_\sigma(P^{N, \sigma}_N) \geqslant G(\rho).
\end{equation*}  
Combining with \eqref{equation_GammaLiminf_A}, we get the analogous of Theorem \ref{theo_GammaLiminf}.

To go the other way around, that is to prove the analogous of Theorem \ref{theo_GammaLimsup}, if $(\rho, \mbf, \rho_1)$ is a solution of the problem \eqref{equation_problem_JKO}, we can use exactly the same discrete competitors as those built in the proof of Theorem \ref{theo_GammaLimsup}. The only new term to handle is the boundary term, but this is easily done with \ref{asmp_G_GammaLimsup} if we take $P^{N, \sigma}_N = S_{\X_\sigma}(\rho_1)$. 
\end{proof}

\subsection{Comments on how to add a running cost depending on the density}

Starting from considerations in mean field games \cite{Benamou2016}, it is also natural to take $F : \M_+(X) \to [0, + \infty]$ convex lower semi-continuous and to consider problems of the form
\begin{equation}
\label{equation_problem_MFG}
\min_{\rho, \mbf} \J_{\rho_0, \rho_1}(\rho, \mbf) + \int_0^1 F(\rho_t) \,\ddr t, 
\end{equation}
where $\rho_0$ and $\rho_1$ are fixed or penalized. For the notation $\rho_t$, we refer to Remark \ref{rmk_rho_curve}. For instance, one can take $F$ to be a quadratic penalization of the density, or (minus) the Boltzmann entropy. The adaptation of the strategy of this article to this case is more involved, and would imply to dive deeper in the details of the discretization. For this reason, and not to overburden the article we will not give precise statements but only give hints about what the additional difficulties are. Let us point out that one can see numerical results in \cite[Section 5.4]{Lavenant2018} when $F$ is a quadratic penalization.

Of course, the approach would be to choose $F_\sigma$ an adapted discretization of $F$ in the sense of Definition \ref{definition_adapted_final}, and then to approximate Problem \eqref{equation_problem_MFG} with the minimization of 
\begin{equation*}
(P, \Mbf) \in (\X_\sigma)^{N+1} \times (\Y_\sigma)^N \mapsto \J^{N, \sigma}_{ S_{\X_\sigma}(\rho_0), S_{\X_\sigma}(\rho_1) }(P,\Mbf) + \tau \sum_{k=1}^N F_\sigma \left( \frac{P_{k-1} + P_k}{2} \right).
\end{equation*}

The analogue of Theorem \ref{theo_GammaLiminf} is very easy to prove once we assume that $F$ is lower semi-continuous. However, in order to prove an analogue of Theorem \ref{theo_GammaLimsup}, one faces two additional difficulties. 
\begin{itemize}
\item[•] Now, in the analogue of the controllability property (Proposition \ref{prop_controllability}), one also needs to control (using the notations of \ref{asmp_controllability}) $F_\sigma(\hat{P})$ with $F(\rho_0)$ and $F(\rho_1)$. If we use the construction of the discretization on triangulations of surfaces (see Section \ref{subsection_finite_elements}), such control would rely on \emph{geodesic convexity} of the functional $F$ (i.e. convexity along geodesics in the Wasserstein space, which is linked to the Ricci curvature of the underlying space), as well as the convexity of $F$, and a finer analysis of the reconstruction and sampling operators. 
\item[•] Moreover, one would want the regularization procedure (Proposition \ref{prop_regularization}) not to increase too much the running cost $\int F(\rho_t) \, \ddr t$. Given the proof of Proposition \ref{prop_regularization}, this is easy if $F$ is either continuous for the topology of weak convergence, or decreases when composed with the heat flow. For instance, $F$ being a potential energy or an internal energy (provided the function giving the density of energy in terms of the density is convex) would easily fit in the framework.
\end{itemize}

\section*{Acknowledgments}

The present article was mostly written while the author was a PhD student at \emph{Université Paris-Sud}, and completed while he was a postdoctoral fellow of the \emph{Pacific Institute for the Mathematical Sciences} \review{at the University of British Columbia}. The author acknowledges the support of ANR MAGA (ANR-16-CE40-0014). He wants to thank Quentin Mérigot for fruitful discussions \review{as well as the anonymous referees for their several comments and suggestions}.

\appendix 

\section{Wasserstein distance}
\label{section_wasserstein}

We briefly recall the definition of the quadratic Wasserstein distance, see \cite{AGS, Villani2008, SantambrogioOTAM} to learn more. We recall that $d_g$ is the geodesic distance on the Riemannian manifold $(X,g)$. The quadratic Wasserstein distance is denoted by $W_2$: if $\mu, \nu \in \M_+(X)$ are measures on $X$ sharing the same total mass then 
\begin{equation*}
W_2^2(\mu, \nu) := \min_{\pi} \iint_{X \times X} d_g(x,y)^2 \,\pi( \ddr x, \ddr y)
\end{equation*}
where $\pi$ is a measure on the product space $X \times X$ whose first marginal is $\mu$ and second marginal is $\nu$. There always exists an optimal $\pi$, and it is called an optimal transport plan between $\mu$ and $\nu$.  

\begin{prop}
For $a > 0$, let $\Prob_a(X) \subset \M_+(X)$ the set of positive measures over $X$ whose total mass is $a$. Then $W_2$ metrizes the weak convergence of measures on $\Prob_a(X)$. 
\end{prop}

\begin{proof}
See for instance \cite[Theorem 6.9]{Villani2008}, and we emphasize that the compactness of $X$ is necessary for this statement to hold. 
\end{proof}

Although, as highlighted in Theorem \ref{theorem_J_W2}, the Wasserstein distance is closely related to the problem we are studying, we have only used the definition above as a convenient tool to quantify convergence in measure, see Propositions \ref{prop_regularization} and \ref{prop_controllability}.

\section{Regularizing curves of measures}
\label{section_regularization_curves}

Regularizing curves valued in the space of probability measures has been attacked in a variety of setting. In Proposition \ref{prop_regularization}, as we want both a regularization of the density and the momentum, and we are in a Riemannian manifold possibly with a smooth and convex boundary, we do not think that a statement for this precise situation already exists. We refer the reader to \cite[Chapter 8]{AGS} or \cite[Chapter 5]{SantambrogioOTAM} for regularization arguments on flat spaces, and to \cite[Proposition 4.3]{Otto2005} or \cite{Erbar2010} for regularization on Riemannian manifolds without boundaries.  

The idea is to use the heat flow to regularize because it has the effect of decreasing the action. However, we apply the heat flow only to the density, and we choose the momentum \emph{a posteriori}, once the density is regularized. To keep it consistent with the point of view chosen in this article, we prefer to use the Bakry-\'Emery estimate to prove that the action decreases along the heat flow. 

\bigskip

For $\rho \in \M_+([0,1] \times X)$ and $\rho_0, \rho_1 \in \M_+(X)$ we define 
\begin{equation}
\label{definition_optimal_m}
\A_{\rho_0, \rho_1}(\rho) = \min_{\review{\bar{\mbf}}} \left\{ \A(\rho, \review{\bar{\mbf}}) \ : \ \review{\bar{\mbf}} \in \M([0,1] \times TX) \text{ and } (\rho, \review{\bar{\mbf}}) \in \CE(\rho_0, \rho_1) \right\}.
\end{equation}

\begin{rmk}
In fact $\rho_0$ and $\rho_1$ are superfluous because if $\review{\bar{\mbf}} \in \M([0,1] \times X)$ is such that $(\rho, \review{\bar{\mbf}})$ satisfies the continuity equation (let's say tested again functions compactly supported in $(0,1) \times X$) and $\A_{\rho_0, \rho_1}(\rho, \review{\bar{\mbf}}) < + \infty$, then one can recover $\rho_0$ and $\rho_1$ from $\rho$, see Remark \ref{rmk_rho_curve}. However, as we want to avoid to deal with disintegration, we prefer to stick to this redundant formulation. 
\end{rmk}

The quantity $\A_{\rho_0, \rho_1}(\rho)$ admits a ``dual'' formulation which will be helpful in the subsequent estimates. 

\begin{lm}
\label{lemma_dual_action}
Let $\rho \in \M_+([0,1] \times X)$ and $\rho_0, \rho_1 \in \M_+(X)$. Then there holds 
\begin{equation*}
\A_{\rho_0, \rho_1}(\rho) = \sup_{\phi} \left\{ \langle \rho_1, \phi(1, \cdot) \rangle - \langle \rho_0, \phi(0, \cdot) \rangle - \left\llangle \rho, \dr_t \phi + \frac{1}{2} |\nabla \phi|^2  \right\rrangle  \ : \ \phi \in C^1([0,1] \times X)  \right\}. 
\end{equation*}
\end{lm}

\begin{proof}
We introduce the Hilbert space $H = L^2_\rho([0,1] \times X, TX)$: it is the space of measurable vector fields $\vbf : [0,1] \times X \to TX$ endowed with scalar product
\begin{equation*}
\llangle \vbf, \wbf \rrangle_H = \iint_{[0,1] \times X} g_x(\vbf(t,x), \wbf(t,x)) \, \rho(\ddr t, \ddr x) 
\end{equation*}
and with a norm $\| \vbf \|_H^2 = \llangle \vbf, \vbf \rrangle_H$, where vector fields which coincide $\rho$-a.e. are considered equal. We know that $\A(\rho, \review{\bar{\mbf}})$ is finite if and only if $\review{\bar{\mbf}}$ has a density w.r.t. $\rho$, hence we can switch back again to take $\wbf$ the density of $\review{\bar{\mbf}}$ w.r.t. $\rho$ as an unknown. Given the definition of $\CE(\rho_0, \rho_1)$ we can write
\begin{equation*}
\A_{\rho_0, \rho_1}(\rho) = \min_{\wbf \in H} \left\{ \frac{1}{2} \| \wbf \|^2_H \ : \  \forall \phi \in C^1([0,1] \times X), \ \llangle \wbf, \nabla \phi \rrangle_H = \langle \rho_1, \phi(1, \cdot) \rangle - \langle \rho_0, \phi(0, \cdot) \rangle -\llangle \rho, \dr_t \phi \rrangle \right\}.
\end{equation*}
Let us call $\vbf$ the optimal vector field in the problem above. We see that $\vbf$ is the projection, in $H$, of the vector $0$ onto the affine space 
\begin{equation*}
Y = \left\{ \wbf \in H \ : \ \forall \phi \in C^1([0,1] \times X), \ \llangle \wbf, \nabla \phi \rrangle_H = \langle \rho_1, \phi(1, \cdot) \rangle - \langle \rho_0, \phi(0, \cdot) \rangle -\llangle \rho, \dr_t \phi \rrangle  \right\}.
\end{equation*}
In particular $\vbf$ is orthogonal to the linear space $Y^\perp$, hence it belongs to the closure in $H$ of $\{ \nabla \phi \ : \ \phi \in C^1([0,1] \times X) \}$. 

To get the conclusion, we simply write that $\frac{1}{2} \| \vbf \|^2_H = \sup_{\wbf \in H} \llangle \vbf, \wbf \rrangle_H - \frac{1}{2} \| \wbf \|^2_H$. Hence, for all $\phi \in C^1([0,1] \times X)$, taking $\wbf = \nabla \phi$, 
\begin{equation*}
\A_{\rho_0, \rho_1}(\rho) = \frac{1}{2} \| \vbf \|^2_H \geqslant \llangle \vbf, \nabla \phi \rrangle_H - \frac{1}{2} \| \nabla \phi \|^2_H = \langle \rho_1, \phi(1, \cdot) \rangle - \langle \rho_0, \phi(0, \cdot) \rangle - \llangle \rho, \dr_t \phi \rrangle - \frac{1}{2} \llangle \rho, |\nabla \phi|^2  \rrangle.
\end{equation*}
Moreover, taking a sequence $(\phi_n)_{n \in \N}$ in $C^1([0,1] \times X)$ such that $\nabla \phi_n$ converges to $\vbf$ in $H$ as $n \to + \infty$, we get the result. 
\end{proof}

Now we can introduce the main tool to regularize on manifold, that is the heat flow. Let $\Phi : [0, \infty) \times C(X) \to C(X)$ denote the heat flow on $X$ with Neumann boundary conditions. Specifically, if $a \in C(X)$, let us call $u : (0, + \infty) \times X \to X$ the unique solution to the heat equation with initial condition $a$, that is the Cauchy problem
\begin{equation*}
\begin{cases}
\dr_s u = \Delta u & \text{ on } (0, + \infty) \times \ring{X}, \\
u \cdot \nbf_{\dr X} = 0 & \text{ on } (0, + \infty) \times \dr X, \\
\dst{ \lim_{s \to 0} u(s, \cdot)  } = a & \text{ in } C(X), 
\end{cases}
\end{equation*}  
where $\nbf_{\dr X} : \dr X \to TX$ denotes the outward normal and $\Delta$ is the Laplace-Beltrami operator on $X$. Then for $s \geqslant 0$ we define $\Phi_s a \in C(X)$ as $u(s, \cdot)$. 

\review{We rely on the Bakry-\'Emery estimate to control the action of the heat flow on spatial derivatives, it is recalled in the following statement.}

\begin{theo}
\label{theorem_bakry_emery}
Let $(X,g)$ be a smooth compact Riemannian manifold with a (possibly empty) smooth and convex boundary. Then there exists $C \in \R$ such that for all $\phi \in C^1(X)$ and any $s \geqslant 0$, 
\begin{equation*}
\left|\nabla \left( \Phi_s \phi \right) \right|^2 \leqslant e^{C s} \Phi_s \left( |\nabla \phi|^2 \right).
\end{equation*}
\end{theo} 

\begin{proof}
As $X$ is smooth and compact, its Ricci curvature is bounded from below by $-C/2 \in \R$. To conclude, we refer the reader to \cite[Theorem 1.1]{Wang2011} for the case where $X$ may have a convex boundary.
\end{proof}

We have all the tools at our disposal to prove Proposition \ref{prop_regularization}. In the proof, if $a$ is a function defined on $\R$ and $b$ a function defined on $X$, then $a \otimes b$ is the function defined on $\R \times X$ by $(a \otimes b)(t,x) = a(t) b(x)$. By a slight abuse of notations, if $\rho \in \M([0,1] \times X)$ and $\phi \in C(\R \times X)$ then $\llangle \rho, \phi \rrangle$ denotes the duality product between $\rho$ and the restriction of $\phi$ to $[0,1] \times X$.

\begin{proof}[Proof of Proposition \ref{prop_regularization}]

Let $(\rho, \mbf) \in \M_+([0,1] \times X) \times \M([0,1] \times TX)$ be given, and $\rho_0, \rho_1 \in \M_+(X)$ such that $(\rho, \mbf) \in \CE(\rho_0, \rho_1)$ and  $\A(\rho, \mbf) < + \infty$. 

To regularize, we extend the curve $\rho$ by setting it equal to $\rho_0$ for $t \leqslant 0$ and $\rho_1$ for $t \geqslant 1$, and then we take a convolution in time and apply the heat flow in space. To be rigorous, we do that by duality.  Let $(\chi_n)_{n \in \N}$ be a smooth even approximation of unity on $\R$ such that $\chi_n$ is compactly supported in $[-1/n, 1/n]$ for any $n$. For any $\phi \in C([0,1] \times X)$, we define $\Theta_n(\phi) \in C(\R \times X)$ by, for any $t \in \R$, 
\begin{equation*}
\Theta_n(\phi)(t, \cdot) = \int_0^1 \chi_n(t-s) \Phi_{1/n}( \phi(s, \cdot) ) \, \ddr s.
\end{equation*}   
Then, we define $\tilde{\rho}_n \in \M([0,1] \times X)$ by 
\begin{equation*}
\llangle \tilde{\rho}_n, \phi \rrangle := \left\langle \rho_0, \int_{- \infty}^0 \Theta_n(\phi)(t, \cdot) \, \ddr t \right\rangle + \llangle \rho, \Theta_n(\phi) \rrangle + \left\langle \rho_1, \int_{1}^{+ \infty} \Theta_n(\phi)(t, \cdot) \, \ddr t \right\rangle.
\end{equation*} 
By parabolic regularity and regularity of $\chi_n$, for a fixed $n$ we know that $\Theta_n(\phi)$ belongs to $C^\infty([0,1] \times X)$, with derivatives controlled by $\| \phi \|_{L^1([0,1] \times X)}$. It allows easily, by duality, to say that $\tilde{\rho}_n$ has a density w.r.t. $\ddr x \otimes \ddr t$, and that this density (still denoted by $\tilde{\rho}_n$) belongs to $C^\infty([0,1] \times X)$. Moreover, as $\Theta_n(\phi)(t, \cdot)$ is uniformly bounded from below by a strictly positive constant depending only on $\| \phi(t, \cdot) \|_{L^1(X)}$ if $\phi$ is nonnegative, we know that $\tilde{\rho}_n$ is uniformly bounded from below by a strictly positive constant.  

Now we have to estimate the action, we will do it with the help of Lemma \ref{lemma_dual_action}. First we identify the boundary conditions: let us define $\tilde{\rho}_{n,0}$ and $\tilde{\rho}_{n,1}$ by, for any $\phi \in C(X)$, 
\begin{multline*}
\langle \tilde{\rho}_{n,0}, \phi \rangle := \frac{1}{2} \langle \rho_0, \Phi_{1/n} \phi \rangle + \left\llangle \rho, \chi_n  \otimes ( \Phi_{1/n} \phi) \right\rrangle, \\ 
\text{and } \langle \tilde{\rho}_{n,1}, \phi \rangle := \frac{1}{2} \langle \rho_1, \Phi_{1/n} \phi \rangle + \left\llangle \rho, (\chi_n(1 - \cdot)) \otimes ( \Phi_{1/n} \phi) \right\rrangle. 
\end{multline*}
They could be guessed by the way we have heuristically chosen $\tilde{\rho}_n$, or simply retro-engineered from the computations below. 

Next, we need to understand how $\Theta_n$ commutes with derivatives. For the temporal derivative, \review{an} integration by parts leads to, for $\phi \in C^1([0,1] \times X)$, 
\begin{equation*}
\dr_t [\Theta_n(\phi)](t, \cdot) = \int_0^1 \chi_n'(t-s) \Phi_{1/n}( \phi(s, \cdot) ) \, \ddr s 
= - \chi_n(t-1) \Phi_{1/n} \phi(1, \cdot) + \chi_n(t) \Phi_{1/n} \phi(0, \cdot) + \Theta_n(\dr_t \phi)(t, \cdot).
\end{equation*} 
On the other hand, using Theorem \ref{theorem_bakry_emery} as well as the convexity of the squared norm, 
\begin{equation*}
|\nabla \Theta_n(\phi)|^2 \leqslant e^{C/n} \Theta_n(|\nabla \phi|^2). 
\end{equation*}

Now let us take $\phi \in C^1([0,1] \times X)$ as a test function to evaluate $\A_{\tilde{\rho}_{n,0}, \tilde{\rho}_{n,1}}(\tilde{\rho}_n)$. We compute: 
\begin{align*}
\langle \tilde{\rho}_{n,1} & , \phi(1, \cdot) \rangle - \langle \tilde{\rho}_{n,0}, \phi(0, \cdot) \rangle - \left\llangle \tilde{\rho}_n, \dr_t \phi + \frac{1}{2} |\nabla \phi|^2  \right\rrangle \\
& =  \frac{1}{2} \langle \rho_1, \Phi_{1/n} \phi(1, \cdot) \rangle + \left\llangle \rho, (\chi_n(1 - \cdot) ) \otimes ( \Phi_{1/n} \phi(1, \cdot)) \right\rrangle \\
& - \frac{1}{2} \langle \rho_0, \Phi_{1/n} \phi(0, \cdot) \rangle - \left\llangle \rho, \chi_n  \otimes ( \Phi_{1/n} \phi(0, \cdot)) \right\rrangle \\
& -  \left\langle \rho_0, \int_{- \infty}^0 \Theta_n \left( \dr_t \phi + \frac{1}{2} |\nabla \phi|^2 \right)  (t, \cdot) \, \ddr t \right\rangle  - \left\langle \rho_1, \int_{0}^{+ \infty} \Theta_n \left( \dr_t \phi + \frac{1}{2} |\nabla \phi|^2 \right)(t, \cdot) \, \ddr t \right\rangle \\
& - \left\llangle \rho, \Theta_n \left( \dr_t \phi + \frac{1}{2} |\nabla \phi|^2 \right)  \right\rrangle.
\end{align*}
We do some cleaning: in the fourth line, we discard the terms in $|\nabla \phi|^2$ as they have a sign. In the last line, we use the estimate coming from Theorem \ref{theorem_bakry_emery}. Also, in the fourth and the last lines, we use the explicit expression that we have for $\Theta_n(\dr_t \phi)$. Specifically, for the fourth line:  
\begin{multline*}
\left\langle  \rho_0  , \int_{- \infty}^0 \Theta_n \left( \dr_t \phi + \frac{1}{2} |\nabla \phi|^2 \right)  (t, \cdot) \, \ddr t \right\rangle  \\
\leqslant \left\langle \rho_0, \int_{- \infty}^0 \left( \dr_t [\Theta_n(\phi)](t, \cdot) + \chi_n(t-1) \Phi_{1/n} \phi(1, \cdot) - \chi_n(t) \Phi_{1/n} \phi(0, \cdot)  \right) \, \ddr t \right\rangle \\ 
= \langle \rho_0, \Theta_n(\phi)(0, \cdot) \rangle - \frac{1}{2} \langle \rho_0, \Phi_{1/n} \phi(0, \cdot) \rangle.
\end{multline*}
Of course, a very similar computation holds for the term involving $\rho_1$. On the other hand, we can also estimate in the last line
\begin{multline*}
\left\llangle \rho, \Theta_n \left( \dr_t \phi + \frac{1}{2} |\nabla \phi|^2 \right)  \right\rrangle \\
 \geqslant \left\llangle \rho, \left( \dr_t [\Theta_n(\phi)] + (\chi_n(\cdot-1)) \otimes \Phi_{1/n} \phi(1, \cdot) - \chi_n \otimes \Phi_{1/n} \phi(0, \cdot)  \right)  \right\rrangle + e^{-C/n} \left\llangle \rho, \frac{1}{2} \left| \nabla \Theta_n(\phi) \right|^2 \right\rrangle.
\end{multline*}
Putting all these computations together, we find many cancellations and end up with 
\begin{multline*}
\langle \tilde{\rho}_{n,1}  , \phi(1, \cdot) \rangle - \langle \tilde{\rho}_{n,0}, \phi(0, \cdot) \rangle - \left\llangle \tilde{\rho}_n, \dr_t \phi + \frac{1}{2} |\nabla \phi|^2  \right\rrangle \\
\leqslant \langle \rho_1, \Theta_n(\phi)(1, \cdot) \rangle - \langle \rho_0, \Theta_n(\phi)(0, \cdot) \rangle - \llangle \rho,  \dr_t [\Theta_n(\phi)] \rrangle - e^{-C/n} \left\llangle \rho, \frac{1}{2} \left| \nabla \Theta_n(\phi) \right|^2 \right\rrangle.
\end{multline*}
The right hand side can be bounded using Lemma \ref{lemma_dual_action} for $\A_{\rho_0, \rho_1}(\rho)$ with test function $e^{-C/n} \Theta_n(\phi)$. Hence we get 
\begin{equation*}
\langle \tilde{\rho}_{n,1}  , \phi(1, \cdot) \rangle - \langle \tilde{\rho}_{n,0}, \phi(0, \cdot) \rangle - \left\llangle \tilde{\rho}_n, \dr_t \phi + \frac{1}{2} |\nabla \phi|^2  \right\rrangle \leqslant e^{C/n} \A_{\rho_0, \rho_1}(\rho).
\end{equation*}
Taking the supremum over $\phi$ and given Lemma \ref{lemma_dual_action}, we see that 
\begin{equation*}
\A_{\tilde{\rho}_{n,0}, \tilde{\rho}_{n,1}}(\tilde{\rho}_n) \leqslant e^{C/n} \A_{\rho_0, \rho_1}(\rho) \leqslant e^{C/n} \A(\rho, \mbf).
\end{equation*}

As we have seen above, the measure $\tilde{\rho}_n$ admits a smooth density bounded from below, and thanks to the equation above its action can be controlled. They are two things left to do: check that the boundary values of $\tilde{\rho}_n$ are almost $(\rho_0, \rho_1)$, and choose the optimal momentum. 

For the boundary values, we need to check that $\tilde{\rho}_{n,0}$ converges weakly to $\rho_0$ as $n \to + \infty$ (the computation is almost identical for $\tilde{\rho}_{n,1}$). Let $\phi \in C^1(X)$ a test function. Let $\Gamma_n : [0, + \infty) \to \R$ a primitive of $-\chi_n$ whose value is $1/2$ for $t=0$ (hence $\Gamma_n(t) = 0$ for $t \geqslant 1/n$). As a test function in the continuity equation, we use $\Gamma_n \otimes (\Phi_{1/n} \phi)$ to write
\begin{equation*}
\frac{1}{2} \langle \rho_0, \Phi_{1/n} \phi \rangle = \llangle \rho, \chi_n  \otimes ( \Phi_{1/n} \phi ) \rrangle - \llangle \mbf, \Gamma_n \otimes (\nabla (\Phi_{1/n} \phi)) \rrangle.
\end{equation*} 
Notice that $\Gamma_n \leqslant \1_{[0,1/n]}$ while $\nabla (\Phi_{1/n} \phi)$ is bounded uniformly, for instance thanks to Theorem \ref{theorem_bakry_emery}. As $\mbf$ is a finite measure, by dominated convergence
\begin{equation*}
\lim_{n \to + \infty} \left( \llangle \mbf, \Gamma_n \otimes (\nabla (\Phi_{1/n} \phi) \rrangle \right) = 0.
\end{equation*}
With the help of this information, we can conclude that 
\begin{equation*}
\lim_{n \to + \infty} \langle \tilde{\rho}_{n,0}, \phi \rangle = \langle \rho_0, \phi \rangle.
\end{equation*}
As the function $\phi$ is smooth but arbitrary, \review{we conclude that $\tilde{\rho}_{n,0}$ converges weakly to $\rho_0$, hence $W_2(\tilde{\rho}_{n,0}, \rho_0)$ also converges to $0$ as the mass of $\tilde{\rho}_{n,0}$ coincides with the one of $\rho_0$}.

Eventually, we need to choose the momentum: we will choose the optimal one given $\tilde{\rho}_n$. Recall that $\tilde{\rho}_n \in C^\infty([0,1] \times X)$ denotes, by abuse of notations, the density of $\tilde{\rho}_n$ w.r.t. $\ddr t \otimes \ddr x$. For each $t \in [0,1]$ we choose $\psi_n(t, \cdot)$ the unique solution with $0$-mean of the elliptic equation 
\begin{equation}
\label{equation_continuity_equation_smooth}
\nabla \cdot (  \tilde{\rho}_n(t,\cdot) \nabla \psi_n(t, \cdot) ) = -\dr_t \tilde{\rho}_n(t, \cdot)
\end{equation}
with Neumann boundary conditions \review{$\nabla \psi_n(t, \cdot) \cdot \nbf_{\dr X} = 0$ on $\dr X$}. As $\tilde{\rho}_n(t, \cdot)$ is smooth and bounded from below, this equation always admits a solution. Such a solution is smooth by elliptic regularity, up to the boundary as $\dr X$ is smooth. By setting $\tilde{\mbf}_n = \nabla \psi_n \tilde{\rho}_n$, we get that $(\tilde{\rho}_n, \tilde{\mbf}_n)$ satisfies the continuity equation ($\psi_n$ was built for that in \eqref{equation_continuity_equation_smooth}) and 
\begin{equation*}
\A(\tilde{\rho}_n, \tilde{\mbf}_n) = \A_{\tilde{\rho}_{n,0}, \tilde{\rho}_{n,1}}(\tilde{\rho}_n).
\end{equation*}
(For the latter, one can use for instance Lemma \ref{lemma_dual_action} with $\psi_n$ as a test function). As a consequence, $\A(\tilde{\rho}_n, \tilde{\mbf}_n) \leqslant e^{C/n} \A(\rho, \mbf)$, and we set $(\tilde{\rho}, \tilde{\mbf}) = (\tilde{\rho}_n, \tilde{\mbf}_n)$ for $n$ large enough to reach the conclusion of the proposition (large enough so that the temporal boundary conditions almost match and the increase in the action is small enough). 

\end{proof}

\begin{rmk}
\label{rmk_m_not_optimal}
Using the notations of the proof of Proposition \ref{prop_regularization} above, it is easy to check that $\tilde{\rho}_n$ converges to $\rho$ weakly in $\M_+([0,1] \times X)$.
\review{Moreover, we have in fact proved that
\begin{equation*}
\limsup_{n \to + \infty} \A(\tilde{\rho}_n, \tilde{\mbf}_n) \leqslant  \A_{\rho_0, \rho_1}(\rho). 
\end{equation*}
So if $\A_{\rho_0, \rho_1}(\rho) < \A(\rho, \mbf)$, that is if the moment $\mbf$ is not the optimal one in \eqref{definition_optimal_m}, then by lower semi-continuity of $\A$ we cannot have $\lim_n \tilde{\mbf}_n = \mbf$. This is not that surprising because $\tilde{\mbf}_n$ was defined thanks the solutions of the elliptic equation \eqref{equation_continuity_equation_smooth} which depend only on $\tilde{\rho}_n$ (hence $\rho$), but not on $\mbf$. }
\end{rmk}

\bibliographystyle{apa}
\bibliography{bibliography}

\end{document}